\documentclass[journal,twoside,web]{ieeecolor}
\usepackage{generic}

\usepackage{cite}
\usepackage{amsmath,amssymb,amsfonts, mathtools}
\usepackage{mathrsfs}   %
\usepackage{accents}
\usepackage{algorithmic}
\usepackage{graphicx}
\usepackage{algorithm,algorithmic}
\usepackage{hyperref}
\hypersetup{hidelinks=true}
\usepackage{textcomp}

\usepackage{verbatim}
\usepackage{etoolbox} %
\AtEndEnvironment{remark}{\hfill$\blacklozenge$}
\AtEndEnvironment{example}{\hfill$\blacktriangle$}

\usepackage{url}
\let\labelindent\relax
\usepackage{enumitem}
\usepackage{xcolor}
\usepackage{subcaption}
\usepackage[ %
    draft,
    commandnameprefix=ifneeded, 
    xcolor=dvipdf,
    todonotes={colorinlistoftodos,
                prependcaption,
                textsize=footnotesize,
                backgroundcolor=orange!10,
                textcolor=black,
                linecolor=orange,
                bordercolor=orange}
]{changes}

\usepackage{marginnote}

\usepackage{tcolorbox}
 
\newif\ifinappendix
\inappendixfalse

\newtheorem{theorem}{Theorem}[section]
\newtheorem{assumption}[theorem]{Assumption}
\newtheorem{definition}[theorem]{Definition}

\newtheorem{lemma}[theorem]{Lemma}

\newtheorem{remark}[theorem]{Remark}
\newtheorem{example}[theorem]{Example}

\newtheorem{proposition}[theorem]{Proposition}

\definecolor{mygreen}{RGB}{0,0,0}
\definecolor{mymagenta}{RGB}{255,0,255}

\usepackage{import}

\newcommand{\crit}[1]{\operatorname{crit}#1}

\newcommand{\diffFunc}[1]{\mathrm{d}{#1}}

\newcommand{\dom}{\mathrm{dom \:}}

\newcommand{\R}[1]{\mathbb{R}^{#1}}

\newcommand{\brackets}[1]{\left(#1\right)}

\newcommand{\T}[2]{\mathrm{T}_{#1}{#2}}

\newcommand{\neighborhood}{\operatorname{Nbd}}
\newcommand{\mobius}{\mathscr{B}}

\newcommand{\rectangle}{\mathscr{R}}
\newcommand{\tangentEq}{\T{}{\mobius}}

\newcommand{\M}{{\mathcal{M}}}
\newcommand{\N}{{\mathcal{N}}}
\newcommand{\X}{\mathcal{X}}

\newcommand{\A}{\mathcal{A}}
\newcommand{\U}{\mathcal{U}}
\newcommand{\hybrid}{\mathcal{H}}
\newcommand{\PD}[1]{\mathcal{PD}(#1)}
\renewcommand{\T}[2]{\mathrm{T}_{#1}{#2}}
\newcommand{\Tcone}[2]{\mathrm{T}^{\M}_{#1}({#2})}
\newcommand{\TconeEuclidean}[2]{\mathrm{T}_{#1}({#2})}

\renewcommand{\dom}[1]{\operatorname{dom}{#1}}
\newcommand{\interior}[1]{\operatorname{int}#1}
\newcommand{\distfromA}[1]{|#1|_\A}

\newcommand{\vlift}[2]{\operatorname{vlft}_{#1}(#2)}
\newcommand{\grad}[2]{\operatorname{grad} {#1}({#2})}
\newcommand{\sublevelSet}[1]{L_{#1}}
\newcommand{\rge}{\operatorname{rge}}
\newcommand{\graph}[1]{\operatorname{gph} {#1}}
\newcommand{\graphlim}[2]{\mathrm{gph \mbox{-}lim}_{#1}{#2}}
\newcommand{\maximalSol}[2]{\mathcal{S}_{{#1}}(#2)}
\newcommand{\setofSol}[2]{\widehat{\mathcal{S}}_{{#1}}(#2)}
\newcommand{\constraintSet}{S} %

\renewcommand{\cal}[1]{\mathcal{#1}}
\renewcommand{\frak}[1]{\mathfrak{#1}}

\newcommand{\angvel}{z_2}

\makeatletter
\newcommand{\myitem}[1]{%
\item[#1]\protected@edef\@currentlabel{#1}%
}
\makeatother

\def\BibTeX{{\rm B\kern-.05em{\sc i\kern-.025em b}\kern-.08em
    T\kern-.1667em\lower.7ex\hbox{E}\kern-.125emX}}
\begin{document}
\title{Geometric Hybrid Dynamical Systems:\\ Part I -- Modeling and Stability}
\author{
Piyush P.\ Jirwankar,
Daniel E.\ Ochoa, and 
Ricardo G.\ Sanfelice\vspace{-15pt}
\thanks{\scriptsize{Under review at IEEE Transactions on Automatic Control.}}
\thanks{\scriptsize The authors are with the Electrical and Computer Engineering Department, University of California Santa Cruz, CA 95060 USA. Emails: \{pjirwank, dochoatamayo, ricardo\}@ucsc.edu.}}

\maketitle

\begin{abstract}
    We present a framework for the modeling and analysis of geometric hybrid dynamical systems as hybrid inclusions on $C^1$-manifolds. Using tools from nonsmooth and set-valued analysis on manifolds, we derive coordinate-independent sufficient conditions for the existence of nontrivial solutions to this type of systems. We present geometric notions of uniform stability and attractivity of compact sets, and establish their equivalence to metric-based stability notions when the manifold is endowed with a Riemannian structure. We also introduce nonsmooth Lyapunov functions and a hybrid Lyapunov theorem providing sufficient conditions for uniform global asymptotic stability of compact sets. Finally, by characterizing $\omega$-limit sets of precompact solutions, we derive a hybrid invariance principle for geometric hybrid dynamical systems. The results are demonstrated through several running examples.
\end{abstract}


\section{Introduction}
\label{sec:introduction}

We study a class of hybrid dynamical systems that evolve on intrinsic manifolds; namely, \emph{geometric hybrid dynamical systems}. Several modeling frameworks exist for hybrid systems that evolve on the Euclidean space~\cite{goebel_hybrid_2012,lygeros2003dynamical,haddad2014impulsive}. Among these, the hybrid inclusions framework~\cite{goebel_hybrid_2012} provides 
a mathematically well-developed setting that accommodates differential and difference inclusions with state constraints. However, these frameworks do not readily transfer to numerous mechanical and aerospace applications where states evolve on nonlinear manifolds, thus motivating the search for a geometric framework for hybrid dynamical systems on manifolds.  

\subsection{On Embeddings and Metrizations}

A central issue in geometric models of hybrid systems is how the state-space manifold is treated. Although every smooth manifold can be embedded in a higher-dimensional Euclidean space by the Whitney embedding theorem~\cite[Thm.~6.15]{Lee}, its intrinsic properties do not depend on any particular embedding. Many frameworks~\cite{CasauCompact,leeLeokMcClamroch2017global} assume an embedding in $\R{n}$ to use Euclidean analytical and computational tools. However, such an embedding adds structure that may be unavailable, noncanonical, or inconvenient.
This issue is particularly relevant for robotic systems with quotient configuration spaces~\cite{orthey2018quotient} or symplectic structure~\cite{hofer1994symplectic}, where symmetries can make Euclidean embeddings impractical or distort the geometry relevant to the dynamics. For example, Nesterov-type acceleration on Riemannian manifolds depends critically on manifold curvature~\cite{alimisis2020continuous}, which may be obscured by a non-isometric embedding.

An intrinsic approach may also require additional metric structure for stability analysis. A common metric choice is a Riemannian metric, which provides a finite geodesic distance on connected manifolds~\cite[Prop.~16.14]{gallierDifferentialGeometry2020}. Hybrid state spaces, however, are often disconnected, with separate components representing different discrete modes~\cite{mayhew2011quaternion}. Since geodesics do not connect these components, other distance functions must be introduced~\cite[Cor.~13.30]{Lee}. Such distances may be noncanonical, causing the resulting Lyapunov analysis to depend on choices beyond the system's intrinsic geometry.

\subsection{Existing Models of Hybrid Systems on Manifolds}
The foregoing discussion underscores the need for a purely topological and geometric framework for modeling hybrid dynamical systems on \emph{intrinsic} manifolds. The existing literature in this direction remains relatively limited. Early work in~\cite{Simic2000_hybrifold} models hybrid automata on submanifolds embedded in Euclidean space, where discrete ``jumps'' are represented through \emph{hybrifolds} that transform hybrid dynamics into nonsmooth dynamics via certain quotient constructions. More recently, a fiber bundle–based formulation for hybrid automata with control inputs has been proposed in~\cite{BARBEROLINAN2020100935}. A more general modeling approach, namely ``topological hybrid systems'', is introduced in~\cite{Kvalheim_Conley_2020}.

While these contributions provide novel and insightful modeling paradigms, their dynamics are typically expressed through differential equations or semiflow relations, and are not naturally suited to capturing nonuniqueness of solutions or accommodating perturbations~\cite{freeman2008robust}. The hybrid inclusions framework addresses these issues systematically in the Euclidean setting, where nonuniqueness is incorporated through set-valued inclusion models~\cite{goebel_hybrid_2012}. An extension of this framework for hybrid systems on connected Riemannian manifolds is proposed in~\cite{ForniAngeli_hybridSmoothManifolds} to analyze multistable inclusions. However, this framework does not model hybrid automata, which is a large class of systems that exhibit disconnected underlying manifolds.
Thus, to the best knowledge of the authors, no readily available framework combines the set-valued generality of hybrid inclusions with an intrinsic treatment of manifolds that admits disconnected state spaces.

\vspace{-10pt}
\subsection{Contributions}

Motivated by the limitations of embedding-based and Riemannian methods, this paper develops a coordinate-independent framework for hybrid systems on $C^1$-manifolds. 
We call these systems \emph{geometric hybrid dynamical systems}. The main contributions are as follows.

\begin{enumerate}[label=\arabic*),leftmargin=*]
    \item We provide sufficient conditions for the existence of solutions using regularity properties of set-valued maps between topological spaces (Proposition~\ref{prop:existence}).
    \item We introduce topological notions of uniform global stability and uniform attractivity for compact sets (Definitions~\ref{def:UGS} and~\ref{def:uniformGlobalAtt}). These notions use neighborhoods~rather than a chosen distance. For a class of Riemannian manifolds, we show that the proposed stability notion is~equivalent to a metric-based characterization (Proposition~\ref{prop:UGS}).
    
    \item We develop nonsmooth analysis tools on $C^1$-manifolds, including tangent cones, locally Lipschitz functions, and present a Rademacher's Theorem in manifolds as well as  generalized directional derivatives (Definitions~\ref{def:tangentConeManifolds},~\ref{def:Lipschitz},-\ref{def:generalizedDerivative}). We use these tools with a nonsmooth Lyapunov function to prove a hybrid Lyapunov theorem for uniform global asymptotic stability of compact sets (Theorem~\ref{theorem:hybridLyapunovTheorem-manifolds}).
    
    \item We provide conditions ensuring a closure property of the solution set (Theorem~\ref{theorem:basicConditions => nominalWellPosedness}), characterize omega-limit sets (Proposition~\ref{prop:omegaLimitSets}), and derive a hybrid invariance principle on $C^1$-manifolds (Theorem~\ref{theorem:invariance}).

    \item We illustrate the framework through two running examples. The first is hybrid attitude stabilization on $\mathbb{S}^3$ (Examples~\ref{ex:quaternion},~\ref{ex:quaternion-solutions}, and~\ref{ex:quaternion-stability}). The second is stabilization of billiard ball on the M\"obius band (Examples~\ref{ex:mobius-hybridDynamics},~\ref{ex:mobius-solutions}, and~\ref{ex:mobius-stability}), which is a second-order system with impacts.
\end{enumerate}

We highlight that our presentation does not require choosing a metric or an embedding. However, it allows using one when convenient, as demonstrated through the examples. To the best of the authors' knowledge, this is the first work to collect and extend the above results to hybrid dynamical systems evolving on abstract manifolds in a fully metric-free, embedding-free, and coordinate-independent setting.

\emph{Relation with previous work:} Preliminary results from this work were presented in the conference papers~\cite{jirwankar2025lyapunov, jirwankar2025invariance}, which omitted intermediate steps, detailed proofs, and additional results. We improve over those versions in the following ways: (i) we introduce Theorem~\ref{theorem:basicConditions => nominalWellPosedness}, which generalizes~\cite[Thm.~2]{jirwankar2025invariance} by allowing a broader class of solution sequences; (ii) we generalize~\cite[Prop.~3]{jirwankar2025lyapunov} to the Riemannian manifold setting in Propositions~\ref{prop:UGS}, showing that metric-based and topological notions of uniform global stability coincide. We also obtain several auxiliary results, including (iii) the coordinate independence of the geometric definition of tangent cones (Lemma~\ref{lemma:coordInvariant-tangentCone}); and (iv) a revised definition of locally Lipschitz functions between $C^1$-manifolds from~\cite[Def.~5]{jirwankar2025lyapunov} (Definition~\ref{def:Lipschitz}), and a manifold version of Rademacher's theorem (Proposition~\ref{prop:rademacher}).

\emph{Organization:} Section~\ref{sec:preliminaries} introduces notation and the necessary preliminaries from differential geometry and set-valued analysis on manifolds. Section~\ref{sec:hybridSystems} introduces geometric hybrid systems and related notions. Existence of solution and properties of solution sets are studied in Section~\ref{sec:solutionSetProperties}. Section~\ref{sec:lyapunovTheroemSection} presents topological notions of (uniform) asymptotic stability and a hybrid Lyapunov theorem. Section~\ref{sec:invariance} presents a hybrid invariance principle. Some proofs and auxiliary results are presented in the Appendix. 

\section{Preliminaries}
\label{sec:preliminaries}

\subsection{Notation}

The set of real, positive real, and nonnegative real numbers is denoted by $\R{}$, $\R{}_{>0}$, and $\R{}_{\geq 0}$, respectively. The set of natural numbers, including $0$, is denoted by $\mathbb{N}$.{\color{mygreen}\relax{} The closed unit ball in $\R{n}$ is denoted by $\mathbb{B}\coloneqq \{\eta\in \R{n}: |\eta|\leq 1\}$, where $|\eta|$ denotes the Euclidean norm of a vector $\eta\in \R{n}$. The distance from a point $x\in \R{n}$ to a nonempty set $\A\subset \R{n}$ is defined as $\distfromA{x}\coloneqq \inf_{y\in \A}|x - y|$.} Given a topological space $(\X, \tau)$,  the closure of a set $S\subset \X$ is denoted by $\overline{S}$ and its interior by $\interior{S}$. The set $S$ is compact if every open cover of $S$ has a finite subcover, and it is precompact if $\overline{S}$ is compact. A set $S \subset \X$ is said to be a compact neighborhood of $\A \subset \X$ if $S$ is compact and $\A \subset \interior{S}$. {\color{mygreen}\relax{}Given a set $S$, we use $\cal P(S)$ to denote the set of all subsets of $S$. }

A function $\alpha: \R{}_{\geq 0}\to\R{}_{\geq 0}$ is a class-$\mathcal{K}$ function, denoted by $\alpha\in\mathcal{K}$, if $\alpha$ is zero at zero, continuous, and strictly increasing.  
The function $\alpha$ is a class-$\mathcal{K}_{\infty}$ function, denoted by $\alpha\in\mathcal{K}_\infty$, if $\alpha \in \mathcal{K}$ and $\lim_{r\to \infty}\alpha(r) =\infty$.
A function $\sigma : X \to \R{}_{\geq 0}$ is positive definite with respect to $\A \subset X$, denoted by $\sigma\in \PD{\A}$, if $\sigma(x) = 0$ if and only if $x\in \A$. 
A map $f:\Omega\to\R{m}$, with $\Omega\subset\R{n}$ open, is said to be $C^k$, $k\in\mathbb{N}\cup\{\infty\}$, if all its partial derivatives of order $j\leq k$ exist and are continuous on $\Omega$. A $C^k$-diffeomorphism on $\R{n}$ is a bijection $f$ such that both $f$ and $f^{-1}$ are $C^k$.
The composition of maps $f:X\to Y$ and $g:Y\to Z$ is denoted by $g\circ f : X \to Z$. A single-valued function $f:X \to\R{}$ is lower semicontinuous if, at each $x\in \M$, $f(x)\leq \liminf_{y\to x} f(y)$, and it is proper if the preimage $\{x\in X: f(x)\leq c\}$ is compact for each $c\in \R{}$. A set-valued map $F:X\rightrightarrows Y$ maps each point $x\in X$ to a subset $F(x)\subset Y$. The domain, range, and graph of $F$ are defined as $\dom{F} \coloneqq \{x\in X : F(x)\neq \varnothing\}$, $\rge F \coloneqq \{y \in Y : \exists x \in \dom{F} \text{ s.t. } y\in F(x)\}$, and $\graph{F} \coloneqq \{(x,y) \in X\times Y : y \in F(x)\}$, respectively. Given a nonempty set $S \subset \R{n}$, the tangent cone $\TconeEuclidean{S}{x}$ to $S$ at $x\in S$ is the set of all vectors $w\in \R{n}$ for which there exist sequences $x_i\in S$, $\tau_i > 0$ with $x_i \to x$, $\tau_i \searrow 0$, and $w = \lim_{i\to \infty} \frac{x_i - x}{\tau_i}$.

\vspace{-6pt}
\subsection{Differential geometry}
\label{sec:differentialGeometry}
An $n$-dimensional topological manifold $\M$ is a second-countable Hausdorff space that is locally Euclidean of dimension $n$. A coordinate chart of $\M$ at $x$ is a pair $(U, \varphi)$ where $x\in U\subset \M$, with $U$ an open set and $\varphi:U\to \varphi(U)\subset\R{n}$ a homeomorphism. A collection of coordinate charts $\{(U_{\alpha}, \varphi_{\alpha}) \}_{\alpha\in \mathcal{I}}$, where $\mathcal{I}$ denotes some index set, such that $\bigcup_{\alpha}U_\alpha = \M$ is called an \emph{atlas}.
For a topological manifold $\M$ with topology $\tau$, we define the set of neighborhoods of $x\in \M$ as $\operatorname{Nbd}(x) \coloneqq \{V\subset\M : \exists U\in \tau \textrm{ such that } x\in U \subset V\}$. Two charts $(U, \varphi)$ and $(W, \psi)$ of $\M$ are $C^k$-compatible if $U\cap W = \varnothing$ or $\psi\circ\varphi^{-1}: \varphi(U\cap W)\to \psi(U\cap W)$ is a $C^k$-diffeomorphism. A $C^k$-atlas  of $\M$ is one whose charts are $C^k$-compatible. A $C^k$-manifold  $\M$ is a topological manifold endowed with a maximal $C^k$-atlas. The tangent space to~$\M$ at $x\in\M$ is denoted by $\T{x}{\M}$ and the tangent bundle of $\M$ is denoted by {$\T{}{\M}\coloneqq \{(x,v) :x\in \M, v\in \T{x}{\M}\}$.}  
A map $f:\M\to\N$ between $C^k$-manifolds is $C^k$ if,~at each $x\in\M$, there exists a coordinate chart $(U,\varphi)$ at $x\in \M$ and a chart $(W, \psi)$ at $f(x)\in \N$ such that the composition $\psi\circ f \circ \varphi^{-1}$ is $C^k$ in the usual Euclidean sense. 
The differential of $f$ at $x\in\M$ is denoted by~$\diffFunc{f_x}:\T{x}{\M}\to\T{f(x)}{\N}$ and defined as $\diffFunc{f_x}(v)\coloneqq  \left.\frac{df}{dt}(\gamma(t)) \right\vert_{t=0}$ for each $v\in \T{x}{\M}$, where $\mathcal{I}\ni t\mapsto \gamma(t)$, with $\mathcal{I}\subset \R{}$ such that $0\in\mathcal{I}$, is a smooth curve on $\M$ satisfying $\gamma(0)=x$ and $\gamma'(0)=v\in\T{x}{\M}$.
For more details on differentiable manifolds, we refer the reader to \cite{Lee}.

A $C^k$-Riemannian manifold~\cite{Lee_Riemannian} is a pair $(\M,g)$, where $\M$ is a $C^k$-manifold and $g$ is a Riemannian metric on $\M$ whose value at each $x\in\M$ defines an inner-product on~$\T{x}{\M}$. The gradient of a $C^k$-function $f\!:\!\M\!\to\!\R{}$ at $x$, denoted~by $\grad{f}{x}$, satisfies $\diffFunc{f_x}(v) \!\!= \!\!g(\grad{f}{x}, v)$ for all $v\in \T{x}{\M}$.

Next, we present an example of a smooth manifold, which will be used throughout this paper to illustrate our results. 

    \begin{example}[The M\"{o}bius band]
    \label{ex:MobiusDef}
        Fix $\varepsilon > 0$ and consider the rectangle $\rectangle \coloneqq [0,1]\times (-1-\varepsilon, 1 + \varepsilon) \subset\R{2}$. The open M\"{o}bius band is given by $\mobius \coloneqq \rectangle / \sim$, where the equivalence relation is given by $(0, z_2) \sim (1, -z_2)$ for each $z_2\in (-1-\varepsilon, 1 + \varepsilon)$; see Figure~\ref{fig:mobius-quotient} for an illustration. Due to this relation, $\mobius$ is a quotient space~\cite[Ex.~10.3]{Lee}. We denote an element in $\mobius$ corresponding to $z \coloneqq (z_1, z_2)\in \rectangle$ by $[z]_{\mobius}$, where $[\cdot]_{\mobius}$ denotes the equivalence class. Equipped with smooth coordinate charts according to~\cite[Ex.~10.3]{Lee}, $\mobius$ is a smooth manifold of dimension two. We also endow $\mobius$ with a Riemannian metric $g_{\mobius}$ according to~\cite[Ex.~2.35]{Lee_Riemannian}, and the associated Levi-Civita connection $\nabla$; see~\cite[Ch.~15.3]{gallierDifferentialGeometry2020}.

        The tangent bundle $\T{}{\mobius}$ is then defined by an induced equivalence relation as $\T{}{\mobius} = \T{}{\rectangle} /\sim$ such that, for each $z_2\in (-1-\varepsilon, 1 + \varepsilon)$ and each $u\coloneqq (u_1, u_2)\in \T{(0,z_2)}{\rectangle} = \R{2}$, one has $\T{(0, z_2)}{\rectangle} \ni (u_1, u_2) \sim (u_1, -u_2) \in \T{(1, -z_2)}{\rectangle}$. We endow $\T{}{\mobius}$ with the natural topology and the smooth structure~\cite[Prop.~3.18]{Lee}, making $\T{}{\mobius}$ a smooth manifold of dimension four. An element in $\T{}{\mobius}$ corresponding to $(z,u)\in \T{}{\rectangle}$, i.e., $z\in \rectangle$ and $u\in \T{z}{\rectangle}$, is denoted by $[z,u]_{\T{}{\mobius}}$. Using the canonical projection map $\pi : \T{}{\mobius} \to \mobius$, we equivalently write $[z,u]_{\T{}{\mobius}}$ as $([z]_{\mobius}, [u]_{\tangentEq})$, where $ [z]_{\mobius} = \pi([z,u]_{\T{}{\mobius}})$, and, with some abuse of notation, $[u]_{\tangentEq} \in \T{[z]_{\mobius}}{\mobius}$ denotes the corresponding tangent vector.         
    \end{example}
\begin{figure}[t]
        \centering
        \includegraphics[width=0.99\linewidth]{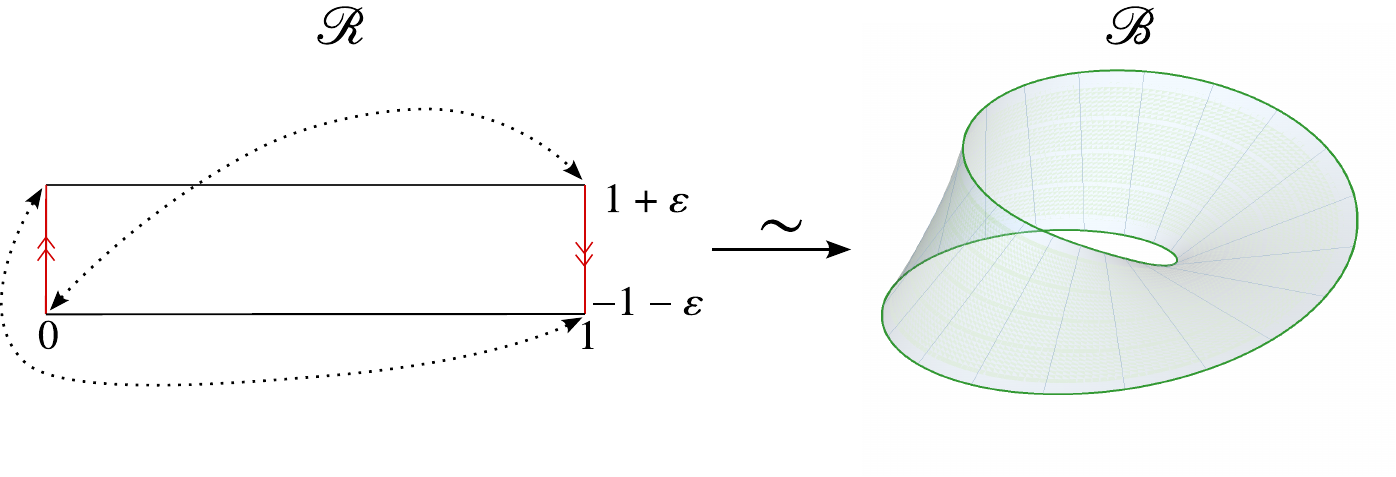}

            \caption{
            The quotient operation defining the open M\"{o}bius band $\mobius$ as described in Example~\ref{ex:MobiusDef}.
            } 
        \label{fig:mobius-quotient}
        \vspace{-0.6cm}
    \end{figure}

\vspace{-3pt}
\subsection{Set-valued analysis}
\label{sec:set-valued}
Set-valued analysis has been foundational for analyzing the behavior of solutions in the hybrid inclusions framework in Euclidean spaces \cite[Ch. 5,6]{goebel_hybrid_2012}. To study set-valued analysis tools in manifolds, we start with a notion of set convergence that reduces to \cite[Def. 5.1]{goebel_hybrid_2012} when $\M=\R{m}$. %

\begin{definition}[{Set convergence~\cite[p.209]{Willard_GeneralTopology}}]%
    \label{def:setConvergence}
   Let $\M$ be a topological manifold and $\{S_i\}_{i=1}^{\infty}$ a sequence of sets in $\M$.
    \begin{itemize}[leftmargin=12pt]
        \item The \emph{inner limit} of the sequence~$\{S_i\}_{i=1}^\infty$ is defined as 
        \(
        \liminf_{n\to \infty} S_i \coloneqq \{x\in \M : \forall V\in \operatorname{Nbd}(x), \exists i_0 > 0 \text{ s.t. } V\cap S_i \neq\varnothing \;\forall i > i_0 \}
        \)
        \item The \emph{outer limit} of the sequence $\{S_i\}_{i=1}^\infty$ is defined as 
        \(
        \limsup_{n\to \infty} S_i \coloneqq \{x\in \M : 
             \forall V\in \neighborhood(x), \forall i_0 > 0,
              \exists i > i_0 \textrm{ s.t. } V\cap S_i \neq \varnothing\}.
        \)
    \end{itemize}
     The sequence $\{S_i\}_{i=1}^\infty$ is said to converge if its inner limit and outer limit are equal, and the limit is given by
        $
            \lim_{i\to \infty} S_i = \liminf_{i\to \infty} S_i = \limsup_{i\to \infty} S_i. 
        $
\end{definition}

By extrapolating the notion from the Euclidean definition \cite[Cor. 4.11]{Rockafellar1997}, we say that the sequence $\{S_i\}_{i=1}^\infty$ of nonempty subsets of $\M$ \emph{escapes to the horizon} if for each compact set $K\subset \M$, there exists $i_0 > 0$ such that $S_i \cap K \!=\! \varnothing$ for all $i\geq i_0$. If $\M$ is compact, one can set $K = \M$, and therefore, every sequence of sets never escapes to the horizon. 

The following lemma shows that, similar to the Euclidean case \cite[Theorem 4.81]{Rockafellar1997}, a sequence of sets in topological manifolds either has a convergent subsequence or escapes to the horizon. The proof follows from \cite[Theorem 5.2.12]{Beer1993}.

\begin{lemma}
    \label{lemma:sequentialCompactness_beer}
    Given a topological manifold $\M$, every sequence $\{S_i\}_{i=1}^\infty \subset \M$ of nonempty sets either escapes to the horizon or has a subsequence converging to a nonempty set. %
\end{lemma}

A set-valued map $F : \M \rightrightarrows \N$ between topological manifolds is \emph{outer semicontinuous} if, for each $x\in \M$, 
\(
    \limsup_{y \to x} F(y) \coloneqq \bigcup_{\{x_i\}\to x} \limsup_{i\to \infty} F(x_i)\subset F(x).
\)
{Given a nonempty set $S\subset \M$, $F$ is \emph{outer semicontinuous relative to $S$} if the map $\widetilde{F} : \M \rightrightarrows \N$, defined for each $x\in\M$ as $\widetilde{F}(x) \coloneqq F(x)$ if $x\in S$ and empty otherwise, is outer semicontinuous.}
{\color{mygreen}\relax{}{\color{mygreen}The map $F$ is \emph{upper semicontinuous} if, for each $x\in \M$ and each open set $W\subset \N$ with $F(x)\subset W$, there exists an open set $U \subset \M$ with $x\in U$ such that $F(U) \subset W$.}} The map $F$ is locally precompact if, for each $x\in \M$, there exists a neighborhood $\U$ of $x$ such that $F(\U)$ is precompact.\footnote{This notion is called ``local boundedness'' in the Euclidean setting; see~\cite[Def.~5.14]{Rockafellar1997}. The change in nomenclature stems from conflating definitions of boundedness and precompactness. The containment in a precompact set in the Euclidean case is enforced by containment in a ``bounded'' metric ball of finite radius. Such metric-based notion cannot be defined on general topological manifolds. Even on Riemannian manifolds where a metric is available, boundedness and precompactness of a set may not be equivalent; see Heine-Borel theorem. This inconsistency is averted by renaming the notion.
} Given a nonempty set $S\subset \M$, $F$ is \emph{locally precompact relative to~$S$} if $\widetilde{F}$ is locally precompact. 

Using Definition \ref{def:setConvergence}, we introduce the notion of graphical convergence; see \cite[Def. 5.18]{goebel_hybrid_2012} for the Euclidean case.

\begin{definition}[Graphical convergence]
\label{def:graphicalConvergence}
Given topological manifolds $\M$ and $\N$, we say that a sequence $\{M_i\}_{i=1}^{\infty}$ of set-valued mappings $M_i: \M\rightrightarrows \N$ converges \textit{graphically} if the sequence of sets $\{\graph{M_i}\}_{i=1}^{\infty}$ converges in $\M\times \N$ in the sense of Definition~\ref{def:setConvergence}. 
The graphical limit $\graphlim{i\to\infty}{M_i}$ of a graphically convergent sequence $\{M_i\}_{i=1}^{\infty}$ is the map $M: \M \rightrightarrows \N$ such that $\graph{M} = \lim_{i \to \infty} \graph{M_i}$.
\end{definition}

{
The following useful result extends~\cite[Ex.~5.19]{goebel_hybrid_2012} to manifolds. The proof follows similarly, and is presented} in~\cite{jirwankarTAC2026}.

{}

{\color{mygreen}\relax{}

{\color{mygreen}
    \begin{lemma}
        \label{lemma:domainAndRange}
        Let $\M$ and $\N$ be topological manifolds. Consider a sequence $\{M_i\}_{i=1}^\infty$ of set-valued map $M_i : \M \rightrightarrows \N$ that is graphically convergent, and let $M = \graphlim{i\to\infty}{M_i}$. Then, 
        \begin{align}
            \label{eq:domainRangeInclusion}
            \dom{M} \subset \lim_{i\to\infty}\dom{M_i} && \textrm{and} && \rge{M} \subset \lim_{i\to\infty}\rge{M_i}. 
        \end{align}
        If, in addition, the sequence $\{M_i\}_{i=1}^\infty$ is locally eventually precompact,\footnote{The sequence $\{M_i\}_{i=1}^\infty$ is locally eventually precompact if for each compact set $K\subset \M$, there exist $i_0 > 0$ and a compact set $K' \subset \N$ such that $M_i(K)\subset K'$ for all $i > i_0$.} then 
        \begin{align}
            \label{eq:domainRangeEquality}
            \dom{M} = \lim_{i\to\infty}\dom{M_i}.  %
        \end{align}
    \end{lemma}
}

{\newcommand{\PN}[1]{\Pi_{\N}\brackets{#1}}
\newcommand{\PM}[1]{\Pi_{\M}\brackets{#1}}
\begin{proof}
    Firstly, we show that $\lim_{i\to\infty}\dom{M_i}$ and $\lim_{i\to\infty}\rge{M_i}$ exist. Let $\tau_\M$ denote the topology on $\M$ and $\tau_\N$ the topology on $\N$. Let $z\coloneqq (x,y) \in \lim_{i\to\infty}\graph{M_i}$, and consider $U\in \tau_\M$ and $V\in\tau_{\N}$ such that $z\in U\times V$. The existence of such $U$ and $V$ is guaranteed by the definition of the product topology on $\M\times\N$ \cite[p. 86]{Munkres2000}. Then, since $z\in\limsup_{i\to\infty}\graph{M_i}$, it follows from Definition~\ref{def:setConvergence} that for each $i_0 > 0$, there exists $i > i_0$ satisfying $(U\times V) \cap M_i \neq \varnothing$. Then, since $M_i \subset \dom{M_i}\times \rge{M_i}$,
    \begin{align}\label{eq:contradiction}
        & (U\cap \dom{M_i}) \times (V\cap \rge{M_i}) \neq \varnothing \nonumber\\
        & \implies \quad \left\{ \begin{array}{lcl}
            U\cap \dom{M_i} & \neq & \varnothing, \\ 
            V\cap \rge{M_i} & \neq & \varnothing.
        \end{array} \right.
    \end{align}
    
    Then, to prove that $\lim_{i\to\infty}\dom{M_i}$ exists, assume the opposite, i.e., assume that $\lim_{i\to\infty}\dom{M_i}$ does not exist. Therefore, without loss of generality, let $x\in \limsup_{i\to\infty}\dom{M_i}$  satisfy $x\notin \liminf_{i\to\infty}\dom{M_i}$. Consequently, from Definition~\ref{def:setConvergence}, there exists an open neighborhood $U'$ of $x$ such that, for each $i_0 > 0$, there exists $i > i_0$ satisfying 
    \begin{align}\label{eq:contradiction2}
        U' \cap \dom{M_i} = \varnothing.    
    \end{align}
    Since \eqref{eq:contradiction} holds for each set $U\in\tau_{\M}$ and each $V\in\tau_{\N}$ satisfying $z\in U\times V$, we can set $U = U'$ in \eqref{eq:contradiction}, resulting in a contradiction with \eqref{eq:contradiction2}, which proves that $\lim_{i\to\infty}\dom{M_i}$ exists. The existence of $\lim_{i\to\infty}\rge{M_i}$ is shown similarly.

    Next, we prove the second inclusion in \eqref{eq:domainRangeInclusion}. The proof of the first inclusion will follow accordingly. Let $\Pi_{\M}$ and $\Pi_{\N}$ denote the projection of any $S\subset \M\times\N$ onto $\M$ and $\N$, respectively. In particular, define 
    \begin{align*}
        \PM{S}\coloneqq \left\{ x \in \M : \exists y\in \N \textrm{ such that } (x,y)\in S \right\}. 
    \end{align*}
    The projection map $\Pi_{\N}$ is defined accordingly. Therefore, $\dom{M} = \PM{\graph{M}}$ and $\rge{M} = \PN{\graph{M}}$. Similarly, for each $i=1,2,\ldots$, $\dom{M_i} = \PM{\graph{M_i}}$ and $\rge{M_i} = \PN{\graph{M_i}}$. Using Definition~\ref{def:setConvergence}, we have
    \begin{align*}
        \rge{M} = \PN{\graph{M}} &= \PN{\lim_{i\to\infty} \graph{M_i}} \\ 
        &= \PN{\liminf_{i\to\infty} \graph{M_i}}.
    \end{align*}
    Pick any point $z\coloneqq (x,y)\in \liminf_{i\to\infty}\graph{M_i} \subset \M\times\N$. Then, $\PN{z} = y$. Since $z$ belongs to the limit inferior of the sequence of graphs of $M_i$, it appears in all but finitely many $\graph{M_i}$. Therefore, $y$ appears in all but finitely many $\PN{\graph{M_i}}$. Consequently, $y \in \liminf_{i\to\infty} \PN{\graph{M_i}}$. Since $y\in \rge{M}$ and $\PN{\graph{M_i}} = \rge{M_i}$, we have $\rge{M}\subset \liminf_{i\to\infty}\rge{M_i}$. Then, since $\lim_{i\to\infty}\rge{M_i}$ exists, it follows that $\rge{M}\subset \lim_{i\to\infty}\rge{M_i}$. 

    To prove \eqref{eq:domainRangeEquality}, suppose that the sequence $\{M_i\}_{i=1}^\infty$ is locally eventually precompact. Consider a point $x\in \limsup_{i\to\infty}\dom{M_i}$ and a subsequence $\{x_{i_k}\}$ of points $x_{i_k}\in \dom{M_{i_k}}$ that belongs to a compact set $K\subset \M$ and converges to~$x$. Using local eventual precompactness, the sequence of sets $\{M_{i_k}(x_{i_k})\}$ lies in a compact set $K'\subset\N$, and therefore, has a subsequence that converges to a nonempty set. The limit of this subsequence is a subset of $M(x)$ because $M = \graphlim{i\to\infty}{M_i}$, causing $x\in \dom{M}$. As a result, $\limsup_{i\to\infty}\dom{M_i} \subset \dom{M}$. This, along with \eqref{eq:domainRangeInclusion}, proves~\eqref{eq:domainRangeEquality}. 
\end{proof}
}
}

\section{Geometric Hybrid Dynamical Systems}
\label{sec:hybridSystems}

We define geometric hybrid dynamical systems as
\begin{equation}\label{eq:HS}
    \mathcal{H}:\: \left\{\begin{aligned}
        \dot{x}\phantom{^+} &\in F(x) & x\in C \\
        x^+ &\in G(x) & x\in D
    \end{aligned}\right. 
\end{equation}
where $x\in \M$ is the state,  $C\subset\M$ and $D\subset\M$ are the flow set and the jump set, respectively, $G:\M\rightrightarrows\M$ is the jump map, and $\M$ is a finite-dimensional $C^1$-manifold. The flow map  $F: \M\rightrightarrows \T{}{\M}$ is a set-valued map satisfying $F(x)\subset \T{x}{\M}$ for each $x\in \M$. We use $\mathcal{H} = (C,F,D,G, \M)$ to refer to a hybrid system defined with this data. 

Solutions to $\hybrid$ are defined on hybrid time domains \cite[Def.~2.26]{HybridFeedbackControl}, parametrized by ordinary time $t\in\R{}_{\geq 0}$, that denotes the amount of time for which the solution has flowed, and a jump counter $j\in\mathbb{N}$ that denotes the number of jumps that have occurred. 
A set $E\subset \R{}_{\ge 0}\times \mathbb{N}$ is a compact hybrid time domain if there exists $J\in \mathbb{N}$ such that
\(
    E = \bigcup_{j=0}^{J} [t_j, t_{j+1}]\times  \{j\}
\)
for some finite sequence of times $\{t_j\}_{j=0}^{J+1}$ satisfying $0 = t_0 \leq t_1\leq t_2 \leq ... \leq t_{J} \leq t_{J+1}$. A set $E \subset \mathbb{R}_{\geq 0} \times \mathbb{N}$ is a hybrid time domain if it is the union of compact hybrid time domains $E_j$ such that $E_0 \subset E_1 \subset E_2 \subset \ldots \subset E_j \ldots$.
To define the behavior of solutions during the flows of geometric hybrid systems, we employ the concept of local absolute continuity.

\begin{definition}[{Local absolute continuity~\cite[Sec.~A.2.1]{BulloLewis}}]
    \label{def:localAbsoluteContinuity}
    A function $\gamma: \R{}\supset I \to \M$ is said to be locally absolutely continuous if, for each $C^1$-function $\psi : \M \to \R{}$, the composition $\psi \circ \gamma : I \to \R{}$ is locally absolutely continuous. 
\end{definition}

\begin{definition}[Hybrid arc]
A function $\phi: E\to \M$ is a hybrid arc if $E$ is a hybrid time domain and, for each $j\in \mathbb{N}$, the curve $t\mapsto \phi(t, j)$ is locally absolutely continuous on the interval $I^j = \{t: (t, j)\in E\}$.     
\end{definition}

We define solutions to a geometric hybrid system $\hybrid$ as hybrid arcs whose flow and jumps are appropriately consistent with the data defining the system.

\begin{definition}[Solutions to geometric hybrid systems]\label{def:solution}
A map $\phi:\dom{\phi} \to \M$ is a solution to the geometric hybrid system $\hybrid = (C, F, D, G, \M)$ if $\phi$ is a hybrid arc and 
\begin{enumerate}[label=$(S_{\arabic*})$, start=0]
    \item $\phi(0,0)\in \overline{C}\cup D$;
    \item for each $j\in \mathbb{N}$ such that $\interior{I^j}\neq \varnothing$, $\phi$ satisfies
    \vspace{-3pt}
    \begin{align*}
        \begin{array}{lcl}
            \phi(t, j) \in C  & \textrm{for all} & t \in \mathrm{int\:}{I^j},\\
           \displaystyle{\frac{d\phi}{dt}}(t, j) \in F(\phi(t, j))  & \textrm{for {almost all }} & t\in I^j ;
        \end{array}
    \end{align*}
    \item for each $(t, j)\in \dom{\phi}$ such that $(t, j+1)\in \dom{\phi}$, 
    \vspace{-5pt}
    \begin{align*}
        \phi(t,j) \in D && \text{and} && 
        \phi(t,j+1) \in G(\phi(t,j)).
    \end{align*}
\end{enumerate}    
\end{definition}

Given a solution $\phi$, its total time of flow is $\sup_{t}\dom{\phi} \coloneqq \sup\{t : \exists j \textrm{ s.t. }(t,j)\!\in\!\dom{\phi}\}$. Similarly, the number of jumps experienced by $\phi$ is $\sup_{j}\dom{\phi}\coloneqq \sup\{j : \exists t \textrm{ s.t. }(t,j)\in\dom{\phi}\}$. We define $\sup\dom{\phi}\coloneqq (\sup_t\dom{\phi}, \sup_{j}\dom{\phi})$ and $\operatorname{length}(\dom{\phi}) \coloneqq \sup_{t}\dom{\phi} + \sup_{j}\dom{\phi}$.

A solution $\phi$ is \emph{nontrivial} if $\dom{\phi}$ contains at least two points. 
It is \emph{maximal} if there does not exist another solution $\psi$ to $\hybrid$ such that $\dom{\phi}$ is a proper subset of $\dom{\psi}$ and $\phi(t,j)= \psi(t,j)$ for all $(t,j) \in \dom{\phi}$. 
It is \emph{complete} if $\mathrm{length}(\dom{\phi})=\infty$.
It is \emph{precompact} if $\rge{\phi}\subset {\M}$ is precompact. 
The set of all solutions starting from $K \subset \M$ by $\widehat{\cal S}_{\hybrid}(K)$. The sets of all maximal solutions $\mathcal{S}_{\hybrid}$, and those starting from $K$, namely, $\maximalSol{\hybrid}{K}$, are defined analogously.

Now, extending~\cite[Assumption~6.5]{goebel_hybrid_2012} to the geometric setting, we provide regularity conditions that impart desirable properties to~$\hybrid$. These properties are leveraged in Sec.~\ref{sec:solutionSetProperties}.

\begin{definition}[Geometric hybrid basic conditions]\label{ass:hybrid_basic_conditions}
    The geometric hybrid dynamical system $\hybrid = (C, F, D, G, \M)$ is said to satisfy the {\em geometric hybrid basic conditions} if
    \begin{enumerate}[label = (A\arabic*), leftmargin=*]
        \item \label{item:HBC1} $\M$ is a $C^1$-manifold;
        \item $C $ and $D$ are subsets of $\M$ that are closed relative to~$\M$;
        \item \label{HBC:F}$F:\M \rightrightarrows \T{}{\M}$ is outer semicontinuous and locally~precompact relative to $C$, $F(x)$ is convex for all $x\in C$, and $C \subset \dom{F}$; 
        \item $G:\M \rightrightarrows \M$ is outer semicontinuous and locally~precompact relative to $D$, and $D \subset \dom{G}$.
    \end{enumerate}
\end{definition}

If $\M = \R{n}$ and the hybrid system $\hybrid$ satisfies Definition~\ref{ass:hybrid_basic_conditions}, it is shown in \cite[Chapter 6]{goebel_hybrid_2012} that $\hybrid$ is well-posed; in particular, the set of solutions to $\hybrid$ with compact time domains is sequentially compact and dynamical properties of $\hybrid$ are robust to arbitrarily small perturbations. This notion of well-posedness requires knowledge of a distance function on the manifold. Since we assume $\M$ to be a $C^1$-manifold, a distance function on $\M$ is not available. Consequently, a similar notion is not yet defined for the class of systems in \eqref{eq:HS}. A weaker, metric-independent, notion requiring only sequential closedness of the set of compact solutions to $\hybrid$ is defined in~\cite{jirwankar2025invariance}.

The following examples illustrate modeling of geometric hybrid dynamical systems on manifolds. %

{
\renewcommand{\angvel}{\xi}
\begin{example}[Quaternion stabilization]
    \label{ex:quaternion}
    We recall the quaternion kinematics in~\cite{mayhew2011quaternion}, given by
    \begin{align*}
        \dot{q} = \frac{1}{2}q \otimes v(\angvel) \quad (q, \angvel)\in\mathbb{S}^3 \times \R{3},
    \end{align*}
    where $\mathbb{S}^3 \coloneqq \{w \in \R{4} : |w| = 1\}$ is the unit sphere in $\R{4}$, $q\in\mathbb{S}^3$ denotes the quaternion describing the attitude of a rigid body, $\angvel\in\R{3}$ denotes the angular velocity of the rigid body in its body-fixed frame, $v(\angvel) \coloneqq \begin{pmatrix}
        0 & \angvel^\top
    \end{pmatrix}^\top$, and the quaternion multiplication~$\otimes$ is defined as $q_1 \otimes q_2 \coloneqq \begin{psmallmatrix}
            \eta_1\eta_2 - \varepsilon_1^\top \varepsilon_2\\
            \eta_1\varepsilon_2 + \eta_2\varepsilon_1 + \varepsilon_1\times\varepsilon_2
        \end{psmallmatrix},$
    where $q_i = {\begin{pmatrix}
        \eta_i &
        \varepsilon_i^\top
    \end{pmatrix}}^\top$ for $i=1,2$. Let $H \coloneqq \{-1,1\}$ denote the set of values for the logic variable $h$. Controlling $\angvel$ with the logic-based hybrid controller in~\cite{mayhew2011quaternion} and {}{\color{mygreen}\relax{}Appendix~\ref{app:QuaternionController}}, the resulting closed-loop system is a geometric hybrid system $\hybrid = (C, F, D, G, \M)$ with $\M \coloneqq \mathbb{S}^3\times H$, state $x = (q, h) \in \M$, flow and jump sets
    \begin{align*}
        C \coloneqq \{(q,h)\in \M : h\eta \geq -\delta\}, && D \coloneqq \overline{\M \setminus C}
    \end{align*}
    for some $\delta\in (0,1)$, and, for each $x\in \M$, 
    \begin{align*}
        F(x) \coloneqq \begin{pmatrix}
                \frac{1}{2}q\otimes v(-hK_\varepsilon \varepsilon),
                &
                0
            \end{pmatrix},
        &&
        G(x) \coloneqq \begin{pmatrix}
                q,
                &
                -h
        \end{pmatrix},
    \end{align*}
    respectively, where $K_\varepsilon \in \R{3\times 3}$ is a symmetric, positive definite matrix. As $C$ and $D$ are closed, and $F$ and $G$ are single-valued and continuous, $\hybrid$ satisfies the geometric hybrid basic conditions from Definition~\ref{ass:hybrid_basic_conditions}.
\end{example}
}

The following example models a billiard ball on a M\"{o}bius band. This is a geometric hybrid system that exhibits two types of jumps, with second-order continuous-time dynamics that are transformed into the first-order form in~\eqref{eq:HS}.

\begin{example}[Billiard on the M\"{o}bius band]
    \label{ex:mobius-hybridDynamics}
    Recall $\mobius$ and $\rectangle$ from Example~\ref{ex:MobiusDef}. Define $\rectangle_{\constraintSet} \coloneqq [0,1]\times [-1,1]\subset \rectangle$,~and consider a point mass moving on $\mobius$ in a compact set $S\subset \mobius$, where $\constraintSet \coloneqq \rectangle_{\constraintSet}/\sim$ and ``$\sim$'' denotes the equivalence relation on $\rectangle$. We denote the position and velocity of the point mass by $y_1 \in \constraintSet$ and $y_2\in \T{y_1}{\mobius}$, respectively, and let $y\coloneqq (y_1, y_2)\in \T{}{\mobius}$. We model the motion of this point mass as a geometric hybrid system $\hybrid$ that evolves continuously inside $S$, has impacts at the boundary of $S$, and approaches a nonempty, compact set $\tilde{\A} \subset \interior{S}$.     

    To this end, let $Q\coloneqq \{-1, 1\}$ and define the smooth, disconnected manifold $\M \coloneqq \T{}{\mobius} \times Q$. For each $q\in Q$, we define a smooth function $V_q \in \PD{\tilde{\A}}$ on $\mobius$. Suppose
    \begin{align}\label{eq:mu}
        \mu\coloneqq \min_{\substack{q\in Q\\ y_1 \in \crit{V_q}\setminus \tilde{\A}}} \left\{ V_q(y_1)- V_{-q}(y_1)\right\} \;\; > \; 0,
    \end{align}
    where $\crit{V_q} \subset \mobius$ denotes the set of critical points of $V_q$. Condition~\eqref{eq:mu} is standard in synergistic control~\cite[Ch.~7]{HybridFeedbackControl}. It amounts to the existence of a family $\{V_q\}_{q\in Q}$ of potential fields whose undesired critical points are separated by $\mu > 0$. With this construction, pick $\delta\in (0, \mu)$. 
    
    Now, let $x\coloneqq (y, q)\in \M$ denote the state of the point mass. 
    While the point mass is moving in the set $S$ with velocity pointing ``inward,'' it evolves  under the action of a dissipative friction force \cite[Def.~4.65]{BulloLewis} and the smooth potential field $V_q$ while keeping $q$ constant, leading to the following dynamics:
    \begin{align*}
        \begin{array}{lr}
             \dot{x} = F(x)\coloneqq \left(\hspace{-0.15cm}
             \begin{array}{c}
             \mathfrak{S}(y) + \vlift{y}{-\grad{V_q}{y_1} - b y_2} \\ 0
             \end{array}\hspace{-0.15cm}\right) & \hspace{-0.2cm} x \in {C}
        \end{array}
    \end{align*}
    where $F: \M \rightrightarrows \T{}{\M}$ is the flow map defined as above {for each $x\in \M$.}
    The map $\mathfrak{S}: \T{}{\mobius}\to \T{}{\T{}{\mobius}}$ denotes the \emph{geodesic spray} on $\mobius$, namely, the unique vector field on $\T{}{\mobius}$ whose integral curves, when projected onto $\mobius$ using the canonical projection map $\pi$, are geodesics on $\mobius$ under the Levi-Civita connection; see~\cite[Ch.~3, Lem.~2.3]{doCarmo1992}. For each $y = (y_1, y_2)\in \T{}{\mobius}$, the map $\operatorname{vlft}_{y} : \T{y_1}{\mobius} \to \T{y}{\T{}{\mobius}}$ denotes the \emph{vertical lift}%
    \footnote{The vertical lift at $y =(y_1, y_2)\in \T{}{\mobius}$ of a vector $w\in \T{y_1}{\mobius}$ is defined as $\vlift{y}{w}\coloneqq \frac{d}{dt}(y_2 + tw) \big|_{t=0}$ for each $w\in \T{y_1}{\mobius}$; see~\cite[Def.~3.73]{BulloLewis}.}
    at $y$. Noting that $\dim{\mobius} = 2$, we denote by $\Gamma^i_{jk}$, $i,j,k\in\{1,2\}$, the Christoffel symbols of the Levi-Civita connection $\nabla$ under a given coordinate chart on $\mobius$. In such a chart, the coordinate expressions for $\mathfrak{S}$ and $\mathrm{vlft}$ around $x = (y,q)$ are\footnote{We have used the Einstein summation convention. In particular, $a^i b_i$ denotes  $\Sigma_{i}a^i b_i$ since summation over $i$ is implied; see~\cite[Sec.~2.2]{BulloLewis}.}
    \begin{align}
        \mathfrak{S}(y) &= \left.\widetilde{y_2}^{(i)} \frac{\partial}{\partial \widetilde{y_1}^{(i)}} \right|_{y} - \left.\Gamma^{i}_{jk} \widetilde{y_2}^{(j)} \widetilde{y_2}^{(k)} \frac{\partial}{\partial \widetilde{y_2}^{(i)}}\right|_{y},  \label{eq:coord_spray}
        \\
        \vlift{y}{w} &= \left.\widetilde{w}^{(i)} \frac{\partial}{\partial \widetilde{y_2}^{(i)}}\right|_{y}, \label{eq:coord_vlift}
    \end{align}
    where $\widetilde{y_1}^{(i)}$, $\widetilde{y_2}^{(i)}$, and $\widetilde{w}^{(i)}$ denote the coordinate representations of $y_1\in \mobius$, $y_2\in \T{y_1}{\mobius}$, and a tangent vector $w \in \T{y_1}{\mobius}$, respectively. The constant $b > 0$ in the flow map $F$ characterizes the dissipative force. 
    
    The set $C \subset\M$ allows flow when the point mass does not collide with $\partial \constraintSet$, while also ensuring that the current choice of $q$ is the \emph{better one}, in the sense that the value of $V_q$ is smaller than $V_{-q}$ by at least $\delta$.
    Alternatively, a jump is triggered when the point mass collides with $\partial \constraintSet$ with velocity pointing ``outwards,'' or if a better choice of $q$ exists in the sense described above. These conditions are captured by the sets $D_1 \subset \M$ and $D_2 \subset \M$, respectively. The jump set of the hybrid system is defined by $D\coloneqq D_1 \cup D_2$. The formal construction of the sets $C$, $D_1$, and $D_2$ is deferred to Example~\ref{ex:mobius-solutions} after we introduce the geometric notion of tangent cones to sets on $C^1$-manifolds, which formally characterizes the notion of an ``outward'' direction used above.

    When the point mass collides with $\partial \constraintSet$, its position $y_1$ and velocity $y_2$ are reset according to the map $G_1 : D_1\rightrightarrows \M$:
    \begin{align*}
             x^+ \in G_1(x) \coloneqq \{ (g_1(y_1), P_{0\mapsto 1}^{\gamma}(y_2^{||} - \lambda y_{2}^{\perp}), q) : \lambda\in [0,\lambda_{\max}]\} 
    \end{align*}
    for all $x \in D_1$, where $\lambda_{\max} \in (0,1)$ is a constant. For each $x\in D_1$, the position $y_1$ above is reset to $g_1(y_1) \coloneqq [z_1, z_2- \operatorname{sign}(z_2)\varepsilon]_{\mobius}$, where $\varepsilon > 0$ is arbitrarily small and, recall from Example~\ref{ex:MobiusDef} that  $y_1 = [z_1, z_2]_{\mobius}$ for $(z_1, z_2)\in \rectangle_S$. The map $G_1$ keeps $q$ unchanged after jumps. For the velocity update, the term $y_2^{||}\in \T{y_1}{\mobius}$ (resp., $y_2^{\perp}$) denotes the component of $y_2$ parallel (resp., orthogonal) to $\partial S$ at $y_1$ obtained using the metric $g_{\mobius}$ on $\mobius$ from Example~\ref{ex:MobiusDef}. In particular, only the perpendicular component of the velocity is scaled by the coefficient of restitution $\lambda$.
    {
        If $\lambda\in [0, \lambda_{\max}]$ is unknown, its uncertainty is captured by the set-valued nature of the map $G_1$. 
    }
    \noindent
    {The scaled velocity is parallelly transported from the tangent space at $y_1$ to the tangent space at $g_1(y_1)$ along the smooth curve $\gamma:[0,1]\to \mobius$, defined as $\gamma(t)\coloneqq [z_1, z_2 -t \operatorname{sign}(z_2)\varepsilon]_{\mobius}$ for each $t\in[0,1]$, such that $\gamma(0)=y_1$ and $\gamma(1) = g_1(y_1)$. 
    The parallel transport map along $\gamma$ is denoted by $P_{0\mapsto 1}^{\gamma} : \T{y_1}{\mobius} \to \T{g_1(y_1)}{\mobius}$. Instead of an explicit expression, it is defined as follows: for each $v\in \T{y_1}{\mobius}$, $P^{\gamma}_{0\mapsto 1}(v) \coloneqq X(\gamma(1))$, where $X$ is the unique $C^1$ vector field that is ``parallel along $\gamma$'' and satisfies $X(\gamma(0))=v$; see~\cite[Ch.~4]{Lee_Riemannian} for more details}.
    
    When jumps are triggered in $D_2$, $q$ is reset while keeping $y$ constant using the map $G_2 : D_2 \to \M$ as follows:
    \begin{align*}
        \begin{array}{cc}
             x^+ = G_2(x) \coloneqq  (y,  -q) & x \in D_2.
        \end{array}
    \end{align*}
    {}
    {\color{mygreen}\relax{}{\color{mygreen}Then, the jump map $G:\M \rightrightarrows \M$ is defined as follows:
    \begin{align*}
        G(x) \coloneqq \left\{
        \begin{array}{ccl}
           G_1(x)  & \text{if } & x\in D_1 \setminus D_2\\
           G_2(x)  & \text{if } & x\in D_2 \setminus D_1\\
           G_1(x) \cup G_2(x) & \text{if } & x\in D_1 \cap D_2\\
           \varnothing & \text{if } & x\notin D
        \end{array}
        \right. \quad \forall x\in \M. 
    \end{align*}}}
    
Combining the flow and the jump behavior yields the hybrid system $\hybrid = (C, F, D, G, \M)$ of the form~\eqref{eq:HS}. 
\end{example}

\section{Basic Properties of the Solution Set of \texorpdfstring{$\hybrid$}{H}}
\label{sec:solutionSetProperties}

We establish desirable properties of the set of solutions $\mathcal{S}_{\hybrid}$ when $\hybrid$ satisfies the geometric hybrid basic conditions in Definition~\ref{ass:hybrid_basic_conditions}. In particular, we provide sufficient conditions for the existence of solutions to $\hybrid$ from a point. Then, we show that the set $\mathcal{S}_\hybrid$ is closed using a sequential argument.

\begin{figure*}[t]
    \centering
    {\includegraphics[width = \linewidth, trim = 0cm 0cm 0cm 16.5cm]{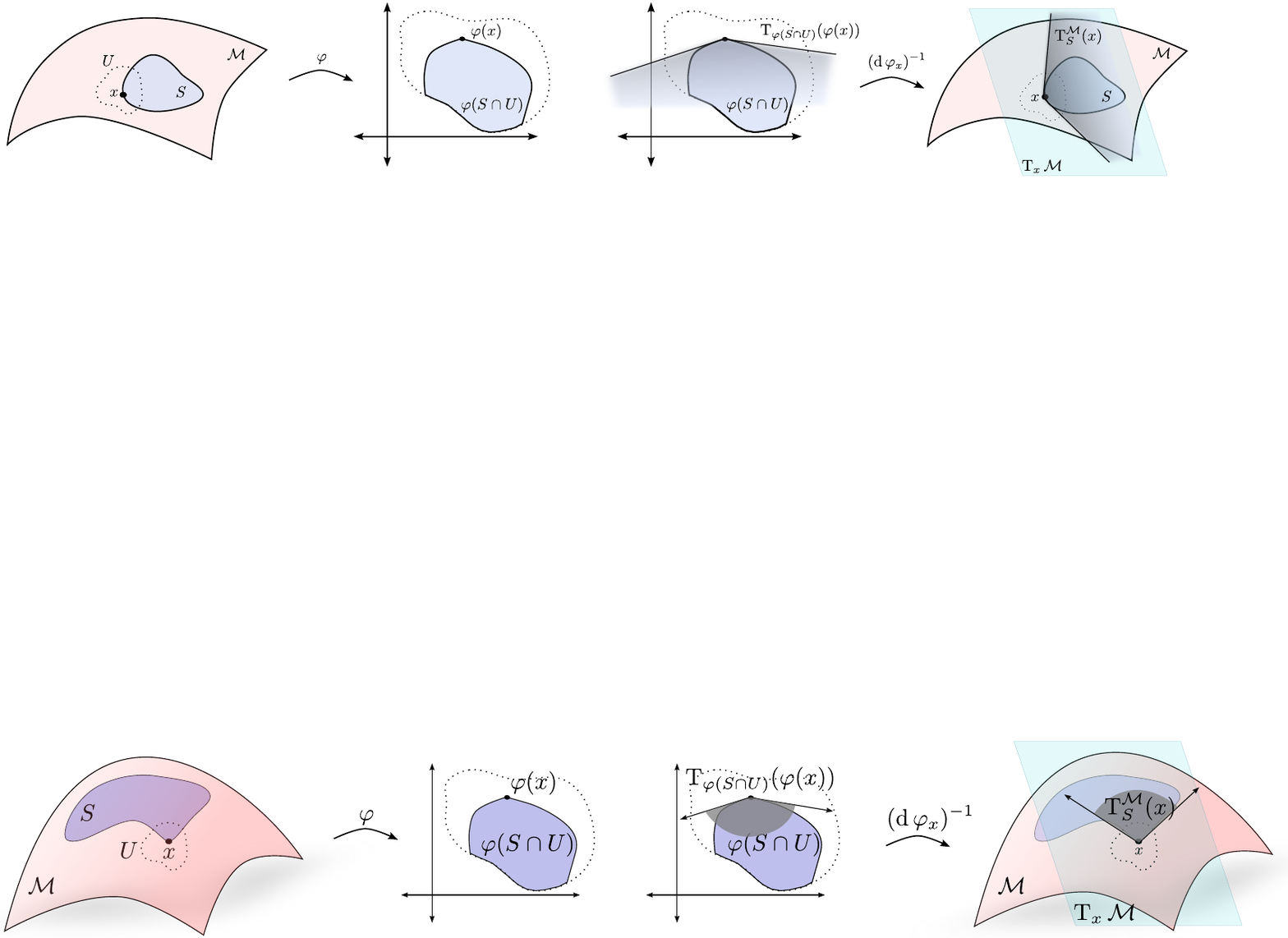}}
    \caption{Construction of the geometric tangent cone $\Tcone{S}{x}$ using a coordinate chart $(U, \varphi)$ as in Definition~\ref{def:tangentConeManifolds}.\vspace{-6pt}}
    \label{fig:TangentCone}
\end{figure*}

\subsection{Existence of Solutions to $\hybrid$}
\label{sec:existence}

We establish the existence of nontrivial solutions to hybrid systems of the form in \eqref{eq:HS}. A standard approach to obtaining such a result in the Euclidean case asks that, near an initial condition belonging to the flow set $C$, the values of $F$ contain a vector tangent to $C$; see \cite[Prop. 2.34]{HybridFeedbackControl}. On Euclidean space, this tangency condition is captured by the contingent cone \cite[Prop.~6.2]{Rockafellar1997}. Since in our case $C$ is a subset of the manifold $\M$, we transport the contingent cone construction to $\M$ via coordinate charts, generalizing the result in \cite[Def.~3.2]{MOTREANU1982116} which carried Clarke's tangent cone to $C^1$-manifolds.
\begin{definition}[{Tangent cone to subsets of manifolds}]
    \label{def:tangentConeManifolds}
    Let $\M$ be an $n$-dimensional $C^1$-manifold and let $S \subset \M$ be nonempty. The tangent cone $\Tcone{S}{x}$ at $x\in S$ is defined as
    \begin{align*}
        \Tcone{S}{x} \coloneqq \left(\diffFunc{\varphi_x}\right)^{-1}(\TconeEuclidean{\varphi(S\cap U)}{\varphi(x)}),
    \end{align*}
    where\footnote{Note that the inverse of the map $\diffFunc{\varphi_x}:\T{x}{\M} \to \R{n}$ exists as it is an isomorphism~\cite[Prop.~3.6(d)]{Lee}.} $(U, \varphi)$ is any coordinate chart at $x$, and $\TconeEuclidean{\varphi(S\cap U)}{\varphi(x)}$ denotes the Euclidean tangent cone to the set $\varphi(S\cap U)\subset \mathbb{R}^n$ at the point $\varphi(x)$; see Figure \ref{fig:TangentCone} for an illustration.
\end{definition}

That Definition~\ref{def:tangentConeManifolds} does not depend on the chosen coordinate chart, and is therefore well-defined on $\M$, is shown in Appendix~\ref{appendix:B}. Equipped with the tangent cone on $\M$, we return to the existence of nontrivial solutions to geometric hybrid systems and state the following result.

\begin{proposition}[Existence of solutions]
    \label{prop:existence}
Consider the geometric hybrid dynamical system $\hybrid = (C, F, D, G, \M)$ that satisfies the geometric hybrid basic conditions in Definition~\ref{ass:hybrid_basic_conditions}. Pick $\xi \in {C}\cup D$. If $\xi\in D$ or
\begin{enumerate}[label=(${VC}^{g}$), leftmargin=1.2cm]
    \item  there exists an open neighborhood $U$ of $\xi$ such that, for each $x\in U \cap C$, \label{VC}
    \(
        F(x) \cap \Tcone{C}{x} \neq \varnothing,
    \)
\end{enumerate}
then, there exists a nontrivial solution $\phi \in \setofSol{\hybrid}{\xi}$. If \ref{VC} holds for every $\xi\in {C}\setminus D$, then there exists a nontrivial solution to $\hybrid$ from each point of ${C}\cup D$, and each maximal solution~$\phi$ to $\hybrid$ satisfies exactly one of the following conditions:

\begin{enumerate}[label=\alph*)]
    \item $\phi$ is complete; \label{item:a-bullet}
    
    \item $\dom \phi$ is bounded, the interval $I^J = \{t : (t,J)\in\dom{\phi}\}$, where $J = \sup_j \dom{\phi}$, has a nonempty interior, and $t\mapsto \phi(t,J)$ is a maximal solution to $\dot{x} \in F(x)$  $\;x\in C$ such that there does not exist a compact set $K \subset \M$ satisfying $\lim_{t \nearrow T}\phi(t,J) \in \interior{K}$, where $T = \sup_t \dom{\phi}$; \label{item:b-bullet} %

    \item $\phi(T,J) \notin C\cup D$, where $(T,J)\coloneqq \sup\dom{\phi}$. \label{item:c-bullet}
    
\end{enumerate}
Furthermore, the following hold:
\begin{enumerate}[label=\roman*)]
    \item If $G(D)\subset C\cup D$, then item c above does not hold. \label{item:1-bullet}
    \item If $C$ is compact, then item b above does not hold.  \label{item:2-bullet}
\end{enumerate}
\end{proposition}

\begin{proof}   
    If $\xi\in D$, then there exists a nontrivial solution that jumps. To prove the existence of a nontrivial solution to $\hybrid$ if \ref{VC} holds, we obtain equivalent flow dynamics in local coordinates. Then, we prove the existence of a nontrivial solution to those dynamics if an equivalent viability condition in local coordinates, denoted \ref{VC_local} below, holds. Finally, we map the resulting nontrivial solution to the flow dynamics in local coordinates back to the manifold $\M$ and show that it is indeed a nontrivial solution to the original hybrid system~$\hybrid$. 
    
    If $\xi \in C$, consider a coordinate chart $(V, \varphi)$ of $\M$ around $\xi$. We denote the coordinates of $x \in C\cap V$ by $y \coloneqq \varphi(x) \in \R{n}$. Then, the local representation of $F$, denoted by $\widetilde{F}$ and obtained using the differential\footnote{The differential of $\varphi$ exists as $\varphi$ is a $C^1$-diffeomorphism onto its image because $\M$ is a $C^1$-manifold; see~\cite[Ex.~2.14(b)]{Lee}.} of $\varphi$, denotes the flow map for the dynamics of $y$.
    {}
    {\color{mygreen}\relax{}{\color{mygreen}In particular, for the purpose of constructing $\widetilde{F}$, if there exists a solution $\phi$ to $\dot{x} \in F(x) \;x\in C$, then define $\widetilde{\phi} = \varphi\circ\phi : \dom{\phi} \to \R{n}$ so that $\widetilde{\phi}(t,j) = \varphi(\phi(t,j))$ for all $(t,j)\in\dom{\phi}$. We observe that $\frac{d\widetilde{\phi}}{dt} (t,j) = \diffFunc{\varphi}_{\phi(t,j)}(\frac{d\phi}{dt}(t,j))$ for almost all $(t,j)\in \dom\phi$. It is easy to check that $\widetilde{\phi}$ is a solution to the flow dynamics of $y$ given as follows:}}
\begin{align}\label{eq:localFlowDynamics}
    \dot{y} \in \widetilde{F}(y) \qquad  y\in \varphi(C\cap V),
\end{align}
where $\widetilde{F} : \varphi(C\cap V) \subset \R{n} \rightrightarrows \R{n}$ is defined by $\widetilde{F}(y) \coloneqq \diffFunc{\varphi_{\varphi^{-1}(y)}}(F(\varphi^{-1}(y))) = \{\diffFunc{\varphi_{\varphi^{-1}(y)}}(f) \in \R{n} : f \in F(\varphi^{-1}(y))\}$ for each $y\in\dom{\widetilde{F}}$. If the viability condition~\ref{VC} holds, then Definition~\ref{def:tangentConeManifolds} yields the following coordinate representation of \ref{VC}:
\begin{enumerate}[label=($VC^g_{\mathrm{local}}$), leftmargin=1.5cm]
    \item  there exists an open neighborhood $\widetilde{V}$ of $\varphi(\xi)$ such that, for each $y\in \widetilde{V} \cap \varphi(C \cap V)$, \label{VC_local}
    \(
    \widetilde{F}(y) \cap \TconeEuclidean{\varphi(C\cap V)}{y} \neq \varnothing.
    \) 
\end{enumerate}
{}{\color{mygreen}\relax{}{\color{mygreen}We now show that if \ref{VC_local} holds, then there exists a nontrivial solution to \eqref{eq:localFlowDynamics} from $\varphi(\xi)$. As $F$ is outer semicontinuous and locally precompact relative to $C$, \cite[Prop. 6.3.2]{Beer1993} implies that $F$ is upper semicontinuous relative to $C$. Then, as $\varphi$ is a diffeomorphism~\cite[Ex.~2.14(b)]{Lee}, $\diffFunc{\varphi}$ is a linear isomorphism \cite[Prop. 3.6]{Lee} and $\widetilde{F}$ is, therefore, upper semicontinuous relative to $\varphi(C\cap V)$. Furthermore, since $F(x)$ is convex for each $x\in C$ and $\diffFunc{\varphi}_x$ is a linear isomorphism, $\widetilde{F}$ is convex valued due to \cite[p. 36]{Boyd_Vandenberghe_2004}. Therefore, Assumption (A)-(C) in~\cite{carjua2000viability} hold. Assumption (D) in~\cite{carjua2000viability} holds as $\widetilde{F}$ is locally precompact relative to $\varphi(C\cap V)$. Additionally, as $\varphi$ is a diffeomorphism, $\varphi(C\cap V) = \varphi(C) \cap \varphi(V)$ is locally closed. }}
Then, by~\cite[Thm.~2.3]{carjua2000viability}, there exist $T > 0$ and a solution $\widetilde{\phi} : [0, T] \to \R{n}$ to~\eqref{eq:localFlowDynamics} with $\widetilde{\phi}(0) = \varphi(\xi)$ such that $\widetilde{\phi}(t)\in \varphi(C\cap V)$ for all $t\in (0,T)$. Next, define $\phi : [0,T]\times \{0\} \to \M$ as $\phi(t,0)\coloneqq \varphi^{-1}\circ\widetilde{\phi}(t)$ for all $t\in [0, T]$. We show that $\phi$ is a nontrivial solution to $\hybrid$ with $\phi(0,0)=\xi$. As $\widetilde{\phi}(t)\in \varphi(C\cap V)$ for all $t\in (0,T)$, it follows that $\phi(t,0)\in C\cap V$ for all $t\in (0,T)$. Furthermore, as $\widetilde{\phi}$ is a solution to \eqref{eq:localFlowDynamics}, we have $\frac{d\widetilde{\phi}}{dt}(t)\in \widetilde{F}(\widetilde{\phi}(t))$ for almost all $t\in (0, T)$. Using the definition of $\widetilde{F}$ and the fact that $\varphi$ is a diffeomorphism onto its image, the dynamics in \eqref{eq:localFlowDynamics} results in $\frac{d\phi}{dt}(t,0) = \diffFunc{\varphi_{\widetilde{\phi}(t)}^{-1}}(\frac{d\widetilde{\phi}}{dt}(t)) \in \diffFunc{\varphi}^{-1}_{\widetilde{\phi}(t)}(\widetilde{F}(\widetilde{\phi}(t))) =  F(\phi(t,0))$ for almost all $t\in (0,T)$. Finally, since $\tilde{\phi}$ is a solution to~\eqref{eq:localFlowDynamics}, it is locally absolutely continuous. As $\varphi$ is a $C^1$-diffeomorphism, {\color{mygreen}\relax{}using Lemma~\ref{lemma:locAbsCont_coord},} $t\mapsto \phi(t,0)$ is locally absolutely continuous.  Consequently, as $T > 0$, $\phi$ is a nontrivial solution to $\hybrid$.

{}
{\color{mygreen}\relax{}
    Next, assume that \ref{VC} holds for each point in $C\setminus D$. From the steps above, a nontrivial solution to $\hybrid$ exists from each point of $C\cup D$. Let $\phi$ be a maximal solution to $\hybrid$ and let $(T,J)\coloneqq\sup\dom\phi$. Since $\phi$ is complete if and only if $\dom\phi$ is unbounded, exactly one of the following holds. Either $\dom\phi$ is unbounded, or $\dom\phi$ is bounded and $(T,J)\in\dom\phi$, or $\dom\phi$ is bounded and $(T,J)\notin\dom\phi$. We show that these three cases correspond to items~\ref{item:a-bullet}, \ref{item:c-bullet}, and \ref{item:b-bullet}, respectively. The first case is item~\ref{item:a-bullet} by the definition of a complete solution. In the second case, $\max\{T,J\} < \infty$ and $\phi(T,J)$ is defined. If $\phi(T,J)\in D$, then $\phi$ can be extended through a jump, and if $\phi(T,J)\in C\setminus D$, then $\phi$ can be extended through flow, since \ref{VC} holds at $\phi(T,J)$ and the first part of this proof yields a nontrivial solution from that point. Both scenarios contradict maximality of $\phi$, so $\phi(T,J)\notin C\cup D$, which is item~\ref{item:c-bullet}.
    
    In the third case, $\interior{I^J}\neq\varnothing$ and $I^J$ has the form $[\min I^J, T)$. Furthermore, $t\mapsto\phi(t,J)$ is a maximal solution to $\dot{x}\in F(x)$ $x\in C$, since any extension of it would extend $\phi$. Proceeding by contradiction, suppose that $\xi^\star\coloneqq\lim_{t\nearrow T}\phi(t,J)$ exists and that $\xi^\star\in\interior{K}$ for some compact set $K\subset\M$. Let $(V,\varphi)$ be a chart around $\xi^\star$. Since $\widetilde{F}$ as above is locally precompact relative to $\varphi(C\cap V)$, there exist a neighborhood $\widetilde{V}$ of $\varphi(\xi^\star)$ and $M > 0$ such that $\widetilde{F}(\widetilde{V}\cap\varphi(C\cap V)) \subset M\mathbb{B}$. Given that $\phi(t,J)\to\xi^\star$ as $t\nearrow T$, there exists $t_0\in I^J$ such that $\varphi(\phi(t,J))\in\widetilde{V}$ for all $t\in [t_0, T)$, and hence $\varphi(\phi(t,J)) - \varphi(\phi(s,J)) \in M|t-s|\mathbb{B}$ for each $s,t\in [t_0,T)$. Then, letting $t\nearrow T$ yields $\varphi(\xi^\star) - \varphi(\phi(s,J)) \in M|T-s|\mathbb{B}$ for all $s\in [t_0,T)$. Therefore,  $t\mapsto\varphi(\phi(t,J))$ with the value $\varphi(\xi^\star)$ at $T$ is Lipschitz, and therefore absolutely continuous, on $[t_0,T]$.
    
    Given that $C$ is closed and $\phi(t,J)\in C$ for all $t\in\interior{I^J}$, we obtain $\xi^\star\in C$. Hence setting $\phi(T,J)\coloneqq\xi^\star$ extends $\phi$ to a solution to $\hybrid$ with domain $\overline{\dom\phi}$, which contradicts maximality of $\phi$ and proves item~\ref{item:b-bullet}.
    
    The three cases above therefore correspond to items~\ref{item:a-bullet}, \ref{item:c-bullet}, and \ref{item:b-bullet}, as claimed, so each maximal solution satisfies at least one of them. The correspondence also holds in the converse direction, since item~\ref{item:a-bullet} implies the first case, item~\ref{item:c-bullet} presupposes $(T,J)\in\dom\phi$ and so implies the second, and item~\ref{item:b-bullet} implies the third, as an arc satisfying it admits no extension to $t = T$. Hence each maximal solution to $\hybrid$ satisfies exactly one of items~\ref{item:a-bullet}--\ref{item:c-bullet}.
}

Finally, if $G(D)\subset C\cup D$, item i holds as $\rge{\phi} \subset C\cup D$. Item ii holds as $C$ compact precludes item b.
\end{proof}

\begin{example}[Quaternion stabilization, revisited]
    \label{ex:quaternion-solutions}
    Recall $\hybrid = (C, F, D, G, \M)$ from Example~\ref{ex:quaternion}. %
    To prove existence of solutions to $\hybrid$ from each point in $\M$, we show that the viability condition~\ref{VC} in Proposition~\ref{prop:existence} holds everywhere in $C\setminus D$. As $C\setminus D = \interior{C}$, it follows that $\Tcone{\interior{C}}{x} = \T{x}{C}$ for all $x\in \interior{C}$. Therefore, $F(x)\in \Tcone{\interior{C}}{x}$ for each $x\in C\setminus D$, thus proving the existence of nontrivial solutions to $\hybrid$ from each point in $\M$ since $C\cup D = \M$. Furthermore, as $\M$ is compact, items i and ii of Proposition~\ref{prop:existence} hold. Consequently, each maximal solution to $\hybrid$ is complete.
\end{example}

\begin{example}[Billiard on the M\"{o}bius band, revisited]
    \label{ex:mobius-solutions}
    Recall $\hybrid = (C, F, D, G, \M)$ from Example~\ref{ex:mobius-hybridDynamics}. The description of the flow set $C$ therein is captured as follows:
    \begin{align}
        \label{eq:mobius-flowSet}
        {C} \coloneqq \left\{x\in \M :
        \begin{array}{c}
        y_1 \in \constraintSet, {y_2 \in \T{y_1}{\mobius},}\\
        V_q(y_1) - V_{-q}(y_1)\leq \delta
        \end{array} \right\}, 
    \end{align}
    Similarly, the jump set $D = D_1 \cup D_2$ is captured as follows:
    \begin{align}
        \begin{aligned}
            D_1 & \coloneqq \{x \in \M : y_1 \in \partial \constraintSet, y_2\in \overline{\T{y_1}{\mobius} \setminus \operatorname{T}_{\constraintSet}^{\mobius}(y_1)} \}, 
            \\
            D_2 & \coloneqq \{x\in \M : y_1 \in \constraintSet, V_q(y_1) - V_{-q}(y_1)\geq \delta\}. 
        \end{aligned}
    \end{align}
    {}
    {\color{mygreen}\relax{}Note that $C \cup D = \M$. To show existence of solutions to $\hybrid$ using Proposition~\ref{prop:existence}, note that $\hybrid$ satisfies the geometric hybrid basic conditions as $C$ and $D$ are closed, $F$ and $G_2$ are single valued and continuous, and $G_1$ is outer semicontinuous and locally precompact relative to $D_1$.
    
    {\color{mygreen}Indeed, to establish the properties of $G_1$, construct the map $\Phi : D_1\times [0, \lambda_{\max}] \to \M$, defined as
    \begin{align}
    \label{eq:Phi}
        \Phi(x,\lambda)\coloneqq (g_1(y_1), \mathfrak{P}(y_1, y_2^{||}-\lambda y_2^{\perp}), q)
    \end{align}
    for each $(x,q)\in D_1\times [0, \lambda_{\max}]$, where $\mathfrak{P}(y_1, v)\coloneqq P^{\gamma}_{0\mapsto 1}(v) = X(\gamma(1))$ for each $y_1\in \partial S$ and each $v\in \T{y_1}{\mobius}$, where recall from Example~\ref{ex:mobius-hybridDynamics} that $X$ is the vector field ``parallel along $\gamma$.'' The argument ``$y_1$'' in $\mathfrak{P}$ specifies that the curve $\gamma$ starts from $y_1\in\mobius$. We establish in Appendix~\ref{app:Phi_continuous} that $\Phi$ is continuous. 

    Then, to show that $G_1$ is outer semicontinuous relative to $D_1$, observe that $G_1(x)= \{\Phi(x,\lambda) : \lambda\in [0, \lambda_{\max}]\}$ for each $x\in D_1$. Consider any sequence $\{x_i\}_{i\in\mathbb{N}}\subset D_1$ such that $x_i \to x$ and any sequence $y_i\in G_1(x_i)$ such that $y_i \to y \in \M$. As $D_1$ is closed, $x\in D_1$. For each $i\in\mathbb{N}$, there exists $\lambda_i\in [0, \lambda_{\max}]$ such that $y_i = \Phi(x_i, \lambda_i)$. As $\lambda_i\in [0, \lambda_{\max}]$ for each $i\in\mathbb{N}$ and using compactness of $[0, \lambda_{\max}]$, we extract a convergent subsequence $\{\lambda_{i_k}\}_{k\in\mathbb{N}}$ such that $\lambda_{i_k} \to \lambda \in [0, \lambda
    _{\max}]$. Using continuity of $\Phi$,
    \begin{align*}
        y = \lim_{k\to\infty}y_{i_k} = \lim_{k\to\infty} \Phi(x_{i_k}, \lambda_{i_k}) = \Phi(x ,\lambda). 
    \end{align*}
    Since $\lambda\in[0, \lambda_{\max}]$, $\Phi(x,\lambda)\in G_1(x)$, causing $y\in G_1(x)$. Therefore, $G_1$ is outer semicontinuous relative to $D_1$. Next, local precompactness of $G_1$ relative to $D_1$ follows using continuity of $\Phi$ and compactness of $[0,\lambda_{\max}]$, together with the fact that $\M$ is locally compact as it is a smooth manifold. Therefore, it is indeed true that $\hybrid$ satisfies the geometric hybrid basic conditions in Definition~\ref{ass:hybrid_basic_conditions}. 
    
    }
    }

    Then, note that
    \begin{align}
        \label{eq:CminusD}
        \hspace{-0.2cm} 
        C\setminus D \coloneqq \left\{x\in \M :
        \begin{array}{c}
        y_1 \in \constraintSet, y_2 \in \interior{\operatorname{T}_{\constraintSet}^{\mobius}(y_1)}, \\
        V_q(y_1) - V_{-q }(y_1) < \delta
        \end{array} \right\}.
    \end{align}
    Pick $\xi\in C\setminus D$. If $\xi\in \interior{(C\setminus D)}$,~\ref{VC} from Proposition~\ref{prop:existence} holds trivially with $U\coloneqq \interior{(C\setminus D)}$ as, for each $x\in U$, $\Tcone{C}{x} = \T{x}{\M}$. 
    Alternatively, pick $\xi\in \partial(C\setminus D)$ and a coordinate chart $(U, \varphi)$ of $\M$ at $\xi$, with $U$ small enough so that $U\cap D = \varnothing$; i.e., $C \cap U = (C\setminus D) \cap U$. For each $x\in \interior{(C\cap U)}$, $\Tcone{C}{x} = \T{x}{\M}$, causing $F(x) \in \Tcone{C}{x}$. Alternatively, pick any $x\in \partial (C\cap U)$. Then,
    \begin{align}
        \Tcone{C}{x} & = (\diffFunc{\varphi_{x}})^{-1}\left( \TconeEuclidean{\varphi((C\setminus D)\cap U)}{\varphi(x)}\right). \label{eq:Tcone-local}
    \end{align}
    Now, we denote the coordinate representation of $x = ((y_1, y_2), q)$ in the chart $(U, \varphi)$ by $\widetilde{x} = \varphi(x) =  ((\widetilde{y}_1, \widetilde{y}_2), q)$ such that $\widetilde{y}_1 = (\widetilde{y}_1^{(1)}, \widetilde{y}_1^{(2)})\in \R{2}$, and $\widetilde{y}_1^{(1)}$ denotes the component of $\widetilde{y}_1$ parallel to $\partial S$ and $\widetilde{y}_1^{(2)}$ denotes the component of $\widetilde{y}_1$ perpendicular to $\partial S$. 
    Then, as $S$ is a smooth manifold with boundary (see~\cite[Ex.~10.3]{Lee}) and a submanifold\footnote{We refer the reader to~\cite[Ch.~5]{Lee} for a detailed treatment of submanifolds.} of $\mobius$, the projection of $\varphi(C\cap U)$ onto the first two components is a half-space in $\R{2}$. Therefore, assume {\color{mygreen}\relax{}{\color{mygreen}without loss of generality }}that $\widetilde{y}_1^{(2)} \leq 0$, and denote the projection of $\varphi(C\cap U)$ onto the first two components by
    \begin{align*}
        \widetilde{S} \coloneqq \left\{ (\widetilde{y}_1^{(1)}, \widetilde{y}_1^{(2)})\in \R{2} : \widetilde{y}_1^{(2)} \leq 0\right\}.
    \end{align*}
    Now, computing the Euclidean tangent cone to $\varphi((C\setminus D) \cap U)$ at $\widetilde{x} = \varphi(x)$, we obtain that
    \begin{align}
        \TconeEuclidean{\varphi((C\setminus D) \cap U)}{\widetilde{x}} &= \TconeEuclidean{\widetilde{S}}{\widetilde{y}_1} \times \R{2} \times \{0\} \label{eq:mobius-tangentCone}
        \\
        &= \{(v_1, v_2)\in \R{2} : v_2 \leq 0\} \times \R{2} \times \{0\}.  \nonumber
    \end{align}
    Now, to establish $F(x) \in\Tcone{C}{x}$, we use~\eqref{eq:Tcone-local} and show that
    \begin{align}
    \label{eq:mobius-inclusion}
        \diffFunc{\varphi_x}(F(x)) \in \TconeEuclidean{\varphi(U\cap (C\setminus D))}{\varphi(x)}.
    \end{align}

    Indeed, note from~\eqref{eq:CminusD} and $x\in \partial (C\cap U) = \partial ((C\setminus D) \cap U)$ that $\widetilde{y}_2 = (\widetilde{y}_2^{(1)}, \widetilde{y}_2^{(2)}) \in \interior{\TconeEuclidean{\widetilde{S}}{\widetilde{y}_1}}$, resulting in $\widetilde{y}_2^{(2)} < 0$. From~\eqref{eq:coord_spray} and \eqref{eq:coord_vlift}, note also that $\widetilde{y}_2^{(2)}$, i.e. the coordinate of $F(x)$ corresponding to the basis vector $\partial/\partial \widetilde{y}_1^{(2)}|_{y}$, is the second component of the coordinate representation of $F(x)$. Using this fact with~\eqref{eq:mobius-tangentCone} implies that~\eqref{eq:mobius-inclusion} holds. Then, as $\diffFunc{\varphi_x}$ is an isomorphism onto its image, applying $\diffFunc{\varphi_x}^{-1}$ to both sides in~\eqref{eq:mobius-inclusion} results in $F(x) \in \Tcone{C}{x}$. 

    As the choice of $x\in \partial (C\cap U)$ was arbitrary,~\ref{VC} holds for the chosen $\xi \in C\setminus D$. Since this choice of $\xi$ is arbitrary,~\ref{VC} holds for each $\xi\in C\setminus D$. Therefore, following Proposition~\ref{prop:existence}, there exists a nontrivial solution to $\hybrid$ from each point in $C\cup D$. Furthermore, as $G(D_2) \subset C\cup  D_1$ and $G(D_1) \subset C \cup D_2$, it follows that $G(D) \subset C\cup D$. Therefore, item c of Proposition~\ref{prop:existence} does not hold. By showing that item b also does not hold, we establish in forthcoming Example~\ref{ex:mobius-stability} that maximal solutions to $\hybrid$ are complete. 
\end{example}

\subsection{Sequential Closedness of the Solution Set of \texorpdfstring{$\hybrid$}{H}}
\label{sec:basicConditionsimplyNominalWellPosedness}

In this section, we study sequences of solutions to $\hybrid$ that are graphically convergent in the sense of Definition~\ref{def:graphicalConvergence}. These results will be crucial to establish the hybrid invariance principle in Section~\ref{sec:invariance}. Much like the Euclidean case~\cite[Thm.~6.8]{goebel_hybrid_2012}, we show that if $\hybrid$ satisfies the geometric hybrid basic conditions, then the graphical limit of a convergent sequence of solutions is also a solution to $\hybrid$. 
First, we specialize the previously defined notion of local eventual precompactness of a sequence of set-valued mappings to a sequence of hybrid arcs.

\begin{definition}[\hspace{-0.3pt}{\cite[Def.~5.24]{goebel_hybrid_2012}}]
    \label{def:locallyEventuallyPrecompctHybridArcs}
    A sequence $\{\phi_i\}_{i=1}^{\infty}$ of hybrid arcs $\phi_i : \dom{\phi_i} \to \M$ is locally eventually precompact if, for each $m > 0$, there exists $i_0 > 0$ and a compact set $K \subset \M$ such that, for each $i > i_0$ and each $(t,j)\in\dom{\phi_i}$ with $t+j < m$, $\phi_i(t,j)\in K$. 
\end{definition}

\begin{theorem}
\label{theorem:basicConditions => nominalWellPosedness}
    Suppose $\hybrid = (C, F, D, G, \M)$ satisfies the geometric hybrid basic conditions in Definition~\ref{ass:hybrid_basic_conditions}. Consider a graphically convergent sequence $\{\phi_i\}_{i=1}^\infty$ of solutions $\phi_i : \dom{\phi_i} \to \M$ to $\hybrid$ that does not escape to the horizon.  Then, 
    \begin{enumerate}[label=(\alph*)]
        \item if the sequence $\{\phi_i\}_{i=1}^\infty$ is locally eventually precompact, then $\phi\coloneqq \graphlim{i\to\infty}{\phi_i}$ is a solution to $\hybrid$;
        \item if the sequence $\{\phi_i\}_{i=1}^\infty$ is not locally eventually precompact, then there exist $\tau > 0$ and a sequence $\{(t_i, j_i)\}_{i=1}^\infty$ of times $(t_i, j_i)\in \dom{\phi_i}$ with $t_i+j_i\nearrow \tau$ such that
        \begin{enumerate}[label=\roman*), leftmargin=*]
            \item the sequence $\{\phi_i(t_i, j_i)\}_{i=1}^\infty$ escapes to the horizon,
            \item the hybrid arc $\phi \coloneqq \left(\graphlim{i\to\infty}\phi_i \right)|_{<\tau}$ is a maximal solution to $\hybrid$ with $\operatorname{length}(\dom{\phi})=\tau$,
            \item for each compact set $K \subset \M$, there exists $\tau' \in (0, \tau)$ such that $K \cap \rge{\phi}|_{>\tau'} = \varnothing$. 
        \end{enumerate}
    \end{enumerate}
\end{theorem}

The proof of Theorem~\ref{theorem:basicConditions => nominalWellPosedness} is presented in Appendix~\ref{app:proofs_nominal_wellposedness}.

\begin{remark}
    The property that the graphical limit of any graphically convergent sequence of solutions to~$\hybrid$ is itself a solution is known as \emph{nominal well-posedness} of~$\hybrid$; see~\cite{jirwankar2025invariance} and~\cite[Ch.~6]{goebel_hybrid_2012}. Recent work~\cite{ochoa2026Topological} studies several topological consequences of this property for the solution set map~$\xi\rightrightarrows\mathcal{S}_{\hybrid}(\xi)$. A related, non-topological treatment in the Euclidean setting appears in~\cite[Ch.~6]{goebel_hybrid_2012}.
\end{remark}

\section{Lyapunov Theorem for \\ Geometric Hybrid Dynamical Systems}
\label{sec:lyapunovTheroemSection}
To obtain a hybrid Lyapunov theorem that certifies uniform global asymptotic stability of a compact set for $\hybrid$, we first define several topological notions of uniform stability, attractivity, and asymptotic stability. Unlike their Euclidean counterparts~\cite[Def.~3.6]{goebel_hybrid_2012}, the proposed notions are independent of any metric on $\M$, and rely solely on the topology of $\M$. 
\vspace{-8pt}
\subsection{Stability and Attractivity Notions on \texorpdfstring{$C^1$}{C1}-Manifolds}
\label{sec:topologicalStability}

We begin by proposing notions of (uniform) stability, attractivity, and asymptotic stability of a compact set $\A \subset \M$ for $\hybrid$. Recall that as $\M$ is a $C^1$-manifold, we may not know the Riemannian metric on it. Hence, the distance between points on this manifold may not be defined. The following definitions are formulated considering this fact.

\begin{definition}[Uniform Global Stability]
    \label{def:UGS}
	Given a geometric hybrid system $\cal H=(C,F,D,G,\M)$, a nonempty, compact set $\A \subset \M$ is said to be
    \begin{enumerate}[leftmargin=18pt]
        \item \emph{uniformly Lyapunov stable} (ULS) for $\hybrid$ if, for each neighborhood $U$ of $\A$, there exists a neighborhood $W$ of $\A$ such that, for {each $\phi\in \maximalSol{\hybrid}{W}$}, $\rge \phi \subset U$;
        \item \emph{uniformly Lagrange stable} (ULaS) for $\hybrid$ if, for each compact neighborhood $W$ of $\A$, there exists a compact neighborhood $U$ of $\A$ such that, {for each $\phi\in \maximalSol{\hybrid}{W}$}, $\rge \phi \subset U$;
        \item \emph{uniformly globally stable} (UGS) for $\cal H$ if it is Lyapunov stable and uniformly Lagrange stable for $\hybrid$.
    \end{enumerate}
\end{definition}
\begin{remark}\label{remark:lyapunovStabilityCompactSets}
    In metric spaces, uniform global stability can equivalently be introduced through a direct $\cal K_\infty$ bound on solutions \cite[Def.~3.6]{goebel_hybrid_2012}, or through growth conditions on the dependence of $\delta$ on $\varepsilon$ in an $\varepsilon$-$\delta$ definition of Lyapunov stability formulated via metric inflations of $\cal A$ \cite[Prop.~2]{Andriano1997-LagrangeStability}. Instead, in Definition~\ref{def:UGS}, we follow the decomposition into uniform Lyapunov and uniform Lagrange stability as in \cite[Thm.~1]{Andriano1997-LagrangeStability} for nonlinear systems. Replacing metric inflations by suitable neighborhoods recasts this decomposition as a separation between topological and bornological (boundedness) content~\cite{Angeli2025_LyapunovLagrange} and admits a metric-independent realization.
\end{remark}

To relate Definition~\ref{def:UGS} to an easily verifiable notion for systems on Riemannian manifolds, we show in  Proposition~\ref{prop:UGS} that, when $\A$ is compact and $\M$ is a complete Riemannian manifold, the above definition is equivalent to the existence of a class-$\mathcal{K}_{\infty}$ function that upper bounds, at each hybrid time, the Riemannian distance from each solution to the $\A$.~Finally, we present notions of attractivity and asymptotic stability.

\begin{definition}[Pre-attractivity]
    \label{def:attractivity}
    A nonempty, closed set $\A\subset \M$ is pre-attractive (pA) for $\hybrid$ if there exists an open neighborhood $U\subset \M$ of $\A$ such that
    \begin{enumerate}
        \item for each $\phi\in \maximalSol{\hybrid}{U}$ with $\operatorname{length}(\dom{\phi}) < \infty$, $\rge{\phi}\subset\M$ is precompact;
        \item for each complete $\phi\in \maximalSol{\hybrid}{U}$ and each open neighborhood $W\subset \M$ of $\A$, there exists $N \geq 0$ such that $(t,j)\in \dom{\phi}$ and $t+j \geq N$ implies that $\phi(t,j)\in W$. 
    \end{enumerate}
\end{definition}

The uniform counterpart of the above definition is provided below when $\A$ is compact. 

\begin{definition}[Uniform global pre-attractivity]
    \label{def:uniformGlobalAtt}
    A non-empty, compact set $\A$ is uniformly globally pre-attractive (UGpA) for $\mathcal{H}$ if, for each compact neighborhoods $U,W\subset\M$ of $\A$, there exists $N \geq 0$ such that for each $\phi\in \maximalSol{\hybrid}{U}$, $(t,j)\in\dom{\phi}$ and $t+j\geq N$ implies $\phi(t,j)\in W$.
\end{definition}

\begin{definition}[Pre-asymptotic stability]
    \label{def:asymptoticStability}
    A nonempty, compact set $\A\subset \M$ is pre-asymptotically stable (pAS) for $\hybrid$ if it is ULS and pA for $\hybrid$.
\end{definition}

\begin{definition}[Uniform global pre-asymptotic stability]
    A nonempty, compact set $\A\subset \M$ is uniformly globally pre-asymptotically stable (UGpAS) for $\hybrid$ if it is UGS and UGpA for $\hybrid$.
\end{definition}

In the definitions above, the prefix ``pre-'' enables maximal solutions to not be complete. If each maximal solution to $\hybrid$ is complete, then the prefix is dropped.

\subsection{Hybrid Lyapunov Theorem on $C^1$-Manifolds}
\label{sec:Lyapunov}

In this section, we present Lyapunov-based sufficient conditions to certify uniform global pre-asymptotic stability of a set $\A$ for the geometric hybrid system $\hybrid$. Inspired by the existing results in the Euclidean case \cite[Thm. 3.19]{HybridFeedbackControl}, we admit Lyapunov function candidates that are merely locally Lipschitz rather than continuously differentiable. This relaxation calls for additional nonsmooth analysis tools on manifolds, which we develop before stating the stability certificates.

We begin by introducing a suitable definition of locally Lipschitz functions. This notion extends the one in~\cite{MOTREANU1982116} and specializes~\cite[Def.~3]{kvalheim2021existence}, and differs from the one in~\cite{jirwankar2025lyapunov}.
\begin{definition}[Locally Lipschitz function] 
    \label{def:Lipschitz}
    Given two $C^1$-manifolds $\M$ and $\N$, a function $f: \M\to \N$ is locally Lipschitz if, for each $x \in \M$, there exist a coordinate chart $(U,\varphi)$ on $\M$ at $x$ and a coordinate chart $(W,\psi)$ on $\N$ at $f(x)$ such that the function $\psi\circ f \circ \varphi^{-1} : \varphi(U\cap f^{-1}(W))  \to \psi(W)$ is locally Lipschitz as a function between Euclidean spaces.\footnote{The preimage of a set $W\subset \N$ under a function $f:\M\to \N$ is defined by $f^{-1}(W)\coloneqq \{x\in\M : f(x)\in W\}$.}
\end{definition}

{That Definition~\ref{def:Lipschitz} is equivalent to the metric-based definition on Riemannian manifolds is shown in~\cite[Lem.~E.6]{jirwankarTAC2026}.} 

\begin{remark}
    For Definition~\ref{def:Lipschitz} to be invariant under the choice of coordinate charts, the transition map~\cite[Ch.~1]{Lee} between coordinate charts on $\M$ (resp., $\N$) should be locally Lipschitz. This property holds since each transition map is a $C^1$-diffeomorphism onto its image.
\end{remark}

Using the geometric version of Rademacher's theorem in Proposition~\ref{prop:rademacher}, a locally Lipschitz function $f : \M \to \N$ may fail to be differentiable at certain points. The following~notion of generalized directional derivative measures the rate of change of $f$ at those points along a specified direction; see also \cite[Eq.~(3.8)]{MOTREANU1982116}.

\begin{definition}[Generalized directional derivative]
    \label{def:generalizedDerivative}
    Let $\M$ be an $n$-dimensional $C^1$-manifold and $f : \M\to\R{}$ be locally Lipschitz. Then, the generalized directional derivative of $f$ at $x\in\M$ in the direction $v\in\T{x}{\M}$ is defined as
    \begin{align}\label{eq:defGeneralizedDerivative}
        f^\circ (x, v) \coloneqq \limsup_{\substack{y \to x\\  t \searrow 0}} \frac{\widetilde{f}(\varphi(y) + t\diffFunc{\varphi_x}(v)) - \widetilde{f}(\varphi(y))}{t}
    \end{align}
    where $(U, \varphi)$ is a chart at $x$ and $\widetilde{f}\coloneqq f \circ\varphi^{-1}$.%
\end{definition}

\begin{remark}
    Following \cite[Lemma~3.3]{MOTREANU1982116}, Definition~\ref{def:generalizedDerivative} is independent of the choice of coordinate chart.
\end{remark}

Equipped with the above nonsmooth analysis tools, we define the following notion of a nonsmooth Lyapunov function candidate $V$ for $\hybrid$. Then, forthcoming Definition~\ref{def:VdotDeltaV} and Lemma~\ref{lemma:bound-VdotDeltaV} together use generalized directional derivatives to bound the rate of change of $V$ along solutions to $\hybrid$, which is crucial to establish the forthcoming hybrid Lyapunov theorem.

\begin{definition}[Lyapunov function candidate]
    \label{def:lyapunovFunctionCandidate}
    Let $\hybrid = (C,F,D,G,\M)$ be a geometric hybrid dynamical system. Given nonempty sets $\U, \A\subset \M$, a function $V: \dom{V} \to \R{}$ defines a Lyapunov function candidate on $\U$ with respect to $\A$ for $\hybrid$ if the following conditions hold:
    \begin{enumerate}
        \item $(\overline{C} \cup D \cup G(D)) \cap \U \subset \dom{V}$;
        \item $\U$ contains an open neighborhood of $\A \cap ({C} \cup D \cup G(D))$;
        \item $V$ is continuous on $\U$ and locally Lipschitz on an open set containing $\overline{C}\cup \U$;
        \item $V$, restricted to $\overline{C \cup D \cup G(D)}$, satisfies $V\in\PD{\A}$.
    \end{enumerate}
\end{definition}

Due to item 3 in the definition above, $V$ is differentiable almost everywhere on $\overline{C}\cup \U$. Then, the rate of change of $V$ along flows of the solutions to $\mathcal{H}$ is defined using the generalized directional derivative.

\begin{definition}[$\dot V$ and $\Delta V$]
    \label{def:VdotDeltaV}
    Given a geometric hybrid dynamical system $\mathcal{H} = (C, F, D, G, \M)$, sets $\U, \A \subset \M$, and a function $V : \dom{V}\to \R{}$ that defines a Lyapunov function candidate on $\U$ with respect to $\A$ for $\mathcal{H}$,
    \begin{itemize}
        \item the change of $V$ along flows is given by
        \begin{align}
            \label{eq:def_Vdot}
            \dot{V}(x) \coloneqq \sup_{v\in F(x)\cap \Tcone{C}{x}} V^\circ(x, v) \quad \forall x\in C\cap \U.
        \end{align}
        \item the change of $V$ at jumps is given by
        \begin{align}
            \label{eq:def_DeltaV}
            \Delta V(x) \coloneqq \sup_{g\in G(x)} V(g) - V(x) \quad \forall x\in D\cap\U.
        \end{align}
    \end{itemize}
\end{definition}

{\color{mygreen}\relax{}
{\color{mygreen}\begin{remark}    
    For any $(x, v)\in \T{}{\M}$, by using the upper one-sided contingent derivative of $V$ at $x$ along $v$, denoted by $D^+V(x, v)$, instead of the generalized directional derivative $V^\circ(x, v)$ in Definition~\ref{def:VdotDeltaV}, we can relax the local Lipschitz regularity of $V$ in Definition~\ref{def:lyapunovFunctionCandidate} to simply continuity. The latter approach, however, requires that $D^+V(x, v)$ is finite for each $(x, v)\in \T{}{\M}$, whereas finiteness of $V^\circ(x, v)$ is guaranteed by Lipschitz $V$; see~\cite[Prop.~2.1.2]{Clarke1998}. The use of upper one-sided contingent derivatives in the context of differential inclusions in $\R{n}$ is explored in~\cite[Sec.~9.4]{goebel_setvaluedAnalysis}. 
\end{remark}}
}

The following result proves that $\dot{V}$ and $\Delta V$, as defined in~\eqref{eq:def_Vdot} and~\eqref{eq:def_DeltaV}, respectively, upper bound the rate of change of the Lyapunov function candidate $V$ along flows and during jumps, respectively. %

\begin{lemma}
\label{lemma:bound-VdotDeltaV}
    For each solution $\phi$ to $\hybrid$ and each $(T,J)\in\dom{\phi}$, let $0\leq t_0\leq t_1\leq \ldots\leq t_{J+1} = T$ satisfy $ \dom{\phi}\cap ([0, T]\times \{0, 1, \ldots, J\}) = \cup_{j=0}^{J} [t_j, t_{j+1}]\times  \{j\}$. Then, 
        
    \begin{enumerate}[label=\roman*)]
        \item for each $j\in\{0,1,\ldots, J\}$ and almost all $t\in [t_j, t_{j+1}]$, 
        \begin{align*}
            \frac{dV}{dt}\brackets{\phi(t,j)} \leq \dot{V}(\phi(t,j));
        \end{align*}
        
        \item for each $j\in \{0,1,\ldots J\}$ with $(t_{j+1}, j+1)\in\dom{\phi}$,
        \vspace{-3pt}
        \begin{align*}
            V(\phi(t_{j+1}, j+1)) - V(\phi(t_{j+1}, j)) \leq \Delta V(\phi(t_{j+1}, j)).
        \end{align*}
    \end{enumerate}
\end{lemma}

\begin{proof}
    Pick a solution $\phi$ and $(T,J)\in\dom{\phi}$ that results in $\{t_j\}_{j=0}^{J+1}$ according to the statement of the lemma. Pick $j\in \mathbb{N}$ and let $I_{\phi}^j\coloneqq [t_j, t_{j+1}]$. Due to Proposition~\ref{prop:rademacher}, let $S(I_{\phi}^j)$ and $S_{\M}(V)$ denote, resp., the measure zero sets in $\interior{I_{\phi}^j}$ and $\M$ where $t \mapsto V\circ \phi(t,j)$ and $V$ are not differentiable. 
    For each $t\in {I^j_\phi}\setminus \{\partial I^j_\phi \cup S(I^j_\phi)\}$, let $v\in F(\phi(t, j))$ be such that $\frac{d\phi}{dt}(t,j) = v$.
    Then, for each $t\in I_{\phi}^j \setminus \{\partial I^j_\phi \cup S(I^j_\phi)\}$,
    
    \vspace{-0.2cm}
    \begin{align*}%
        \frac{dV}{dt} (\phi(t,j)) & = \diffFunc{V}_{\phi(t,j)}(v) \leq \sup_{f\in F(\phi(t,j))\cap \Tcone{C}{\phi(t,j)}} \diffFunc{V_{\phi(t,j)}}(f).
    \end{align*}
    Next, for each $x\in \M\setminus S_\M(V)$ and each $f\in \T{x}{\M}$,
    \begin{align*}
        \diffFunc{V_x}(f) &= \lim_{\substack{t\to 0}} \frac{V\circ \varphi^{-1}(\varphi(x) + t\diffFunc{\varphi_x}(f)) - V\circ \varphi^{-1}(\varphi(x))}{t}\\
        &\leq \limsup_{\substack{y \to x\\  t \searrow 0}} \frac{\widetilde{V}(\varphi(y) + t\diffFunc{\varphi_x}(f)) - \widetilde{V}(\varphi(y))}{t} \leq V^\circ(x, f)
    \end{align*}
    where $(U, \varphi)$ is a chart at $x\in\M$ and $\widetilde{V}\coloneqq V\circ\varphi^{-1}$. Using the above inequalities with the fact that $S_{\M}(V)$ has zero Lebesgue measure proves item i. 
    The proof of item ii follows by noting that $V(\phi(t,j+1))-V(\phi(t,j)) \leq \sup_{g\in G(\phi(t,j))} V(g)-V(\phi(t,j))$ for each jump time $(t,j)\in\dom{\phi}$.
\end{proof}

Following \cite[Thm. 3.18]{goebel_hybrid_2012} and using the results in \cite{PaulACC2025}, we obtain the following result. Note that it does not require the hybrid basic conditions.

\begin{theorem}[Hybrid Lyapunov Theorem]
    \label{theorem:hybridLyapunovTheorem-manifolds}
    Consider a nonempty, compact set $\A\subset \M$ and a function $V : \M \to \R{}_{\geq 0}$ that defines a Lyapunov function candidate on $\M$ with respect to $\A$ for $\hybrid = (C, F, D, G, \M)$. The set $\A$ is uniformly globally pre-asymptotically stable for $\hybrid$ if $V$ is proper %
    and one of the following conditions hold:
    \begin{enumerate}[label=(\alph*)]
        \item \emph{Strict decrease during flows and jumps}: there exist lower semicontinuous functions $\rho_\mathrm{C}, \rho_\mathrm{D}\in \PD{\A}$ such that
        \begin{align}
            \dot{V}(x) &\leq -\rho_\mathrm{C}(x)  \qquad \forall x\in C,
            \label{eq:V_dot_theorem}
                \\
            \Delta V(x) &\leq -\rho_\mathrm{D}(x)  \qquad \forall x\in D.
            \label{eq:DeltaV_theorem}
        \end{align} 

        \item \emph{Strict decrease during flows and no increase at jumps}: there exists a lower semicontinuous function $\rho_{\mathrm{C}}\in \PD{\A}$ such that \eqref{eq:V_dot_theorem} holds, \eqref{eq:DeltaV_theorem} holds with $\rho_{\mathrm{D}} \equiv 0$, and, for each compact neighborhood $R$ of $\A$, there exist $\gamma \in \mathcal{K}_{\infty}$ and $N \geq 0$ such that, for each solution $\phi \in \setofSol{\hybrid}{R\setminus \A}$, $(t,j)\in \dom{\phi}$ implies $t \geq \gamma(t+j) - N$.

        \item \emph{Strict decrease at jumps and no increase during flows}: there exists a lower semicontinuous function $\rho_{\mathrm{D}}\in \PD{\A}$ such that \eqref{eq:DeltaV_theorem} holds, \eqref{eq:V_dot_theorem} holds with $\rho_{\mathrm{C}} \equiv 0$, and, for each compact neighborhood $R$ of $\A$, there exists $\gamma \in \mathcal{K}_{\infty}$ and $N \geq 0$ such that, for each solution $\phi\in \setofSol{\hybrid}{R\setminus \A}$, $(t,j)\in \dom{\phi}$ implies $j \geq \gamma(t+j) - N$.

        \item \emph{Increase balanced by decrease}: there exist constants $\lambda_c \in \R{}$ and $\lambda_d \in \R{}$ such that
            \begin{align}
            \label{eq:V_increaseDecrease}
            \begin{array}{lll}
                \phantom{\Delta}\dot{V}(x) \leq \lambda_c V(x) & &\forall x\in C,\\
                \Delta V(x) \leq \brackets{\exp(\lambda_d)-1} V(x) & & \forall x\in D,
            \end{array}
        \end{align}
        and there exist $M, \gamma>0$ such that, for each solution $\phi$ to $\hybrid$ and each $(t,j)\in\dom{\phi}$, $\lambda_c t + \lambda_d j \leq M - \gamma(t+j)$.

        \item \emph{Strict decrease at jumps and bounded time of flow}: there exist a lower semicontinuous function $\rho_D\in \PD{\A}$ and $\lambda \in \R{}$ such that \eqref{eq:DeltaV_theorem} holds, $\dot{V}(x) \leq \lambda V(x)$ for each $x\in C$, and, for each compact neighborhood $R$ of $\A$, there exists~$T \geq 0$ such that each $\phi\in \setofSol{\hybrid}{R\setminus \A}$ satisfies $\sup_{t}\dom{\phi} \leq T$.

        \item \emph{Strict decrease during flows and finite number of jumps}: there exist a lower semicontinuous function $\rho_{\mathrm{C}}\in \PD{\A}$ and $\lambda\in \mathcal{K}_\infty$ such that \eqref{eq:V_dot_theorem} holds, $V(\chi) \leq \lambda(V(x))$ for each $x\in D$ and each $\chi\in G(x)$, and, for each compact neighborhood $R$ of $\A$, there exists $J > 0$ such that each solution $\phi\in \setofSol{\hybrid}{R\setminus \A}$ satisfies $\sup_j \dom{\phi}\leq J$. 
        
    \end{enumerate}
\end{theorem}

\begin{proof}
    {} 
    Pick $\phi \in \maximalSol{\hybrid}{\overline{C}\cup D}$ and  $(t,j)\in\dom{\phi}$. Let $\{t_i\}_{i=0}^{j+1}$ be a nondecreasing sequence obtained from hybrid time domain structure such that $t_0 = 0$ and $t_{j+1}=t$. {}
    {\color{mygreen}\relax{}{\color{mygreen}
    For each $i\in\{0,1,\dots,j\}$ and almost all $s\in[t_i, t_{i+1}]$, $\phi(s, i)\in C$. Then, for each $i\in\{0,1,\dots,j\}$ and almost all $s\in[t_i, t_{i+1}]$, Lemma~\ref{lemma:bound-VdotDeltaV} and \eqref{eq:V_dot_theorem} imply that
    \(
        \frac{d}{ds}V(\phi(s, i)) \leq -\rho_{\mathrm{C}}(\phi(s,i)). 
    \)
    Integrating both sides, we have, for each $i\in\{0,1,\dots,j\}$, that
    \begin{align*}
        V(\phi(t_{i+1},i)) - V(\phi(t_i, i)) \leq - \int_{t_i}^{t_{i+1}} \rho_\mathrm{C}({\phi(s,i)}) ds.
    \end{align*}
    Similarly, since $\phi(t_i,i-1)\in D$ for each $i\in\{1,2,\dots, j\}$, the following holds for each $i\in\{1,2,\ldots, j\}$:
    \begin{align*}
        V(\phi(t_i, i)) - V(\phi(t_i, i-1)) \leq -\rho_\mathrm{D}({\phi(t_i, i-1)}). 
    \end{align*}
    Using the above two inequalities and summing over $i$'s,}}
    \begin{align}
        \label{eq:ineq_manifoldCase}
        V(\phi(t,j)) + \sum_{i=0}^{j} \int_{t_i}^{t_{i+1}}\rho_{\mathrm{C}}({\phi(s,i)})ds &+ \sum_{i=1}^{j} \rho_{\mathrm{D}}({\phi(t_i, i-1)}) \nonumber\\
        &\leq V(\phi(0,0)). 
    \end{align}
    Let $X \coloneqq \overline{C} \cup D \cup G(D)$ and $\widetilde{\A} \coloneqq \A \cap X$ for brevity. Using~\eqref{eq:ineq_manifoldCase} with item 4 of Definition~\ref{def:lyapunovFunctionCandidate} and noting that $\rge{\phi} \subset X$ by definition of a solution, we have $0 \leq V(\phi(t,j))\leq V(\phi(0,0))$ for all $(t,j)\in \dom{\phi}$. For each $c \geq 0$, let $\sublevelSet{V}(c)\coloneqq \{x\in \M : V(x) \leq c\}$ denote the $c$-sublevel set of $V$, which is compact as $V$ is proper.  Then, for each $(t,j)\in \dom{\phi}$, $\sublevelSet{V}(V(\phi(t,j))) \cap X \subset \sublevelSet{V}(V(\phi(0,0))) \cap X$, leading to
    \begin{align}\label{eq:FI_sublevelSets}
        \rge{\phi} \subset L_V(V(\phi(0,0))) \cap X \subset L_V(V(\phi(0,0))). 
    \end{align}

    To show UGS of $\A$ for $\hybrid$, by local compactness of $\A$, pick any compact neighborhood $W$ of $\A$. Since $\dom{V} = \M$ by construction, $V$ is defined on $\overline{X} \setminus \interior{W}$, which is closed and disjoint from $\widetilde{\A}$. 
    Together with items~3 and 4 of Definition~\ref{def:lyapunovFunctionCandidate}, a proof by contradiction yields $\inf_{x\in \overline{X} \setminus \interior{W}} V(x) >0$.
    
    Then, pick $u > 0$ such that $u < \inf_{x \in \overline{X}\setminus \interior{W}} V(x)$ and let $U \coloneqq L_V(u)$. Note that $U \cap X \subset W \cap X \subset W$. Similarly, noting that $\max_{x\in W}V(x)$ is finite as $V$ is continuous and $W$ is compact, pick $r > \max_{x\in W}V(x)$ and let $R \coloneqq L_V(r)$. By construction, $W \subset R$. Figure~\ref{fig:stability} illustrates the construction of the sets above. Pick any solution $\phi \in \maximalSol{\hybrid}{U}$ and note that $\phi\in \maximalSol{\hybrid}{U \cap X}$. Following~\eqref{eq:FI_sublevelSets}, $\rge{\phi} \subset U \subset W$. Therefore, $\A$ is ULS for $\hybrid$. Similarly, pick any $\phi\in \maximalSol{\hybrid}{W} = \maximalSol{\hybrid}{W \cap X}$. As $W \subset R$, $\phi\in \maximalSol{\hybrid}{R}$ and note from~\eqref{eq:FI_sublevelSets} that $\rge{\phi} \subset R$. Therefore, $\A$ is ULaS for $\hybrid$. Consequently, $\A$ is UGS for $\hybrid$.

        \begin{figure}
        \vspace{0.3cm}
        \centering
        {\includegraphics[width=0.6\columnwidth, trim = 0.2cm 13.5cm 17cm 0.2cm]{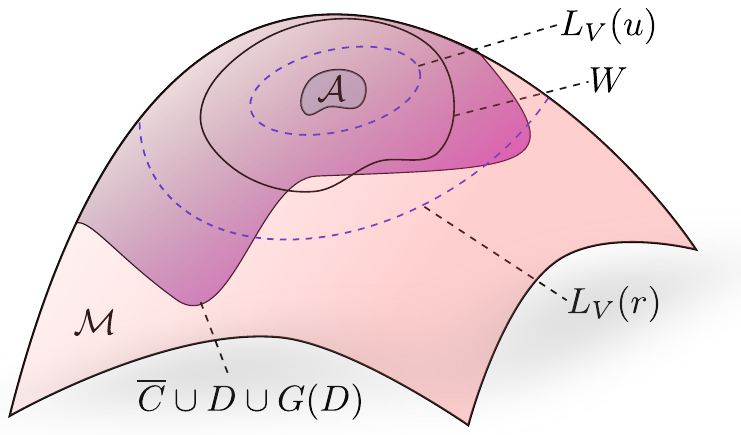}}
        \caption{Set construction in the UGS proof of Theorem~\ref{theorem:hybridLyapunovTheorem-manifolds}.\vspace{-16pt}}
        \label{fig:stability}
    \end{figure}

    To prove uniform global pre-attractivity of $\A$ for $\hybrid$, we leverage positive definiteness and properness of $V$ to equivalently replace the compact neighborhoods of $\A$ in Definition~\ref{def:uniformGlobalAtt} to nested sublevel sets of $V$. Consequently, we show that for each $u', r' > 0$, there exists $N' \geq 0$ such that each maximal solution $\phi$ to $\hybrid$ with $\phi(0,0)\in \sublevelSet{V}(r')$ satisfies $\phi(t,j)\in \sublevelSet{V}(u')$ for all $(t,j)\in \dom{\phi}$ with $t+j\geq N'$. 

    To do so, define $\rho(x)\coloneqq \min\{\rho_\mathrm{C}(x), \rho_{\mathrm{D}}(x)\}$ for all $x\in \M$. Since $\rho_\mathrm{C}, \rho_\mathrm{D}$ are lower semicontinuous functions that are positive definite with respect to $\A$, $\rho$ is also lower semicontinuous and positive definite with respect to $\A$. Pick arbitrary $u', r' > 0$ and define $U'\coloneqq \sublevelSet{V}(u')$ and $R'\coloneqq \sublevelSet{V}(r')$. Without loss of generality, assume $u'\leq r'$ so that $U' \subset R'$. Let $m\coloneqq \min_{x\in \sublevelSet{V}([u', r'])}\rho(x)$. Note that $m > 0$. Then, for each $\phi\in \maximalSol{\hybrid}{R'}$, using~\eqref{eq:FI_sublevelSets} with \eqref{eq:ineq_manifoldCase} results in
    \begin{align*}%
        V(\phi(t,j)) \leq V(\phi(0,0)) - (t+j)m
    \end{align*}
    for each $(t,j)\in\dom{\phi}$ with $\phi(t,j)\in R'\setminus U'$. Since $\phi(0,0)\in R'$, it follows that $V(\phi(t,j))\leq r' - (t+j)m$ for all such $(t,j)$. Clearly, for $(t,j)$ such that $t+j = N'\coloneqq (r' - u')/m$, $V(\phi(t,j))\leq u'$. Using~\eqref{eq:FI_sublevelSets}, $\phi(t,j)\in U'$ for all $(t,j)\in\dom{\phi}$ such that $t+j\geq N'$. Furthermore, if $\operatorname{length} \dom{\phi} < N'$, $\rge{\phi}$ is contained in a compact set due to~\eqref{eq:FI_sublevelSets}. Therefore, $\A$ is UGpA for $\hybrid$ as Definition~\ref{def:uniformGlobalAtt} holds.
    
    {\color{mygreen}
    \medskip
    \noindent
    (Proof of item b:)
    The proof of UGS of $\A$ for $\hybrid$ follows exactly as shown in the proof of item a. For the proof of uniform global pre-attractive, consider any compact neighborhood $U$ of $\A$ such that $R$ is a compact neighborhood of $U$. Then, consider $r', u' > 0$ such that $\sublevelSet{V}(u') \subset U \subset R \subset \sublevelSet{V}(r')$. Define $m \coloneqq \min_{x\in \sublevelSet{V}([u', r'])}\rho_\mathrm{C}(x)$. Then, for each solution $\phi$ to $\hybrid$ with $\phi(0,0)\in R\setminus U$ so that $V(\phi(0,0))\leq r'$, we have
    \begin{align}
    \label{eq:V_decrease_item_b}
        V(\phi(t,j)) \leq V(\phi(0,0)) - tm \qquad \forall (t,j)\in\dom{\phi}.
    \end{align}
    Therefore, each sublevel set of $V$ is forward pre-invariant for $\hybrid$. Following item b, pick $\gamma\in\mathcal{K}_\infty$ and $N \geq 0$ such that $t \geq \gamma(t+j)-N$ holds for all $(t,j)\in \dom{\phi}$. Using this inequality with~\eqref{eq:V_decrease_item_b} yields
    \begin{align*}
        V(\phi(t,j)) \leq V(\phi(0,0)) - m \brackets{\gamma(t+j) - N }
    \end{align*}
    for all $(t,j)\in\dom{\phi}$. Then, for $(t,j)\in\dom{\phi}$ such that $t + j \geq T' \coloneqq \gamma^{-1}\brackets{\frac{r' - u'}{m} + N}$, we have $V(\phi(t,j))\leq u'$, i.e., $\phi(t,j)\in U$. If $\operatorname{length}\dom{\phi} < T'$, the existence of a compact neighborhood that contains $\phi(t,j)$ for all $(t,j)\in \dom{\phi}$ is guaranteed as $V$ is proper and its sublevel sets are forward pre-invariant for $\hybrid$. This proves uniform global pre-attractivity and completes the proof of item b.

    \medskip
    \noindent
    (Proof of item c:) The proof follows exactly like the proof of item b, with the roles of flows and jumps interchanged. It is therefore omitted.

    \medskip
    \noindent
    (Proof of item d.)
    Let $\phi$ be a maximal solution to $\hybrid$. Using the bounds in~\eqref{eq:V_increaseDecrease} with the assumption that $\lambda_c t + \lambda_d j \leq M - \gamma(t+j)$ for each $(t,j)\in\dom{\phi}$, we have that
    \begin{align}
        \label{eq:item-d-exp-bound}
        V(\phi(t,j)) \leq \exp(M-\gamma(t+j))V(\phi(0,0))    
    \end{align}
    holds for each $(t,j)\in\dom{\phi}$. 
    
    We first prove UGS. Pick any compact neighborhood $W$ of $\A$. Let
    \(
        a_W \coloneqq
        \inf_{x\in \overline{X}\setminus \interior W}V(x)>0.
    \)
    Pick $u>0$ such that $e^M u<a_W$, and define
    \(
        U\coloneqq \sublevelSet{V}(u).
    \)
    If $\phi\in \maximalSol{\hybrid}{U}$, then by \eqref{eq:item-d-exp-bound},
    \[
        V(\phi(t,j)) \leq e^M V(\phi(0,0)) \leq e^M u < a_W
    \]
    for all $(t,j)\in\dom{\phi}$. Then, as $L_V(a_W) \subset W$ by construction, $\rge{\phi}\subset W$. Therefore, $\A$ is ULS for $\hybrid$.
    
    Next, define $r_W\coloneqq \max_{x\in W}V(x)$ and $R\coloneqq \sublevelSet{V}(e^M r_W)$. Then, $R$ is a compact neighborhood of $\A$. For each $\phi\in\maximalSol{\hybrid}{W}$, \eqref{eq:item-d-exp-bound} yields
    \begin{align*}
        V(\phi(t,j))
        \leq e^M V(\phi(0,0))
        \leq e^M r_W ,
        &&
        \forall (t,j)\in\dom{\phi}.
    \end{align*}
    As a result, $\rge{\phi}\subset R$. Therefore, $\A$ is ULaS, and consequently, UGS, for $\hybrid$.
    
    We now prove uniform global pre-attractivity. Let $U$ and $R$ be compact
    neighborhoods of $\A$ with $U\subset R$. Choose $0 < u < r$ such that $\sublevelSet{V}(u)\cap X\subset U$ and $R\cap X\subset \sublevelSet{V}(r)$. If $\phi\in\maximalSol{\hybrid}{R}$, then
    \eqref{eq:item-d-exp-bound} yields, for each $(t,j)\in\dom{\phi}$, that $V(\phi(t,j))\leq e^{M-\gamma(t+j)}r$. Consequently, $V(\phi(t,j))\leq u$ whenever $(t,j)\in \dom{\phi}$ satisfy $t+j\geq T \coloneqq \frac{M+\ln r-\ln u}{\gamma}$. For all such $(t,j)\in \dom{\phi}$, 
    $\phi(t,j)\in U$. If $\operatorname{length}\dom{\phi}<T$, then
    \eqref{eq:item-d-exp-bound} implies
    $\rge{\phi}\subset \sublevelSet{V}(e^M r)$. Therefore
    $\A$ is UGpA for $\hybrid$.

    \medskip
    \noindent
    (Proof of item e:) Pick a maximal solution $\psi$ to $\hybrid$. For each $(t,j)\in\dom{\psi}$, define the sequence $\{t_i\}_{i=0}^{j+1}$ as in the proof of item a. Due to the existence of $\rho_{\mathrm{D}}\in\PD{\A}$ satisfying \eqref{eq:DeltaV_theorem}, the inequality in item e yields
    \begin{align}
        V(\psi(t,j)) &\leq \exp(\lambda t) V(\psi(0,0)) - \sum_{i=1}^{j} \rho_{\mathrm{D}}({\psi(t_i, i-1)}) \label{eq:Vdot_exp_0}
    \end{align}
    for all $(t,j)\in\dom{\psi}$. Now, pick a compact neighborhood $W$ of $\A$, and let $T\geq 0$ be as in item e such that $\sup_{t}\dom{\phi}\leq T$ for each $\phi\in \widehat{\mathcal{S}}_{\hybrid}(W\setminus \A)$. Consequently, for each $\phi\in \widehat{\mathcal{S}}_{\hybrid}(W\setminus \A)$,~\eqref{eq:Vdot_exp_0} yields that, for each $(t,j)\in\dom{\phi}$,
    \begin{align}
        V(\phi(t,j)) &\leq \exp(\lambda T) V(\phi(0,0)) - \sum_{i=1}^{j} \rho_{\mathrm{D}}({\psi(t_i, i-1)}) \label{eq:Vdot_exp_rhoD}
        \\
        &\leq \exp(\lambda T) V(\phi(0,0)) \label{eq:Vdot_exp}.
    \end{align}
    Then, the proof of UGS follows similarly to item d above. 
    
    To establish UGpA of $\A$ for $\hybrid$, recall the compact set $W$ above, and pick arbitrary $r' > 0$ such that $R' \coloneqq L_V(r') \subset W$. Pick also arbitrary $u'\in  (0, r')$ so that $U' \coloneqq L_V(u') \subset R'$. Pick any solution $\phi\in \maximalSol{\hybrid}{R'}$, and note that $\phi\in \maximalSol{\hybrid}{W}$. Let $m\coloneqq \min_{x\in \sublevelSet{V}([u', r'])}\rho_\mathrm{D}(x)$. As a result, \eqref{eq:Vdot_exp_rhoD} results in
    \begin{align*}
        V(\phi(t,j)) \leq \exp(\lambda T) V(\phi(0,0)) - jm 
    \end{align*}
    for all $(t,j)\in\dom{\phi}$ such that $\phi(t,j)\in R'\setminus U'$. Since $x(0,0)\in R'$, it follows from the above inequality that $V(x(t,j))\leq \exp(\lambda T) r' - jm$ for all such $(t,j)$. With $T' \coloneqq T + \brackets{\exp(\lambda T)r' - u'}/{m}$, UGpA of $\A$ for $\hybrid$ is established using similar arguments to the proof of item a.
    
    \medskip
    \noindent
    (Proof of item f:) The proof follows exactly like the proof of item e, with the roles of flows and jumps interchanged. It is therefore omitted.
    }
\end{proof}

Theorem~\ref{theorem:hybridLyapunovTheorem-manifolds} is stated for Lyapunov function candidates that are locally Lipschitz and proper, without reference to a distance function on $\M$. When a Riemannian structure is available, both properties of the Lyapunov function candidate suitably relate to the Riemannian distance; {}{\color{mygreen}\relax{}see Lemma~\ref{lemma:LocLipRiemannian} and Proposition~\ref{prop:Kinfty_lowerBoundOnV}.} Next, we illustrate Theorem \ref{theorem:hybridLyapunovTheorem-manifolds} using the following example.

\begin{example}[Quaternion stabilization, revisited]\label{ex:quaternion-stability}
    Recall $\hybrid$ from Example~\ref{ex:quaternion}. Denote the identity~element by $\mathbf{1}\coloneqq \begin{pmatrix}
        1 & \mathbf{0}_{1\times 3}^\top
    \end{pmatrix}^\top$ and define 
    \(
        \A \coloneqq \{(q,h)\in \M : q = h\mathbf{1}\}.
    \)
    Following~\cite{mayhew2011quaternion}, \(V(x) = 2(1-h\eta)\) for each $x\in \M$ defines a Lyapunov function candidate on $\M$ with respect to $\A$~for $\hybrid$.
    {}{\color{mygreen}\relax{}{\color{mygreen}Then, $V \in \PD{\A}$ is proper as it is a map from a compact set to~$\R{}$~\cite[Prop.~A.53]{Lee}.}}  During flows, {\color{mygreen}\relax{}{\color{mygreen}using $\dot{h} = 0$, we have}}
    \begin{align}
        \label{eq:Vdot_quaternion}
        \diffFunc{V_x}(F(x)) = -2h\dot{\eta} = - \varepsilon^\top K_{\varepsilon} \varepsilon < 0 \quad \forall x\in C\setminus\A,
    \end{align}
    where the last equality above follows from
    \begin{align*}
        \dot{\eta} &= [1 \;\;\; \mathbf{0}_{1\times 3}] \brackets{\frac{1}{2} q\otimes v(-hK_{\varepsilon}\varepsilon)} = \frac{1}{2} h \varepsilon^\top K_\varepsilon \varepsilon.
    \end{align*}
    For each $x\in \M$, let $\rho_{\mathrm{C}}(x)\coloneqq 0$ if $x\in \A$ and $\rho_{\mathrm{C}}(x)\coloneqq \varepsilon^\top K_{\varepsilon} \varepsilon$ otherwise. Note that $\rho_{\mathrm{C}}\in \PD{\A}$ as $K_\varepsilon$ is positive definite, and that $\rho_{\mathrm{C}}$ is lower semicontinuous as $\A$ is closed. Together with~\eqref{eq:Vdot_quaternion},
    \eqref{eq:V_dot_theorem} holds. 
    During jumps, we have that
    \(
        \Delta V(x) = V(G(x)) - V(x) 
        = 4h\eta \leq -4\delta
    \)
    for each $x\in D$.
    We note that $\A\cap D = \varnothing$. Then, setting $\rho_{\mathrm{D}}(x)\coloneqq 0$ for each $x\in \A$, and $\rho_{\mathrm{D}}(x)\coloneqq 4\delta$ for each $x\in \M\setminus \A$, 
    we note that $\rho_{\mathrm{D}}$ is lower semicontinuous since $\A$ is closed, and \eqref{eq:DeltaV_theorem} holds. It follows from Theorem~\ref{theorem:hybridLyapunovTheorem-manifolds} and the completeness of maximal solutions to $\hybrid$ from Example~\ref{ex:quaternion-solutions} that $\A$ is UGAS for $\mathcal{H}$.
\end{example}

\section{Hybrid Invariance Principle}
\label{sec:invariance}
In practice, the decrease conditions of Theorem \ref{theorem:hybridLyapunovTheorem-manifolds} may only be satisfied with $\dot{V} \leq 0$ or $\Delta V \leq 0$. Examples of such functions are abundant in applications, where $V$ commonly denotes the total energy of the system; see~\cite[Prop.~4.66]{BulloLewis}.
The hybrid invariance principle \cite[Ch. 8]{goebel_hybrid_2012} resolves this limitation when $\M = \R{n}$, but such a result for hybrid systems on manifolds is absent from the literature. To this end, in this section, we present a hybrid invariance principle for geometric hybrid systems. First, we establish properties of $\omega$-limit sets of complete and precompact solutions to geometric hybrid dynamical systems satisfying the geometric hybrid basic conditions. For an equivalent characterization when $\M = \R{n}$, see \cite[Ch.6]{goebel_hybrid_2012},~\cite{Sanfelice_Invariance}. 

\begin{definition}[Weak invariance]
\label{def:weakInvariance}
    Given a geometric hybrid system $\hybrid$, a nonempty set $K \subset \M$ is said to be
    \begin{enumerate}[leftmargin=15pt]
        \item \emph{weakly forward invariant} for $\hybrid$ if, for each $\xi \in K$, there exists a complete solution $\phi\in\maximalSol{\hybrid}{\xi}$ with $\rge{\phi} \subset K$;
        \item \emph{weakly backward invariant} for $\hybrid$ if, for each $\xi \in K$ and each $N > 0$, there exists $\phi\in\maximalSol{\hybrid}{K}$ such that, for some hybrid time $(t^\star, j^\star)\in\dom{\phi}$ with $t^\star + j^\star \geq N$, $\phi$ satisfies $\phi(t^\star, j^\star) = \xi$ and $\rge{\phi}|_{\leq t^\star + j^*}\subset K$;
        \item \emph{weakly invariant} for $\hybrid$ if it is both weakly forward invariant and weakly backward invariant. 
    \end{enumerate}
\end{definition}

\begin{definition}[$\omega$-limit set of a hybrid arc]
    \label{def:omegaLimitSet}
    The $ \omega$-limit set of a hybrid arc $\phi\!:\! \dom{\phi} \to \M$, denoted by $\Omega(\phi)$, is the set of all points $x\in \M$ for which there exists a sequence $\{(t_i, j_i)\}_{i = 1}^\infty\subset \dom{\phi}$ with $\lim_{i\to \infty} t_i + j_i = \infty$ such that $\lim_{i\to\infty}\phi(t_i, j_i)= x$. %
\end{definition}

\begin{proposition}[Properties of $\Omega(\phi)$]
    \label{prop:omegaLimitSets}
    Let $\hybrid$ be a geometric hybrid dynamical system satisfying the geometric hybrid basic conditions in Definition~\ref{ass:hybrid_basic_conditions}, and $\phi$ be a complete and precompact solution to $\hybrid$. Then, $\Omega(\phi)$ is nonempty, compact, weakly invariant, and, for each open neighborhood $\mathcal{U}$ of $\Omega(\phi)$, there exists $N > 0$ such that $\phi(t, j)\in \mathcal{U}$ for all $(t,j)\in\dom{\phi}$ such that $t + j \geq N$.
\end{proposition}

\begin{proof}
    That $\Omega(\phi)$ is nonempty follows from compactness of $\overline{\rge{\phi}}$ and extraction of a convergent sequence $\{\phi(t_i, j_i)\}_{i=1}^\infty$, with $(t_i, j_i)\in\dom{\phi}$ and $t_i + j_i \to \infty$ such that $\lim_{i\to\infty} \phi(t_i, j_i)\in\Omega(\phi)$. That $\Omega(\phi)$ is closed follows from Definition~\ref{def:omegaLimitSet}. Then, as $\Omega(\phi)\subset \overline{\rge{\phi}}$ and $\overline{\rge{\phi}}$ is compact, $\Omega(\phi)$ is compact due to~\cite[Prop.~A.45]{Lee}. To show weak forward and backward invariance of $\Omega(\phi)$, pick $\xi\in\Omega(\phi)$ and a sequence $\{(t_i, j_i)\}_{i=1}^\infty \subset \dom{\phi}$ such that $t_i + j_i \to \infty$ and $\phi(t_i, j_i) \to \xi$. Pick any $\tau > 0$ and a sequence  $\{(t_i', j_i')\}_{i=1}^\infty \subset \dom{\phi}$ such that $t_i +  j_i + \tau-1 \leq t_i' + j_i' \leq t_i + j_i + \tau + 1$ for each sufficiently large $i$. We consider the sequence of truncated hybrid arcs $\phi_i(t,j) \coloneqq \phi(t + t_i', j+j_i')$ and note that $\{\phi_i\}_{i=1}^\infty$ is graphically convergent, locally eventually precompact, and $\lim_{i\to\infty}\phi_i(0,0) = \xi$. Then, using Theorem~\ref{theorem:basicConditions => nominalWellPosedness}(a), $\phi^{\star}\coloneqq \graphlim{i\to\infty}\phi_i$ is a complete solution to $\hybrid$. For each $i$, as $\tau \leq t_i - t_i' + j_i - j_i' \leq \tau+1$, we assume, by passing to a subsequence without relabeling, that the sequence $\{t_i - t_i', j_i - j_i'\}_{i=1}^\infty$ converges to $(t^\star, j^\star)$. Then, $\lim_{i\to\infty}\phi_i(t_i - t_i', j_i - j_i') = \phi^\star(t_i^\star, j_i^\star) = \xi$. Additionally, using Lemma~\ref{lemma:domainAndRange}, $\rge{\phi^\star} \subset \lim_{i\to\infty}\rge{\phi_i} = \Omega(\phi)$. Therefore, $\Omega(\phi)$ is weakly backward invariant. Weak forward invariance follows by considering a solution $\phi'(t,j)\coloneqq \phi^{\star}(t + t^\star, j+j^\star)$ and noting that $\phi'(0,0) = \xi$ and $\rge{\phi'}\subset \Omega(\phi)$. 

    Finally, that $\phi$ converges to $\U$ follows by constructing a sequence $\{(t_i, j_i)\}_{i=1}^\infty$ such that $\phi(t_i, j_i) \to x\in\M$ and $x\notin \Omega(\phi)$, which is a contradiction as $x\in \Omega(\phi)$ by  Def.~\ref{def:omegaLimitSet}. 
\end{proof}

To obtain the invariance principle, given a function $V : \mathrm{dom}\,V \to \R{}$ satisfying items~1 and~3 of Definition~\ref{def:lyapunovFunctionCandidate}, define
\begin{equation}
    \begin{aligned}
        V^{-1}(r) &\coloneqq \{x\in \M: V(x) = r\}\quad \forall r\in \R{},
        \\
        \dot{V}^{-1}(0) &\coloneqq \{x\in C : \dot{V}(x) = 0\},
        \\
        \Delta V^{-1}(0) &\coloneqq \{x \in D : \Delta V (x) = 0\}.
    \end{aligned}
\end{equation}
First, we use the above definitions to present a key result that characterizes the value of $V$ in the $\omega$-limit set of a hybrid arc. This result, which will lead to the hybrid invariance principle, extends~\cite[Lemma~4.1]{Sanfelice_Invariance} to the geometric setting.

\begin{lemma}
    \label{lemma:V_on_omegaLimitSet}
    Let $\phi : \dom{\phi} \to \M$ denote a hybrid arc, and suppose that a continuous function $V : \dom{V} \to \R{}$ with $\dom{V}\subset \M$ satisfies $V(\phi(t',j')) \leq V(\phi(t, j))$ for each $(t,j), (t', j')\in \dom{\phi}$ satisfying $t + j \leq t' + j'$. Then, there exists $r\in \R{}$ such that $\Omega(\phi) \subset V^{-1}(r)$. 
\end{lemma}

\begin{theorem}[Hybrid invariance principle]
\label{theorem:invariance}
    Consider a geometric hybrid system $\hybrid = (C, F, D, G, \M)$ that satisfies the geometric hybrid basic conditions in Definition~\ref{ass:hybrid_basic_conditions}, a nonempty set $\U\subset \M$, and a function $V : \dom{V}\to \R{}$ satisfying items 1 and 3 of Definition~\ref{def:lyapunovFunctionCandidate} such that the following holds:
    \begin{align}\label{eq:VdotDeltaV_invariance}
        \begin{array}{rc}
            \dot{V}(x) \leq 0 & \quad \forall x \in C \cap \U,\\
            \Delta V(x) \leq 0 & \quad \forall x\in D \cap \U. 
        \end{array}
    \end{align}
    Let $\phi^\star$ be a complete and precompact solution to $\hybrid$ with $\overline{\rge{\phi^\star}} \subset \U$. Then, for some $r\in V(\U \cap (C \cup D \cup G(D)))$, the set $\Omega(\phi^\star)$ is the largest weakly invariant set in
    \begin{flalign}
    \label{eq:invariance}
        \hspace{-3pt}V^{-1}(r) \cap \U \!\cap\! \left[\dot{V}^{-1}(0) \cup \left( \Delta V^{-1}(0) \!\cap\! G(\Delta V^{-1}(0)) \right) \right].
    \end{flalign}
\end{theorem}

\medskip

{\color{mygreen}\relax{}\begin{proof}
    Following Proposition~\ref{prop:omegaLimitSets}, $\Omega(\phi^\star)$ is nonempty, compact, and weakly invariant, and $\phi^{\star}$ approaches $\Omega(\phi^\star)$ ; in particular, 
    \begin{align*}
        \limsup_{\substack{t + j \to \infty \\ (t,j)\in\dom{\phi^\star}}} \phi^\star(t, j) \subset \Omega(\phi^{\star}). 
    \end{align*}
    \noindent
    From Lemma~\ref{lemma:V_on_omegaLimitSet}, there exists $r\in V(\U\cap \brackets{C\cap D\cap G(D)})$ such that $V(\Omega(\phi^\star)) = r$. Consequently, $V$ remains constant along each solution $\phi$ to $\hybrid$ with $\rge{\phi}\subset \Omega(\phi^\star)$. Pick any such solution $\phi$ and pick any $(\underline{t}, \underline{j}), (t,j), (\bar{t}, \bar{j})\in \dom{\phi}$ satisfying {$\underline{t}+\underline{j} \leq t+j \leq \bar{t}+\bar{j}$}. Define $J \coloneqq \bar{j}-\underline{j}$ and let $\underline{t} = t_0 \leq t_1 \leq \dots \leq t_{J+1} = \bar{t}$ satisfy
    \(
        \dom{\phi} \cap \brackets{[\underline{t},\bar{t}]\times \{\underline{j}, \underline{j}+1,\dots, \bar{j}\}} = \bigcup_{i=0}^{J} \left([t_i, t_{i+1}]\times \{\underline{j} + i\} \right) 
    \).
    Using Definition~\ref{def:VdotDeltaV}, we obtain
    \begin{align}\label{eq:zeroChangeV_invariance}
        \sum_{i=0}^{J} \int_{t_i}^{t_{i+1}}\dot{V}({\phi(s,\underline{j}+i)})ds  + \sum_{i=1}^{J} \Delta V({\phi(t_i, \underline{j} + i-1)}) =  0. 
    \end{align}
    
    Pick any $\xi\in\Omega(\phi^\star)$. Since $\Omega(\phi^\star)$ is weakly forward invariant, suppose $\phi$ is a complete solution from $\xi$ satisfying $\rge{\phi}\subset \Omega(\phi^\star)$. If $(0,1)\in\dom{\phi}$, then setting $(\underline{t}, \underline{j})=(0,0)$ and $(\bar{t}, \bar{j}) = (0,1)$ and using \eqref{eq:zeroChangeV_invariance}, we obtain $\Delta V(\phi(0, 0)) = \Delta V(\xi) =  0$, resulting in
    \begin{align}\label{eq:OmegaLimitSetSubsetOfDeltaV}
        \quad \xi \in \Delta V^{-1}(0) \implies \Omega(\phi^\star)\subset \Delta V^{-1}(0).
    \end{align}
    If there exists $T>0$ such that $(T,0)\in\dom{\phi}$, then setting $(\underline{t}, \underline{j}) = (0,0)$, $(\bar{t}, \bar{j}) = (0, T)$, and using \eqref{eq:zeroChangeV_invariance} results in $\int_{0}^T \dot{V}(\phi(s,0)) ds = 0$. Due to \eqref{eq:VdotDeltaV_invariance}, it follows that $\dot{V}(\phi(t,0)) = 0$ for almost all $t\in [0,T]$. Therefore, $t\mapsto V(\phi(t,0))$ must be piecewise constant on $[0,T]$. However, since $V$ is locally Lipschitz on $\overline{C}\cap \U$, it follows that $V(\phi(t,0))$ is constant for all $t\in [0,T]$, resulting in 
    \begin{align}\label{eq:OmegaLimitSetSubsetOfVdot}
    \xi \in \dot{V}^{-1}(0) \quad \implies \quad \Omega(\phi^\star)\subset \dot{V}^{-1}(0).
    \end{align}

    Next, we use the weak backward invariance property of $\Omega(\phi^\star)$ to obtain similar inclusions as above. In particular, suppose that item 2 of Definition~\ref{def:weakInvariance} holds with the choice of $\phi$ above; in particular, there exist $\xi^\star\in \Omega(\phi^\star)$ and $(t^\star, j^\star)\in\dom{\phi^\star}$ satisfying $\phi(0,0) = \xi^\star$, $t^\star + j^\star \geq 1$, $\phi(t^\star, j^\star)=\xi$, and $\phi(t,j)\in \Omega(\phi^\star)$ for all $(t,j) \in \dom{\phi}$ with $t+j\leq t^\star + j^\star$. If there exists $(t^\star, j^\star-1)\in\dom{\phi}$, then setting $(\underline{t}, \underline{j})=(t^\star, j^\star-1)$, $(\bar{t}, \bar{j})=(t^\star, j^\star)$, and using \eqref{eq:zeroChangeV_invariance} results in
    \begin{align*}
        & \Delta V(\phi(t^\star, j^\star-1)) = 0 \implies \phi(t^\star, j^\star-1)\in \Delta V^{-1}(0).
    \end{align*}
    Consequently,
    \begin{align}\label{eq:OmegaLimitSetSubsetOfDeltaV_backward}
        \xi \in G(\Delta V^{-1}(0)) \implies \Omega(\phi^\star) \subset G(\Delta V^{-1}(0)). 
    \end{align}
    Similarly, if there exists $T>0$ such that $(t^\star - T, j^\star)\in\dom{\phi}$, we obtain $\int_{t^\star-T}^{t^\star}  \dot{V}(\phi(s, j^\star)) ds = 0$. Due to \eqref{eq:VdotDeltaV_invariance}, it follows that $\dot{V}(\phi(t,j^\star)) = 0$ almost everywhere on $[t^\star - T, t^\star]$. Then, similar arguments as above follow, resulting in $\Omega(\phi^\star)\subset \dot{V}^{-1}(0)$. Now, combining this result with \eqref{eq:OmegaLimitSetSubsetOfVdot}, \eqref{eq:OmegaLimitSetSubsetOfDeltaV}, and \eqref{eq:OmegaLimitSetSubsetOfDeltaV_backward}, and using the facts that $\overline{\rge{\phi^\star}} \subset \U$ and $V(\Omega(\phi^\star)) = r$, we obtain that $\Omega(\phi^\star)$ is contained in~\eqref{eq:invariance}. 
    
    Finally, that $\Omega(\phi^\star)$ is the \emph{largest} weakly invariant set in~\eqref{eq:invariance} follows by a simple contradiction argument, as otherwise, there exists a point in~\eqref{eq:invariance}, but not in $\Omega(\phi^\star)$, to which $\phi^\star$ converges. This is a contradiction by Definition~\ref{def:omegaLimitSet}. 
\end{proof}}

\begin{example}[Billiard on the M\"{o}bius band, revisited]
    \label{ex:mobius-stability}
    Recall the hybrid system $\hybrid$ from Example~\ref{ex:mobius-hybridDynamics}. Recall also from Example~\ref{ex:MobiusDef} that, for each $z = ({z_1}, {z_2})\in \rectangle_S$, $y_1 \coloneqq [z]_{\mobius}\in \mobius$. Let $\tilde{\A}\subset\mobius$ in Example~\ref{ex:mobius-hybridDynamics} be defined as $\tilde{\A} \coloneqq \{[{z_1}, {z_2}]_{\mobius}\in \mobius : z_1 \in \{0,1\}, z_2 = 0\}$, and note that $\tilde{\A}$ is a point in the interior of the constraint set $S\subset \mobius$ as defined in Example~\ref{ex:mobius-hybridDynamics}. 
    Recall the state $x = (y, q) \in \M$ of the hybrid system $\hybrid$, where $y = (y_1, y_2)\in \T{}{\mobius}$ and $q \in Q$. We certify that the set
    \[
    \A \coloneqq \{x \in \M : y_1 \in \tilde{\A}, y_2 = 0 \in \T{y_1}{\mobius}, q \in Q\}
    \]
    is asymptotically stable for $\hybrid = (C, F, D, G, \M)$ in~\eqref{eq:mobius-inclusion}.

    Before defining the Lyapunov function candidate, we note that the ``centerline'' of the M\"{o}bius strip and the constraint set $S$ is the unit circle $\mathbb{S}^1 \coloneqq [0,1]/\sim_{\scriptstyle\mathbb{S}^1}$, where the equivalence relation  $\sim_{\scriptstyle\mathbb{S}^1}$ defines the equivalence class $[\cdot]_{\scriptstyle\mathbb{S}^1}$ such that $[z_1]_{\scriptstyle\mathbb{S}^1}=\{ z_1\}$ for each $z_1\in (0,1)$, and $[0]_{\scriptstyle\mathbb{S}^1} = [1]_{\scriptstyle\mathbb{S}^1} = \{0,1\}$. 
        
    Letting $\U \coloneqq \M$, we define the Lyapunov function candidate $V:\M \to \R{}_{\geq 0}$ on $\U$ with respect to $\A$ for $\hybrid$~as 
    \vspace{-5pt}
    \begin{align}
    \label{eq:mobius-lyapunov}
        V(x) \coloneqq V_q(y_1) + \frac{1}{2}g_{\mobius}^{y_1}(y_2, y_2) \quad \forall x\in \M.
    \vspace{-5pt}
    \end{align}
    Here, $g_{\mobius}^{y_1}$ denotes the evaluation of the Riemannian metric $g_{\mobius}$ on $\mobius$ at $y_1$. For each $q\in Q$, the function $V_q : \mobius \to \R{}_{\geq0}$ is defined so that 
    $
        V_q(y_1) = V_q([z]_\mobius) \coloneqq U(\Gamma([z_1]_{\mathbb{S}^1},q)) + z_2^2/2
    $
    for each $z=(z_1, z_2)\in \rectangle$,
    where $U([z_1]_{\mathbb{S}^1})\coloneqq (1-\cos(2\pi z_1))/2$ for each $z_1\in [0,1]$, $\Gamma([z_1]_{\mathbb{S}^1}, q)\coloneqq [\left(z_1 + q\arcsin(U([z_1])/2) / \pi \right)\mathrm{mod}(1)]_{\mathbb{S}^1}$ for each $(z_1,q)\in [0,1]\times Q$, and $\mathrm{mod}$ denotes the modulo operator. Note that $V_q\in \PD{\widetilde{\A}}$ {is proper} and continuously differentiable for each $q\in Q$. This construction of the map~$\Gamma$ is motivated from synergistic control~\cite{Berkane2017synergistic}. Then, $V \in \PD{\A}$ and $V$ is smooth.

    At jumps from $D_1$, we have that, for each $x\in D_1$, 
    \vspace{-5pt}
    \begin{align*}
        & \Delta V(x) =  V_q(g_1(y_1)) - V_q(y_1) - \frac{1}{2}g_{\mobius}^{y_1}(y_2, y_2)
        \\
        & + \frac{1}{2} \sup_{\lambda\in [0, \lambda_{\max}]}g_{\mobius}^{g_1(y_1)}(P^{\gamma}_{0\mapsto 1}(y_2^{\parallel} - \lambda y_2^{\perp}), P^{\gamma}_{0\mapsto 1}(y_2^{\parallel} - \lambda y_2^{\perp})).
    \end{align*}
    Note that $V_q(g_1(y_1)) - V_q(y_1) = ((z_2 - \operatorname{sign}(z_2)\varepsilon)^2 - z_2^2)/2 = \varepsilon(\varepsilon-2)/2 < 0$, where the last equality~follows as $x\in D_1$ implies $|z_2| = 1$. For brevity, define $\delta_\varepsilon = \varepsilon(2 - \varepsilon)/2 > 0$. 
    Additionally, as parallel transport preserves the Riemannian inner product~\cite[Prop.~5.5(e)]{Lee_Riemannian}, we have that, for each $x\in~D_1$, $\Delta V(x) = -\delta_{\varepsilon} + \frac{1}{2} \sup_{\lambda\in [0, \lambda_{\max}]}\left (g_{\mobius}^{y_1}(y_2^{\parallel} - \lambda y_2^{\perp}, y_2^{\parallel} - \lambda y_2^{\perp}) - g_{\mobius}^{y_1}(y_2, y_2)\right ).$
    As $y_2 = y_2^{\parallel} + y_2^{\perp}$ and $g_{\mobius}^{y_1}(y_2^\parallel, y_2^\perp) = 0$ due to orthogonality, %
    \begin{align*}
        \Delta V(x) &= -\delta_{\varepsilon} + \frac{1}{2}\sup_{\lambda\in [0, \lambda_{\max}]}\left (\lambda^2 - 1\right )g_{\mobius}^{y_1}(y_2^{\perp}, y_2^{\perp})
        \\
        &= -\delta_{\varepsilon} + \frac{1}{2}\left (\lambda_{\max}^2 - 1\right )g_{\mobius}^{y_1}(y_2^{\perp}, y_2^{\perp})  \leq -\delta_{\varepsilon},
    \end{align*}
    for each $x\in D_1$, where the last inequality is obtained using $\lambda_{\max} < 1$ and $g_{\mobius}^{y_1}(y_2^{\perp}, y_2^{\perp}) \geq 0$.
    Alternatively, at jumps from $D_2$, $\Delta V(x) \leq -\delta$ for each $x\in D_2$. Therefore, for each $x \in D=D_1\cup D_2$, $\Delta V(x) \leq -\min\{\delta, \delta_{\varepsilon}\} < 0$.
    
    During flows, for each $x\in C$, we have
    \begin{align*}
        \dot{V}(x) &= \diffFunc{{V}_{y}}\left(\mathfrak{S}(y) + \vlift{y}{-\grad{V_q}{y_1} - b y_2}\right)
        \\
        & = -b g_{\mobius}^{y_1}(y_2, y_2) \leq 0,
    \end{align*}
    where the last inequality above follows from~\cite[Prop.~4.66]{BulloLewis} and by relating the first-order continuous-time dynamics in $\hybrid$ to second-order dynamics as in~\cite[\S4.6.3]{BulloLewis}. We present the details of the computation in {}{\color{mygreen}\relax{}Appendix~\ref{app:mobius-Vdot}}.

    Therefore,~\eqref{eq:VdotDeltaV_invariance} holds. Then, pick any maximal $\phi \in \mathcal{S}_{\hybrid}$. We show that $\phi$ is complete. Using Example~\ref{ex:mobius-solutions}, it suffices to show that item b in Prop.~\ref{prop:existence} does not hold. Indeed, suppose by contradiction that $\phi$ satisfies item b. Noting that the position $y_1$ and the logic variable $q$ belong to compact sets $S$ and $Q$, respectively, it must follow that the velocity~$y_2$~goes ``unbounded.'' Therefore, $\lim_{t+j \to \infty} V(\phi(t,j)) = \infty$. However, as $V(\phi(t,j)) \leq V(\phi(0,0))$ for each $(t,j)\in \dom{\phi}$ due to the bound on $\dot V$ and $\Delta V$, and since $V(\phi(0,0))$ is finite, we have a contradiction. Therefore, each maximal $\phi \in \mathcal{S}_{\hybrid}$ is complete.    

    {\color{mygreen}\relax{}{\color{mygreen}Next, we show that $\phi$ is also precompact. Let $v_0 \coloneqq \sqrt{V(\phi(0,0))}$. For each $(t,j)\in\dom{\phi}$, as $V(\phi(t,j))\leq v_0^2$, using~\eqref{eq:mobius-lyapunov} yields $\phi(t,j)\in K$, where
    \[
        K \coloneqq \left\{(y,q) \in \M : y_1 \in S, q \in Q, y_2 \in \mathbb{B}_{{v_0}}(0, y_1)\right\},
    \]
    and, for $r\geq 0$, $\mathbb{B}_{r}(0, y_1) \coloneqq \{v\in \T{y_1}{\mobius} : g_{\mobius}^{y_1}(v, v) \leq r^2\}$ denotes a compact ball of radius $r$ in the tangent space $\T{y_1}{\mobius}$. As $S$ and $Q$ are compact, and $\mathbb{B}_{v_0}(0, y_1)$ is compact for each $y_1\in S$, it follows that $K$ is compact. Then, $\rge\phi \subset K$ and, therefore, $\phi$ is precompact. }}
    
    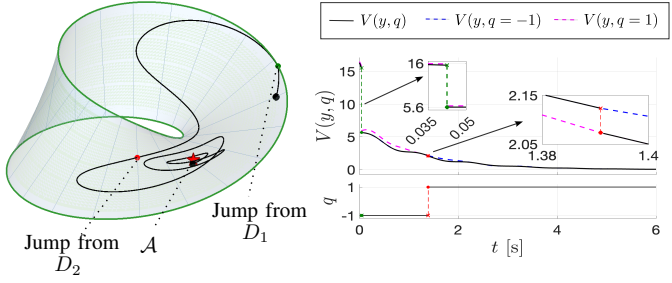
\begin{figure}
    \centering
    \vspace{0.2cm}
        \begin{tikzpicture}
    
            \node[anchor=south west, inner sep=0] (img) 
                at (0,0) {{\includegraphics[width=\linewidth, trim = 0cm 11.88cm 6.8cm 0cm, clip]{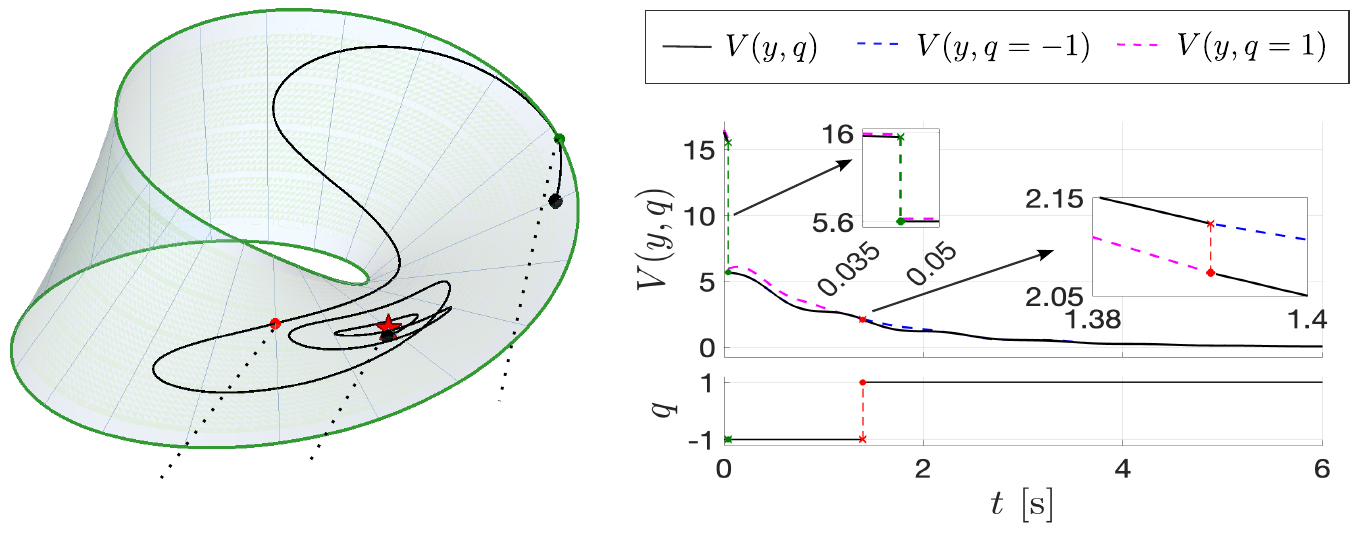}}};
            
            \begin{scope}[x={(img.south east)},y={(img.north west)}]
                \node at (0.38,0.2) {\scalebox{0.8}{ Jump from}};

                \node at (0.38,0.12) {\scalebox{0.8}{$D_1$}};

                \node at (0.1,0.08) {\scalebox{0.8}{ Jump from}};
                \node at (0.1,0.0) {\scalebox{0.8}{$D_2$}};

                \node at (0.22,0.08) {\scalebox{0.8}{$\A$}};

            \end{scope}
        \end{tikzpicture}
    \caption{\emph{Left.} A solution to~$\hybrid$ from Example~\ref{ex:mobius-hybridDynamics} that converges to $\A$. Jumps from the sets $D_1$ and $D_2$ are marked. \emph{Right.} $V(x)$ and $q$ vs time. A jump from the set $D_1$ occurs at $0.04$s. The second jump (from $D_2$) at $1.39$s toggles $q$ from $-1$ to $1$. \vspace{-0.4cm}
    }
    \label{fig:mobius-stability}
\end{figure}

    {} Then, for some $r \in V(C\cup D\cup G(D))$ and using Theorem~\ref{theorem:invariance}, $\phi$ converges to the largest weakly invariant set in~\eqref{eq:invariance}. As $\dot{V}^{-1}(0) = \{x \in C :  y_1 \in S, y_2 = 0\}$ and $\Delta V^{-1}(0) = \varnothing$,~\eqref{eq:invariance} reduces to $X \coloneqq V^{-1}(r) \cap \dot{V}^{-1}(0)$. Now, to obtain the largest weakly invariant set in $X$, pick $x\in X$ and note that $y_2 = 0$. Then, using~\eqref{eq:coord_spray}, $\mathfrak{S}(y) = 0$,~causing $F(x) = \begin{psmallmatrix}
        \vlift{y}{-\grad{V_q}{y_1}}
            \\
            0
    \end{psmallmatrix}$.
    If $\grad{V_q}{y_1} \neq 0$, then, using~\eqref{eq:coord_vlift}, each solution to $\hybrid$ from $x$ leaves $X$ as $y_2$ becomes nonzero along the solution. Therefore, the largest invariant set in $X$ is the set $\{x \in X : \grad{V_q}{y_1} = 0\} = \A$. Therefore, $\phi$ approaches $\A$. As $\phi \in \mathcal{S}_{\hybrid}$ was arbitrary, $\A$ is (globally) attractive for $\hybrid$. 
    Stability of $\A$ for $\hybrid$ follows as sublevel sets of $V$, when restricted to $C\cup D$, are compact and forward invariant for $\hybrid$. In particular, for each $v' > 0$, there exists $u' \in (0, v')$ such that each $\phi\in \maximalSol{\hybrid}{L_V(u')}$ satisfies $\rge{\phi}\in L_V(u') \subset L_V(v')$. Consequently, $\A$ is (globally) asymptotically stable for $\hybrid$. Figure~\ref{fig:mobius-stability} illustrates convergence of a sample solution to $\A$ in the presence of collisions at the boundary $\partial S$ and jumps due to synergistic switching. 
\end{example}

\section{Conclusion}
\label{sec:conclusion}
We have proposed a framework for the formulation and analysis of hybrid dynamical systems whose state evolves on a $C^1$-manifold, which we call geometric hybrid dynamical systems. The framework separates the minimal geometric structure required to analyze such systems from common assumptions that may carry information not naturally associated with the dynamics, and in particular it is independent of structure that is extrinsic to the underlying manifold, such as a Riemannian metric or an embedding into Euclidean space.

Future work includes: (i) formulating robustness properties of geometric hybrid systems with respect to structural perturbations, along with minimal assumptions on their data, in the spirit of~\cite[Ch.~6]{goebel_hybrid_2012}; and (ii) extending the framework to geometric hybrid dynamical systems with inputs, which would enable control-synthesis methods and input-to-state stability certificates with respect to measurable disturbances.

\section*{References\vspace{-14pt}}

\makeatletter
\renewcommand\@biblabel[1]{{\footnotesize[#1]}}
\let\oldbibliography\thebibliography
\renewcommand\thebibliography[1]{
  \oldbibliography{#1}
  \footnotesize
}
\makeatother

\bibliographystyle{ieeetr}
\bibliography{references}

\appendices

\inappendixtrue
\renewcommand{\thesection}{\Alph{section}}
\renewcommand{\thesubsection}{\thesection.\arabic{subsection}}

\makeatletter
\def\thesectiondis{\Alph{section}}
\def\thesubsectiondis{\arabic{subsection}.}
\makeatother

\section{Proof of Theorem~\ref{theorem:basicConditions => nominalWellPosedness}}
\label{app:proofs_nominal_wellposedness}

The proof of Theorem~\ref{theorem:basicConditions => nominalWellPosedness} follows using the next steps.
\begin{enumerate}[label=\arabic*), leftmargin=15pt]
    \item In Proposition~\ref{prop:limit_is_solution_inclusion}, we extend~\cite[Lemma 5.27]{goebel_hybrid_2012} for the {constrained differential inclusion (CDI)} described by $\hybrid' = (C, F, \varnothing, G, \M)$, where the jump set is empty, and show that the graphical limit of a locally eventually precompact and graphically convergent sequence of solutions with the \emph{same time domain} is also a {solution to the CDI. }
    \item In Proposition~\ref{prop:nominalWellPosedness_inclusion_localEventualBoundedness}, we show that the above result also holds for a sequence of solutions to $\hybrid'$ with different time domains. The extension to non-locally eventually precompact sequences of solutions is shown in~{} {\color{mygreen}\relax{}Proposition~\ref{prop:nominalWellPosedness_inclusion_nonlocalEventualBoundedness}.}
    \item Finally, using above results and also accounting for jumps in the solutions to $\hybrid$, we obtain Theorem~\ref{theorem:basicConditions => nominalWellPosedness}.
\end{enumerate}

\begin{proposition}
    \label{prop:limit_is_solution_inclusion}
    Suppose that $\hybrid' = (C, F, \varnothing, G, \M)$\footnote{The system $\hybrid'$ is a continuous-time system as the jump set $D$ is empty.} satisfies Definition~\ref{ass:hybrid_basic_conditions}. Consider a graphically convergent and locally eventually precompact sequence $\{\phi_{i}\}_{i=1}^\infty$ of solutions $\phi_{i}$~to $\hybrid'$, where $\dom{\phi_i} = I \subset \R{}_{\geq 0}$ is a connected, compact interval with $\interior{I}\neq\varnothing$. Then, $\phi \coloneqq \graphlim{i\to\infty}\phi_i$ is a solution to $\hybrid'$ with $\dom{\phi} = I$.
\end{proposition}

\begin{proof}
    The proof follows in three steps.

    \noindent\emph{Step 1. $\phi(t)\in C$ for all $t\in I$:} Let $t\in I$. Since $I$ is closed, there exists  a sequence $\{t_i\}_{i=1}^\infty \subset I$ such that $\lim_{i\to\infty} t_i=t$.
    Consider the sequence $\{\phi_i(t_i)\}_{i=1}^\infty$. As $\{\phi_i\}_{i=1}^\infty$ is graphically convergent and locally eventually precompact, we have $\lim_{i\to \infty}\phi_i(t_i) = \phi(t)$.
    Since each $\phi_i(t_i)\in C$, and $C$ is closed by assumption, it follows that $\phi(t)=\lim_{i\to\infty}\phi_i(t_i)\in C$. Since $t\in I$ was arbitrarily, the conclusion holds for all $t\in I$. 
    
    \medskip
    \noindent\emph{{Step 2.} $\phi$ is locally absolutely continuous on $I$:} %
    Using Lemma~\ref{lemma:domainAndRange}, $\dom{\phi} = I$ and $\rge{\phi} = \lim_{i\to\infty}\rge{\phi_i}$. In~particular, this implies that $\phi(I)$ is closed, as it is the limit of a sequence of sets, which is always closed. Additionally, by the definition of local eventual precompactness, there exists a compact set $K\subset \M$ such that $\phi(I)=\lim_{i\to\infty}\phi_i(I)\subset K$. This implies that $\phi(I)$ is compact as it is a closed subset of a compact set \cite[Proposition~A.45]{Lee}.
    
    Next, as $F$ is locally precompact relative to $C$, {\color{mygreen}\relax{}it follows from Lemma~\ref{lemma:LB_equivalence} that, }for each $x\in \phi(I)\subset C$, there exists a chart $(U_x, \varphi)$ at $x$ and a constant $k_x > 0$ such that $\left(\diffFunc{\varphi_{y} \circ F}\right)(y) \subset k_x \mathbb{B}$ for each $y\in U_x$. Pick any compact connected component of $\phi^{-1}(U_x) \subset I$ and denote it by $I_x$. Note that, as $\phi$ is continuous and $U_x$ is open relative to $\M$, $\phi^{-1}(U_x)$ is open relative to $I$. Then, since $I_x$ is closed relative to $\phi^{-1}(U_x)$, it follows that $I_x\subset \interior{\phi^{-1}(U_x)}$. Consequently, since $\phi=\graphlim{i\to\infty}\phi_i$, there exists $i_0 > 0$ such that $\phi_i(I_x)\subset U_x$ for each $i > i_0$. As each $\phi_i$ is a solution to the differential inclusion $\dot{x}\in F(x) \;\; x\in C$, we have that 
    \(
        \frac{d(\varphi \circ \phi_i)}{dt}(t) \in F(\phi_i(t)) 
    \)
    for almost all $t\in I$. Then, using $\left(\diffFunc{\varphi_y \circ F}\right)(y) \subset k_x \mathbb{B}$ for each $y\in U_x$, we obtain 
    $\frac{d(\varphi \circ \phi_i)}{dt}(t) \in k_x \mathbb{B}$ for almost all $t\in I_x$.
    Pick any interval $[a,b]\in I_x$ with $a < b$. Then, the above inclusion yields
    \(
        |(\varphi_x\circ \phi_i)(b) - (\varphi_x\circ \phi_i)(a)|\leq k_x |b-a| 
    \)
    for all $i\geq i_0$. 
    
    Now, to show that $\varphi_n \circ \phi$ is absolutely continuous on $[a,b]$, choose arbitrary $\varepsilon > 0$ and let $\delta = \varepsilon/k_x$. Pick any countable collection of disjoint nonempty sets $[a_k, b_k]\subset [a,b]$ satisfying $\sum_{k}(b_k-a_k) \leq \delta$. Then, since $\phi$ is the graphical limit, pick sequences $\{a_k^i\}_{i=1}^\infty$ and $\{b_k^i\}_{i=1}^\infty$ that converge to $a_k$ and $b_k$, respectively, such that $\lim_{i\to\infty}\phi_i(a_k^i) = \phi(a_k)$ and $\lim_{i\to\infty}\phi_i(b_k^i) = \phi(b_k)$, and for each $i$, the countable collection of intervals $[a_k^i, b_k^i]$ is disjoint. The existence of such countable disjoint collection of intervals for each $i$ follows as the collection $\{[a_k, b_k]\}_k$ is disjoint. Then, for each $i > i_0$,
    \(
        \sum_{k}\left\vert \brackets{\varphi\circ\phi_i}(b_k^i) - \brackets{\varphi\circ\phi_i}(a_k^i)\right\vert \leq k_x \sum_{k}\left\lvert b_k^i - a_k^i\right\rvert.
    \)
    Taking limit as $i\to \infty$, the above inequality results in
    \begin{align*}
        \sum_{k}\left\vert \varphi(\phi(b_k)) - (\varphi(\phi(a_k))\right\vert \leq k_x \sum_{k} |b_k - a_k| &\leq k_x\delta = \varepsilon.
    \end{align*}
    Therefore, $\varphi\circ\phi$ is absolutely continuous on $[a,b]$. Since the choice of $[a,b]\subset I_x$ and the choice of $I_x\subset \varphi^{-1}(U_x)$ are arbitrary, it follows that $\varphi\circ\phi$ is locally absolutely continuous on each compact and connected component of $\phi^{-1}(U_x)$. Furthermore, since each connected component of $\phi^{-1}(U_x)$ can be written as a countable union of its connected, compact components, it follows that $\varphi\circ\phi$ is locally absolutely continuous on each connected component of $\varphi^{-1}(U_x)$. As $x\in \phi(I)$ is arbitrary and $\varphi$ is a $C^1$-diffeomorphism, {\color{mygreen}\relax{}it follows from Lemma~\ref{lemma:locAbsCont_coord} that}  $\phi$ is locally absolutely continuous on $I$.

    {}
    {\color{mygreen}\relax{}\medskip
    \noindent \emph{Step 3. $\phi$ satisfies $\frac{d\phi}{dt}(t) \in F(\phi(t))$ for almost all $t\in I$:} {\color{mygreen}Pick any $x \in \phi(I)$, let $(U_x, \varphi_x)$ be a coordinate chart coming from local precompactness of $F$, and let $[a,b]\subset I$ be any nonempty interval such that $\phi([a,b])\subset U_x$. Following the proof of Proposition~\ref{prop:existence}, consider the coordinate representation of the constrained differential inclusion $\hybrid'=(C, F, \varnothing, G, \M)$ as
    \begin{align}\label{eq:localFlowDynamics1}
        \dot{y} \in \widetilde{F}(y) \qquad  y\in \varphi_x(C\cap U_x),
    \end{align}
    where $y = \varphi_x(x)$ and $\widetilde{F}$ as in~\eqref{eq:localFlowDynamics} with the chart $(V, \varphi)$ therein replaced by $(U_x, \varphi_x)$. Similarly, we also obtain that $\widetilde{F}$ is convex valued and upper semicontinuous relative to $\varphi_x(C\cap U_x)$. Since each $\phi_i$ is a solution to $\hybrid'$, each $\varphi_x\circ \phi_i$, when restricted to $[a, b]$, is a solution to~\eqref{eq:localFlowDynamics1} such that $\varphi_x \circ\phi = \graphlim{i\to\infty}\varphi_x\circ \phi_i$.  Then, for each large enough $i$, $\varphi_x\circ \phi_i(a)$ is contained in a compact set. Consequently, \cite[Thm.~3.1.7]{clarke1990optimization} yields that 
    \(
        \frac{d (\varphi_x \circ \phi)}{dt}(t) = \diffFunc{(\varphi_x)_{\phi(t)}}(\dot{\phi}(t)) \in \widetilde{F}((\varphi_x\circ\phi)(t))
    \)
    for almost all $t\in [a,b]$. Applying the linear isomorphism $\brackets{\diffFunc{(\varphi_x)_{\phi(t)}}}^{-1}$ to both sides, the inclusion is preserved, and we obtain $\frac{d\phi}{dt}(t) \in F(\phi(t))$ for almost all $t\in [a,b]$. As $x\in \phi(I)$ and $[a,b]\subset \phi^{-1}(I)$ are arbitrary, $\phi \in \mathcal{S}_{\hybrid'}$.}}
\end{proof}

\begin{proposition}
\label{prop:nominalWellPosedness_inclusion_localEventualBoundedness}
    Suppose that $\hybrid' = (C, F, \varnothing, G, \M)$ satisfies the geometric hybrid basic conditions in Definition~\ref{ass:hybrid_basic_conditions}. Consider a graphically convergent and locally eventually precompact sequence $\{\phi_i\}_{i=1}^\infty$ of solutions $\phi_i : \dom{\phi_i}\to \M$ to $\hybrid'$. 
    Then, $\phi=\graphlim{i\to\infty}\phi_i$ is a solution to $\hybrid'$. 
\end{proposition}

\begin{proof}
    As $\dom{\phi_i}$ is connected for each $i$ and $\dom{\phi} = \lim_{i\to\infty} \dom{\phi_i} $ by Lem.~\ref{lemma:domainAndRange}, $\dom{\phi}$ is connected and closed. Pick any compact interval $[a,b]\subset \interior{\dom{\phi}}$. Then, $[a,b]\subset \interior{\lim_{i\to\infty}\dom{\phi_i}}$. By Def.~\ref{def:setConvergence}, there exists $i_0 > 0$ such that $[a,b]\subset\interior{\dom{\phi_i}}$ for all $i > i_0$. Consider~the sequence $\{\phi_i\big|_{[a,b]}\}_{i=1}^\infty$. Using graphical convergence and Lem.~\ref{lemma:domainAndRange}, $\phi\big|_{[a,b]} = \graphlim{i\to\infty}\phi_i\big|_{[a,b]}$. By Prop.~\ref{prop:limit_is_solution_inclusion},~$\phi\big|_{[a,b]}$ is a solution to $\hybrid'$. As $[a,b]\subset \interior{\dom{\phi}}$ is arbitrary, $\phi \in \mathcal{S}_{\hybrid'}$.
\end{proof}

{{
\begin{proposition}
    \label{prop:nominalWellPosedness_inclusion_nonlocalEventualBoundedness}
    Suppose that $\hybrid' = (C, F, \varnothing, G, \M)$ satisfies Definition~\ref{ass:hybrid_basic_conditions}. Consider a graphically convergent sequence $\{\phi_i\}_{i=1}^{\infty}$ of solutions $\phi_i$ to $\hybrid'$ that is not locally eventually precompact and does not escape to the horizon. Then, 
    \begin{enumerate}[label = (\roman*)]
        \item \label{item:Prop-i} there exists the smallest $t^* > 0$ such that, for each sequence $\{t_i\}_{i=1}^{\infty}$ of times $t_i \in \dom{\phi_i}$ with $\lim_{i\to\infty} t_i = t^*$, the sequence $\{\phi_i(t_i)\}_{i=1}^\infty$ escapes to the horizon;
        \item \label{item:Prop-ii} {\color{mygreen}\relax{}{\color{mygreen}the truncated graphical limit}} $\phi \coloneqq \left(\graphlim{i\to\infty}{\phi_i} \right)|_{< t^*}$ is a maximal solution to $\hybrid'$;
        \item \label{item:Prop-iii} for each sequence $\{t_i\}_{i=1}^\infty$ of times $t_i \in \dom{\phi}$ with $t_i \to t^*$, the sequence $\{\phi(t_i)\}_{i=1}^\infty$ escapes to the horizon.
    \end{enumerate}
\end{proposition}

{} %
}}

{\color{mygreen}\relax{}{\color{mygreen}\begin{proof}
    The proof follows similarly to the proof of~\cite[Thm.~5.29]{goebel_hybrid_2012}. We present it here for completeness. 

    \medskip\noindent
    (i) As $\{\phi_i\}_{i=1}^{\infty}$ is not locally eventually precompact, by negating Definition~\ref{def:locallyEventuallyPrecompctHybridArcs}, there exists $t \geq 0$, a subsequence $\{\phi_{i_k}\}_{k=1}^\infty$, and points $t_k \in \dom{\phi_{i_k}}$ such that $\lim_{k\to\infty} t_k = t$ and the sequence $\{\phi_{i_k}(t_k)\}_{k=1}^\infty$ escapes to the horizon. Let $T$ be the set of all such $t$'s, and let $t^* \coloneqq \inf T$. Now, proceeding by contradiction, suppose that there exists a subsequence $\{\phi_{i_k}\}_{k=1}^\infty$ of $\{\phi_i\}_{i=1}^\infty$, points $t_k \in \dom{\phi_{i_k}}$ for each $k\in \{1,2,\ldots\}$ such that $t_k \to t^*$, and a compact set $K\subset \M$ such that $\phi_{i_k}(t_k)\in K$ for each $k\in\{1,2,\ldots \}$. Consequently, the subsequence $\{\phi_{i_k}|_{\leq t_k}\}_{k=1}^\infty$ is locally eventually precompact. Then, following~Lemma~\ref{lemma:domainAndRange} and the fact that $\{\phi_i\}_{i=1}^\infty$ is graphically convergent, we have that $t^* \in \dom{\left(\graphlim{k\to\infty}{\phi_{i_k}} \right)}$ and, therefore, that $\left(\operatorname{gph-limsup}_{k\to\infty} \phi_{i_k}\right)(t^*) \cap K \neq \varnothing$. As $\{\phi_{i_k}\}_{k=1}^\infty$ was chosen to be a subsequence of $\{\phi_{i}\}_{i=1}^\infty$ and the latter sequence converges, there exist points $t_i \in \dom{\phi_i}$ with $t_i \to t^*$ such that, for any compact set $K' \subset \M$ with $K \subset \interior{K'}$,  $\phi_{i}(t_i) \in K'$ for each large enough $i$. Next, pick any larger compact set $K'' \subset \M$ so that $K' \subset \interior{K''}$. Then, as each $\phi_i$ is continuous and $\phi_i(t_i) \in K'$ for each large enough $i$, there exists $\tau > 0$ such that
    \[
    \phi_i(t_i) \in K'' \qquad \forall t \in \dom{\phi_i} \cap [t_i - \tau, t_i + \tau]
    \]
    holds for each large enough $i$. Consequently, as $t_i \to t^*$, it follows that $T \cap (t^* - \tau, t^* + \tau) = \varnothing$. Therefore, $t^*$ cannot be the infimum of $T$, causing a contradiction. As a result, for any sequence of times $t_i \in \dom{\phi_i}$ with $t_i \to t^*$, the sequence $\{\phi_i(t_i)\}_{i=1}^\infty$ escapes to the horizon. As a result, we also have that $t^* \in T$. 
    
    To show that $t^* > 0$, note that $t^*\geq 0$ as $t_i \to t^*$ and each $t_i \geq 0$. Suppose by contradiction that $t^*=0$. Let $\phi' \coloneqq \graphlim{i\to\infty}{\phi_i}$. Then, pick any sequence $\{t_i\}_{i=1}^\infty$ of points $t_i \in \dom{\phi_i}$ such that $t_i \to t^*$, causing $\{\phi_i(t_i)\}_{i=1}^\infty$ to escape to the horizon. As the sequence $\{\phi_i\}_{i=1}^\infty$ is graphically convergent and does not escape to the horizon, and as $0 \in \dom{\phi_i}$ for each $i \in\{1,2,\ldots\}$ by Definition~\ref{def:solution}, it follows that $0 \in \dom{\left(\graphlim{i\to\infty}{\phi_i} \right)}$ and there exists a compact set $L \subset \M$ such that $\phi'(0) = \phi'(t^*) \subset L$. This contradicts the claim that $\{\phi_i(t_i)\}_{i=1}^\infty$ escapes to the horizon. Therefore, $t^*$ must be positive. 

    \medskip\noindent
    (ii) To show that $\phi$ is a maximal solution to $\hybrid'$, it suffices to show that, for each $t' \in [0, t^*)$, the restriction $\phi|_{\leq t'}$ is a solution to $\hybrid'$ and $\dom{\phi} = [0, t^*)$. Noting that the sequence $\{\phi_i|_{\leq t'}\}_{i=1}^\infty$ is locally eventually precompact, it follows from Proposition~\ref{prop:limit_is_solution_inclusion} that $\graphlim{i \to\infty}{\phi_i|_{\leq t'}}$ is a solution to $\hybrid'$. As the previous statement holds for each $t' \in [0, t^*)$, it follows that $\phi$ is a solution to $\hybrid'$ with $\dom{\phi} = [0, t^*)$. 

    To prove maximality of $\phi$, suppose by contradiction that $\phi$ is not a maximal solution to $\hybrid'$. Then, there exists a solution $\psi$ to $\hybrid'$ with $\dom{\psi} = [0, t^*]$ and $\psi(t) = \phi(t)$ for each $t\in \dom{\phi}$. Now, pick a sequence $\{t_i\}_{i=1}^\infty \subset [0, t^*)$ with $t_i \in \dom{\phi_i}$ and $t_i \to t^*$. As $\phi_i \to \phi$ graphically and $\graph{\psi} = \overline{\graph{\phi}}$ by construction, it follows that $\phi_i(t_i) \to \psi(t^*)$. This is a contradiction as $\{\phi_i(t_i)\}_{i=1}^\infty$ escapes to the horizon. Therefore, $\phi$ is maximal. 

    \medskip\noindent
    (iii) Pick any sequence $\{t_i\}_{i=1}^\infty \subset [0, t^*)$ with $t_i \to t^*$. Suppose by contradiction that $\{\phi(t_i)\}_{i=1}^{\infty}$ does not escape to the horizon; i.e., there exists a compact set $K \subset \M$ and a subsequence $\{t_{i_k}\}_{k=1}^\infty$ of $\{t_i\}_{i=1}^\infty$ such that $\phi(t_{i_k})\in K$ for each $k\in \{1,2,\ldots\}$. As $\phi_i \to \phi$ graphically, for each $k\in \{1,2,\ldots\}$, there exists a sequence $\{s_{k,n}\}_{n=1}^\infty$ with $s_{k,n}\in\dom{\phi_{i_k}}$, $\lim_{n\to\infty} s_{k,n} = t_{i_k}$, and $\lim_{n\to\infty}\phi_{i_k}(s_{k,n}) = \phi(t_{i_k}) \in K$. Then, for each compact $K' \subset \M$ with $K \subset \interior{K'}$, $\phi_{i_k}(s_{k,n})\in {K'}$ for each $k \in \{1,2,\ldots\}$ and each large enough $n$. Let $n(k)$ denote the large enough $n$ for each $k\in \{1,2,\ldots\}$, and suppose without loss of generality that $n : \{1,2,\ldots\} \to \{1,2,\ldots\}$ is an increasing function. Then, letting $l_k \coloneqq s_{k, n(k)}$, we have that 
    \begin{align}
    \label{eq:contradiction_nonlocalprecompact}
        \phi_{i_k}(l_k) \in {K'} \qquad \forall k\in \{1,2,\ldots\}.
    \end{align}
    As
    \(
        \lim_{k\to\infty}l_{k} = \lim_{k \to \infty} s_{k, n(k)} = \lim_{k\to\infty} t_{i_k} = t^*,
    \) 
    item~\ref{item:Prop-i} contradicts with~\eqref{eq:contradiction_nonlocalprecompact}. Therefore, $\{\phi_i(t_i)\}_{i=1}^\infty$ escapes to the horizon, and the proof is complete. 
\end{proof}
}}

{\color{mygreen}\relax{}Now, we are ready to present the proof of Theorem~\ref{theorem:basicConditions => nominalWellPosedness}. }
\subsubsection{Proof of Theorem~\ref{theorem:basicConditions => nominalWellPosedness}}
    (a) Suppose that the sequence $\{\phi_i\}_{i=1}^\infty$ is locally eventually precompact. Following Lemma~\ref{lemma:domainAndRange} and \cite[Ex.~12.11]{goebel_setvaluedAnalysis}, $\dom{\phi} = \lim_{i\to\infty}\dom{\phi_i}$ is a hybrid time domain. {\color{mygreen}\relax{}Note that $\dom{\phi}$ may not be compact.} Pick any $(T,J)\in\dom{\phi}$ and consider a sequence $0\leq t_0 \leq t_1 \leq \ldots \leq t_{J+1} = T$ such that
    \(
        \dom{\phi} \cap \brackets{[0, T] \times \{0, 1, \ldots, J\}} = \bigcup_{j=0}^J \brackets{[t_j, t_{j+1}]\times \{j\}}.
    \)
    Define $I_{\phi}^j \coloneqq [t_j, t_{j+1}]$ for each $j\in \{0,1,, \ldots, J\}$. Pick~any $j\in \{0,1,\ldots, J\}$ for which $\interior{I^j_{\phi}}\neq \varnothing$ and any~connected, compact interval $\Xi^j\subset \interior{I^j_{\phi}}$. Since $\dom{\phi} = \lim_{i\to\infty}\dom{\phi_i}$, it follows that there exists a connected, compact interval $I^j_{\phi_i}\subset \R{}_{\geq 0}$ such that $I^j_{\phi_i}\times \{j\}\subset \dom{\phi_i}$ and  $(I_{\phi}^j \times \{j\}) \cap (\Xi^j\times \{j\}) \neq \varnothing$ for each large enough $i$. Let $I_{\phi_i}^j$ be the largest such interval. Using $\dom{\phi} = \lim_{i\to\infty}\dom{\phi_i}$, it follows that $\lim_{i\to\infty}I^j_{\phi_i} = I^j_{\phi}$. Consequently, $\Xi^j \subset I^j_{\phi_i}$ for each large enough $i$. Consider the truncated sequence $\{\phi_i\big|_{\Xi^j\times \{j\}}\}_{i=1}^\infty$. It follows from graphical convergence of this sequence that $\phi\big|_{\Xi^j\times\{j\}} = \graphlim{i\to\infty}\phi_i \big|_{\Xi_j\times\{j\}}$. Following Proposition~\ref{prop:nominalWellPosedness_inclusion_localEventualBoundedness}, $\phi\big|_{\Xi^j\times\{j\}}$ solves the constrained differential inclusion $\dot{x}\in F(x) \; x\in C$. Since this holds for each $(T,J)\in\dom{\phi}$, each $j\in\{0,1,\ldots, J\}$, and each $\Xi^j\subset I^j_{\phi}$, $\phi$ satisfies the constrained differential inclusion in $\hybrid$. 

    {}
    {\color{mygreen}\relax{}{To show that $\phi$ also satisfies the constrained difference inclusion in $\hybrid$, pick any $(t,j)\in\dom{\phi}$ such that $(t,j+1)\in\dom{\phi}$. Then, $I_{\phi}^{j}$ is closed, and $I_{\phi}^{j+1}$ is closed on the left. Then, for some sequences $(t'_i, j), (t''_i, j+1)\in \dom{\phi_i}$, we have $\phi(t,j) = \lim_{i\to\infty}\phi_i(t'_i, j)$ and $\phi(t, j+1) = \lim_{i\to\infty} \phi_i(t''_i, j+1)$. As $\dom\phi_i$ is a hybrid time domain for each $i$, there exists $t_i \in [t'_i, t''_i]$ such that $(t_i, j), (t_i, j+1)\in \dom{\phi_i}$ for each large enough $i$, causing $\phi_i(t_i, j)\in D$ and $\phi_i(t_i, j+1)\in G(\phi_i(t_i, j))$ for all such $i$'s. As $\phi_i(t_i, j) \to \phi(t,j)$ and $\phi_i(t_i, j+1) \to \phi(t,j+1)$, it follows from closedness of $D$ that $\phi(t,j)\in D$. Next, as $G$ is outer semicontinuous relative to $D$, $(\phi_i(t_i, j), \phi_i(t_i, j+1)) \to (\phi(t,j), \phi(t,j+1)) \in \graph{G|_{D}}$, causing $\phi(t,j+1) \in G(\phi(t,j))$. Therefore, $\phi$ satisfies the constrained difference inclusion. }}

    (b) Now, suppose that the sequence $\{\phi_i\}_{i=1}^\infty$ is not locally eventually precompact. Let $j^*$ denote the smallest $J \in \mathbb{N}$ for which the sequence $\{\phi_i\}_{i=1}^\infty$, when truncated to $\R{}_{\geq 0} \times \{0,1,\ldots, J\}$, is not locally eventually precompact. For each $i\in \{1,2,\ldots\}$, define $a_i \coloneqq \inf\{t:  (t,j^*)\in \dom{\phi_i}\}$. Similarly, let $a \coloneqq \inf\{t:(t,j^*)\in \dom{\phi}\}$. Due to graphical convergence, for each large enough $i$, $a_i$ is finite and achieves the infimum, causing $(a_i, j^*)\in\dom{\phi_i}$ for all such $i$'s. If $j^* = 0$, then $a_i=0$ for each large enough $i$. If $j^* > 0$, then for each large enough $i$, $(a_i, j^* - 1) \in \dom{\phi_i}$. For such $i$'s, 
    \begin{align}
    \label{eq:jumpInclusion}
        \phi_i(a_i, j^*)\in G(\phi_i(a_i, j^*-1)). 
    \end{align}
    Now, consider the truncation of $\{\phi_i\}_{i=1}^\infty$ to $\R{}_{\geq 0}\times \{0,1,\ldots, j^*-1\}$, which is locally eventually precompact. Using the proof of item a above, the graphical limit of this restriction, i.e., $\phi|_{\R{}_{\geq 0} \times \{0,1,\ldots, j^*-1\}}$, is a solution to $\hybrid$. Local eventual precompactness also yields that $\phi_i(a_i, j^*-1) \to \phi(a, j^* - 1)$. Using this property with local precompactness of $G$, there exists an open neighborhood $U$ of $\phi(a, j^*-1)$ such that $\phi_i(a_i, j^*-1)\in U$  for each large enough $i$ and $G(U)$ is compact. Therefore, from~\eqref{eq:jumpInclusion}, the sequence $\phi_i(a_i, j^*)$ converges to $\phi(a, j^*)$. Using outer semicontinuity of $G$ relative to $D$, $\phi(a, j^*) \in G(\phi(a, j^*-1))$. Finally, as $\phi_i(a_i, j^*) \to \phi(a, j^*)$, the sequence $\{\phi_i\}_{i=1}^\infty$, restricted to $\R{}_{\geq a}\times \{j^*\}$, does not escape to the horizon. 

{\color{mygreen}\relax{}Then, Proposition~\ref{prop:nominalWellPosedness_inclusion_nonlocalEventualBoundedness} completes the proof, with $\tau$ from the theorem statement equal to $a+t^*+j^*$, where $t^*$ is obtained from Proposition~\ref{prop:nominalWellPosedness_inclusion_nonlocalEventualBoundedness}. }
\hfill $\blacksquare$
\section{Auxiliary Results on Geometric Notions}
\label{appendix:B}

\begin{lemma}[Coordinate independence of Definition~\ref{def:tangentConeManifolds}]
\label{lemma:coordInvariant-tangentCone}
    Given a $C^1$-manifold $\M$, a nonempty set $S\subset \M$, a point $x\in S$, and two coordinate charts $(U, \varphi)$ and $(V, \psi)$ of $\M$ at $x$,
    \(
        (\diffFunc{\varphi_x})^{-1} (\TconeEuclidean{\varphi(S\cap U)}{\varphi(x)}) = (\diffFunc{\psi_x})^{-1} (\TconeEuclidean{\psi(S\cap V)}{\psi(x)}).
    \)
\end{lemma}

{
\begin{proof}
    For brevity, let $A\coloneqq \varphi(S\cap U)$, $B \coloneqq \psi(S\cap V)$, $z \coloneqq \varphi(x)$, and $w \coloneqq \psi(x)$. Using~\cite[Ex.~2.14(b)]{Lee}, $\varphi$ and $\psi$ are $C^1$-diffeomorphisms onto their images, and so is $h \coloneqq \psi \circ \varphi^{-1}$. Then, using~\cite[Prop.~3.6(b,d)]{Lee}, it suffices to show that
    \(
        \diffFunc{h_{z}}(\TconeEuclidean{A}{z}) = \TconeEuclidean{B}{w}.
    \)
    We show $\diffFunc{h_z}(\TconeEuclidean{A}{z}) \subset \TconeEuclidean{B}{w}$. 
    Pick any $v \in \TconeEuclidean{A}{z}$. Following~\cite[\S4.1.1]{AubinFrankowska2009}, there exist sequences $a_i \in A$ and $\tau_i > 0$ with $\tau_i\searrow 0$ such that 
    \(
        a_i \to z, \text{and } (a_i - z)/{\tau_i}\to v.
    \)
    Since $h$ is differentiable at $z$,
    \(
        \lim_{i \to\infty}{(h(a_i) - h(z))}{/\tau_i} = \diffFunc{h_z}(\lim_{i\to\infty} {(a_i - z)}{/\tau_i}) = \diffFunc{h_z}(v). 
    \)
    As $a_i \in A$ implies $h(a_i)\in B$ and $h(a_i) \to h(z) = w  \in B$, we have $\diffFunc{h_z}(v)\in \TconeEuclidean{B}{w}$. As $v\in \TconeEuclidean{A}{z}$ is arbitrary, $\diffFunc{h_z}(\TconeEuclidean{A}{z})\subset \TconeEuclidean{B}{w}$. That $\diffFunc{h_z}(\TconeEuclidean{A}{z}) \supset \TconeEuclidean{B}{w}$ follows similarly by replacing $h$ with $h^{-1}$.
\end{proof}
}

\begin{proposition}[Rademacher's theorem on $C^1$-manifolds]
    \label{prop:rademacher}
    A locally Lipschitz function $f : \M \to \N$ between~$C^1$-manifolds is differentiable almost everywhere on $\M$.
\end{proposition}

A similar result appears in [38, Cor. B.5], where the argument is presented within a single coordinate chart. Our proof complements this result by employing~\cite[Lem.~6.6]{Lee} to obtain the corresponding manifold-wide conclusion.

{
\begin{proof}
    Following~\cite[Prop.~A.16]{Lee}, pick any countable open cover $\{U_{\alpha}\}_{\alpha\in \mathcal{I}}$ of $\M$, where $\cal I$ is a countable index set and $(U_{\alpha}, \varphi_{\alpha})$ is a chart at $x_{\alpha}\in \M$ for each $\alpha\in\cal I$. From any open cover $\{W_y\}_{y\in \N}$ of $\N$ obtained from charts $\{(W_y, \psi_y)\}_{y\in \N}$, we use local Lipschitz continuity of $f$ and Def.~\ref{def:Lipschitz} to pick a countable subcover $\{W_\alpha\}_{\alpha\in\cal I}$ that covers $f(\M)\subset \N$, where $(W_\alpha, \psi_{\alpha})$ is a chart at $f(x_{\alpha})$. 

    Using~\cite[Thm.~3.1.6]{Federer1996}, for each $\alpha\in \mathcal{I}$, there exists a Lebesgue measure zero set $\widetilde{S}_{\alpha} \subset \varphi_{\alpha}(U_{\alpha} \cap f^{-1}(W_{\alpha}))$ such that $\psi_{\alpha} \circ f \circ \varphi_{\alpha}^{-1}$ is differentiable at each $\widetilde{x}\in \varphi_{\alpha}(U_{\alpha} \cap f^{-1}(W_{\alpha}))\setminus \widetilde{S}_{\alpha}$. Let $S\coloneqq \bigcup_{\alpha\in\mathcal{I}}\varphi^{-1}_{\alpha}(\widetilde{S}_{\alpha}) \subset \M$. Then, for~each $x\in\M\setminus S$, there exists a chart $(U, \varphi)$ at $x \in \M$ and $(W,\psi)$ at $f(x)\in \N$ such that $\psi \circ f\circ \varphi^{-1}$ is differentiable at $\varphi(x)$. Finally, $S$ has Lebesgue measure zero due to~\cite[Lemma 6.6]{Lee}. 
\end{proof}

}

Finally, to present an equivalence between the UGS notion in Definition~\ref{def:UGS} and its Riemannian counterpart defined using a $\mathcal{K}_{\infty}$ function, we recall the following notions for a Riemannian manifold $(\M, g_{\M})$. If $\M$ is connected, its metric induces the geodesic distance
    \(
    d_\M(x,y),%
    \)
which is the infimum of length of all $C^1$-curves joining $x$ and $y$. If $\M$ is disconnected, such curves do not exist. Instead, we fix a distance $d_\M$ whose open balls generate the manifold topology~\cite[Cor.~13.30]{Lee}; such a distance is called \emph{compatible} with the topology of $\M$. For a nonempty closed set $\A\subset\M$, define
    \(
    \distfromA{x}\coloneqq\inf_{y\in\A}d_\M(x,y)
    \)
and $\mathbb{B}_{r}(\A) \coloneqq \{x\in \M : \distfromA{x}\leq r\}$ for each $r \geq 0$. 
We say that $\M$ is a \emph{complete Riemannian manifold} if each Cauchy sequence converges under the given distance function, namely, if for each sequence $\{x_i\}_{i=1}^{\infty} \subset \M$ such that, for each $\varepsilon > 0$, there exists $i_0 > 0$ satisfying $d_{\M}(x_i, x_j) < \varepsilon$ for each $i, j > i_0$, the sequence $\{x_i\}_{i=1}^\infty$ converges.

{
First, we enforce the following mild assumption on $d_{\M}$ while allowing $\M$ to be potentially disconnected. The disconnectedness of $\M$ enables application of the forthcoming result to the hybrid systems setting, where the underlying state space is often disconnected to model finitely many logic modes.

\begin{assumption} \label{ass:disconnectedManifold}
    Given a Riemannian manifold $(\M, g_{\M})$ with a distance function $d_{\M}$ that is compatible with the topology of $\M$,
    \begin{enumerate}[label=(\roman*)]
        \item \label{item:i-RiemannianAssumption} the restriction of $d_{\M}$ to each connected component\footnote{A \emph{connected component} of a manifold $\M$ is a subset that is not properly contained in any larger connected subset of $\M$. 
        } of $\M$ yields the Riemannian distance function, and
        \item \label{item:ii-RiemannianAssumption} $\M$ has finitely many connected components.
    \end{enumerate}
\end{assumption}

The next result on compactness of metric inflations~is crucial for the forthcoming UGS equivalence; see~\cite{jirwankarTAC2026} for details. 
\begin{lemma}\label{lemma:metric_inflation_compact}
    Let $(\M, g_{\M})$ be a complete Riemannian manifold satisfying
    Assumption~\ref{ass:disconnectedManifold} and $\A \subset \M$ be nonempty and compact. Then, the following hold:
    \begin{enumerate}[label={(\roman*)}]
        \item For each $r > 0$, $\mathbb{B}_r(\cal A)$ is a compact neighborhood of $\cal A$.
        \item For each neighborhood $W$ of $\cal A$, there exists $r > 0$ such
            that $\mathbb{B}_r(\cal A) \subset W$.
    \end{enumerate}
\end{lemma}

\begin{proposition}%
\label{prop:UGS}\!\!\textit{(Riemannian Equivalence of UGS notion)}
    Let $(\M, g_{\M})$ be a complete Riemannian manifold endowed with a distance function $d_{\M}$ that satisfies Assumptions~\ref{ass:disconnectedManifold}. A nonempty, compact set $\A\subset\M$ is UGS for $\mathcal{H}$ if and only if there exists $\alpha\in\mathcal{K_\infty}$ such that each $\phi \in \mathcal{S}_{\hybrid}$ satisfies 
    $\distfromA{\phi(t,j)} \leq \alpha(\distfromA{\phi(0,0)})$  for all $(t,j)\in \dom{\phi}$.
\end{proposition}

{
    A similar result in the Euclidean setting is presented in~\cite[Thm.~1]{Andriano1997-LagrangeStability}; see also~\cite{Angeli2025_LyapunovLagrange} for a different topological version.  
}

{
\begin{proof}
    $(\Leftarrow)$ Assume that there exists $\alpha\in \cal K_{\infty}$ such that $\distfromA{\phi(t,j)}\le \alpha(\distfromA{\phi(0,0)})$ for all $\phi\in \cal S_{\cal H}$ and $(t,j)\in \dom \phi$.  Let $W$ be a compact neighborhood of $\A$, and let $\underline{w} \coloneqq \sup \{r' \ge 0 : \mathbb{B}_{r'}(\cal A) \subset W\}$ and $\overline{w} \coloneqq \inf\{r'\ge0: W\subset \mathbb{B}_{r'}(\A)\} =\sup_{x\in W}|x|_{\cal A}$. By definition, $0 < \underline{w}\le\overline{w}$ and $ \mathbb{B}_{\underline{w}}(\A)  \subset W \subset \mathbb{B}_{\overline{w}}(\A)$. By compactness of $W$, $\overline{w}<\infty$. Let $ U\coloneqq  \mathbb{B}_{\alpha^{-1}(\underline{w})}(\A)$. Then, by assumption and given that $\alpha\in \cal K_{\infty}$, for each $\phi\in \cal S_{\cal H}(U)$, $\rge{\phi} \subset \mathbb{B}_{\underline{w}}(\A) \subset W$. Thus, from Definition~\ref{def:UGS}, $\A$ is ULyS for $\hybrid$. Furthermore, for each $\phi\in \cal S_{\cal H}(W)$ it follows that $|\phi(t,j)|_{\cal A}\le\alpha(|\phi(0,0)|_{\cal A}) \le\sup_{x\in W}\alpha(|x|_{\cal A})\le \alpha(\overline{w})$ for all $(t,j)\in\dom{\phi}$. Consequently, $\rge{\phi}\in \mathbb{B}_{\alpha(\overline{w})}(\A) \eqqcolon X$. Since $X$ is compact by Lemma~\ref{lemma:metric_inflation_compact}, $\A$ is ULaS for $\hybrid$. Therefore, $\A$ is UGS for $\hybrid$.

    \vspace{-2pt}
    \medskip\noindent
    $(\Rightarrow)$ 
    Let $\cal A$ be UGS for $\mathcal{H}$. Define $\frak u :
    \R{}_{>0} \to \R{}_{>0}$ as
    \vspace{-6pt}
    \begin{equation*}
        \mathfrak{u}(\varepsilon) \coloneqq \sup\left\{\delta\in (0, \varepsilon] : 
            \forall \phi \in \cal S_{\cal H}\left(\mathbb{B}_{\delta}(\cal A)\right), 
            \rge{\phi}\subset \mathbb{B}_{\varepsilon}(\cal A)
        \right\}
        \vspace{-6pt}
    \end{equation*}
    for each $\varepsilon> 0$. As $\A$ is ULyS for $\hybrid$, Lemma~\ref{lemma:metric_inflation_compact} causes $\mathfrak{u}$ to be well-defined. {\color{mygreen}\relax{}{\color{mygreen}By construction, each $\phi \in \mathcal{S}_{\hybrid}$ is such that 
    \(
        \distfromA{\phi(0,0)} \leq \frak{u}(\varepsilon)
    \)
    implies
    \(
        \distfromA{\phi(t,j)} \leq \varepsilon
    \)
    for each $(t,j)\in\dom{\phi}$ and each $\varepsilon\in \R{}_{>0}$.}}
   Additionally, $\frak u $ is  nondecreasing, possibly discontinuous, and $\lim_{s\to 0}\mathfrak{u}(s) = 0$. {\color{mygreen}\relax{}{\color{mygreen}Then, by~\cite
   [Lemma~2.5]{CLARKE199869},}} {}there exists $\zeta\in \mathcal{K}$ such that $\zeta(s) \leq \mathfrak{u}(s)$ for all $s\in \dom{\mathfrak{u}}$. Note that $\rge{\zeta} = [0, a)$ for some $a\in \R{}_{>0}\cup \{\infty\}$.
   Define $\underline{\alpha}\coloneqq \zeta^{-1} \in \mathcal{K}$ so that $\dom{\underline{\alpha}} =\rge{\zeta} = [0, a)$. {}
   {\color{mygreen}\relax{}{\color{mygreen}Since $\zeta(\varepsilon)\le \frak{u}(\varepsilon)$ for each $\varepsilon>0$ by construction, and $\underline{\alpha}=\zeta^{-1}$, we obtain for all $\varepsilon >0$,  all $\phi \in \cal S_{\cal H}$ and all $(t,j)\in\dom{\phi}$ that}}
        \vspace{-6pt}
        \begin{multline}
            \distfromA{\phi(0,0)}= \zeta(\varepsilon) \leq \mathfrak{u}(\varepsilon)  \implies 
            \\
            \distfromA{\phi(t,j)} \leq \varepsilon = \underline{\alpha}(\zeta(\varepsilon)) = \underline{\alpha}(\distfromA{\phi(0,0)}).
            \label{eq:proof:UGS:Lyapunov}
        \end{multline}
    In particular, picking $\zeta^\star\in (0, a)$,~\eqref{eq:proof:UGS:Lyapunov} holds for each $\varepsilon\in (0,\zeta^\star]$.
    Following a similar procedure with ULaS of $\A$ in Definition~\ref{def:UGS}, let $\mathfrak{v} : \R{}_{>0} \to \R{}_{>0}$ be defined by
    {}
    {\color{mygreen}\relax{}{\begin{equation*}
        \mathfrak{v}(\delta) \coloneqq \inf\left\{\varepsilon \in [\delta, \infty) :  \begin{array}{c}
            \text{for each } \phi\in \cal{S}_{\cal H}\left(\mathbb{B}_{\delta}(\cal A)\right),
            \\
            \rge{\phi} \subset \mathbb{B}_{\varepsilon}(\cal A)
        \end{array}\right\}
    \end{equation*}}}
    for each $\delta > 0$. Note that $\frak v$ is well-defined by Definition~\ref{def:UGS}. Then, each $\phi \in \mathcal{S}_{\hybrid}$ is such that $\distfromA{\phi(0,0)} \leq \delta$ implies $\distfromA{\phi(t,j)} \leq \frak{v}(\delta)$ for each $(t,j)\in\dom{\phi}$. 
    We also have that $\frak v$ is nondecreasing, possibly discontinuous, and $\lim_{s\to \infty}\frak{v}(s) = \infty$. Recalling $\zeta^\star${\color{mygreen}\relax{} {\color{mygreen}and using Lemma~\ref{lemma:Kinfty_upperBound}}}, there exists $\overline{\alpha} \in \mathcal{K}_{\infty}$ satisfying $\overline{\alpha}(s) \geq \frak{v}(s)$ for all $s\in [\underline{\alpha}^{-1}(\zeta^\star), \infty)$. 

    By construction, $\underline{\alpha}^{-1}(s) \leq s$ for each $s\geq 0${\color{mygreen}\relax{}{\color{mygreen}, as $\underline{\alpha}^{-1}(s) = \zeta(s)\leq \mathfrak{u}(s) \leq s$ for each $s\geq 0$, where the last inequality follows from the very definition of $\mathfrak{u}$}}. Consequently, $\underline{\alpha}^{-1}(\zeta^\star) \leq \zeta^\star$, causing $\overline{\alpha}(s)\geq \mathfrak{v}(s)$ for each $s \in [\zeta^\star, \infty)$. Then, for each $\delta \geq \zeta^\star$, each $\phi\in \mathcal{S}_{\hybrid}$, and each $(t,j)\in\dom{\phi}$,
    \begin{multline}
        \distfromA{\phi(0,0)} = \delta \implies 
        \\
        \distfromA{\phi(t,j)} \leq \frak v(\delta) 
        \leq \overline{\alpha}(\delta)
        = \overline{\alpha}(\distfromA{\phi(0,0)}).
        \label{eq:proof:UGS:Lagrange}
    \end{multline}
    
    Finally, to unify the bounds in \eqref{eq:proof:UGS:Lyapunov} and \eqref{eq:proof:UGS:Lagrange}, define $\alpha(s) \coloneqq \max\{\underline{\alpha}(s), \overline{\alpha}(s)\}$ for each $s\geq 0$. As $\underline{\alpha}, \overline{\alpha}\in \mathcal{K}_{\infty}$, it follows {\color{mygreen}\relax{}{\color{mygreen}from~\cite[p.~342]{Kellett2014}}} that $\alpha\in \mathcal{K}_{\infty}$. Then, using \eqref{eq:proof:UGS:Lyapunov} and \eqref{eq:proof:UGS:Lagrange}, each $\phi\in\mathcal{S}_{\hybrid}$ satisfies $\distfromA{\phi(t,j)}\leq \alpha(\distfromA{\phi(0,0)})$ for all $(t,j)\in\dom{\phi}$.
\end{proof}
}

}

{\color{mygreen}

\section{Hybrid Controller for Quaternion Stabilization in Example~\ref{ex:quaternion}}
\label{app:QuaternionController}

We recall the hybrid control law from~\cite{mayhew2011quaternion} that achieves the desired objective in Example~\ref{ex:quaternion}. Recalling the flow set $C$ and the jump set $D$, the hybrid controller is a dynamic control law with state $h\in H$, input $q$, and output $\xi$, satisfying the following hybrid dynamics:
\begin{align*}
    \left\{
    \begin{array}{cccl}
        \dot{h} & = & 0 & (q, h) \in C \\
        h^+ & = & -h & (q, h) \in D \\
        \xi & = & -h K_{\varepsilon} \varepsilon &
    \end{array}
    \right.
\end{align*}
where $K_{\varepsilon}\in \R{3\times 3}$ is a symmetric, positive definite matrix. The above hybrid controller, when interconnected with the continuous-time quaternion kinematics, results in the system $\hybrid = (C, F, D, G, \M)$ as described in Example~\ref{ex:quaternion}.

\section{Example~\ref{ex:mobius-solutions} Continued: Continuity of $\Phi$}
\label{app:Phi_continuous}

    Note that the map $(x,\lambda)\mapsto (y_1, y_2^{||}-\lambda y_2^{\perp}) \in \T{}{\mobius}$ is continuous as the projection of $y_2$ onto its components that are parallel and perpendicular to the \emph{smooth} boundary $\partial S$ is continuous. Next, consider the family of curves $H : \partial S \times [0,1] \to \mobius$, defined as $H(y_1, t) \coloneqq [z_1, z_2 - t \operatorname{sign}(z_2)\varepsilon]_{\mobius}$ for each $(y_1,t)\in \partial S\times [0,1]$. As $\operatorname{sign}(z_2)$ is locally constant and equal to either one or negative one, the map $H$ is continuous. Then, as $g_1(y_1) = H(y_1, 1)$ for each $y_1\in \partial S$, it follows that $g_1$ is continuous as well. Therefore, to establish continuity of $\Phi$, it suffices to show that $\T{y_1}{\mobius} \ni (y_1, v) \mapsto \mathfrak{P}(y_1, v)$ is continuous. 

    To this end, note that, for each $y_1\in D_1$ and each $v\in \T{y_1}{\mobius}$, we have that $\mathfrak{P}(y_1, v) = X(H(y_1, 1))$ such that $X(H(y_1,0)) = v$. Recall that $X$ is the vector field parallel along the curve $\gamma_{y_1} = H(y_1, \cdot)$. Following~\cite[p.~106]{Lee_Riemannian}, $t\mapsto X(H(y_1, t))$ is obtained by solving the following ODE in the local coordinate representation of $X(H(y_1, \cdot))$:
    \begin{align}
    \label{eq:ODE}
        \dot{\widetilde{X}} = A({\gamma_{y_1}}, \dot{\gamma}_{y_1}) \widetilde{X}, \quad \widetilde{X}(0) = \widetilde{v},
    \end{align}
    where $\widetilde{v}$ denote the coordinate representation of $v\in \T{y_1}{\mobius}$, and $A : \T{}{\mobius} \to \R{2\times 2}$ is the map whose $(k,j)$-th element\footnote{The $(k,j)$-th element is the element in the $k$'th row and $j$'th column.} is denoted $A^k_j(z,w)\coloneqq -\Gamma^{k}_{ij}(z)\widetilde{w}^i$ for each $(z,w)\in\T{}{\mobius}$. Now, as the Christoffel symbols $\Gamma^{k}_{ij}$ of the Levi-Civita connection are smooth functions, the map $A$ is smooth. Therefore, following~\cite[Theorem~D.1(c)]{Lee}, solutions to~\eqref{eq:ODE} depend continuously on their initial conditions. Then, as $\mathfrak{P}(y_1, v) = X(H({y_1},1))$ for each $(y_1, v)\in \T{}{\mobius}$, it follows that $(y_1, v)\mapsto \mathfrak{P}(y_1, v)$ is a continuous function. Consequently, we have that $\Phi$ in~\eqref{eq:Phi} is continuous.

\section{Compendium of Auxiliary Results}

In this appendix, we present several useful auxiliary results related to set-valued analysis, graphical convergence, comparison functions, a version of the Hopf-Rinow theorem on complete Riemannian manifolds, and geodesic convexity on Riemannian manifolds. 

\subsection{Auxilliary results on graphical convergence}

\begin{lemma}
    \label{lemma:escapeToHorizon}
    Consider a sequence of sets $\{S_n\}_{n=1}^\infty$ on a topological manifold $\M$. The sequence escapes to the horizon if and only if $\lim_{n\to\infty}S_n = \varnothing$.
\end{lemma}

\begin{proof}
    $(\Leftarrow)$ The limit of the sequence exists, and $\lim_{n\to \infty}S_n  = \varnothing$. Therefore, $\limsup_{n\to \infty} S_n = \varnothing$ and $\liminf_{n\to\infty} S_n = \varnothing$. Then, to prove that the sequence escapes to the horizon, assume the opposite, i.e., the sequence does not escape to the horizon. Then, there exists a compact set $K\subset \M$ such that, for each $n_0 > 0$, there exists $n \geq n_0$ satisfying $S_n \cap K \neq \varnothing$. As a result, $K\ \cap\limsup_{n\to\infty}S_n \neq \varnothing$, causing $\limsup_{n\to\infty}S_n \neq \varnothing$, which is a contradiction.

    $(\Rightarrow)$ Reasoning by contradiction, assume that there exists $x\in \M$ such that $x\in\lim_{n\to\infty} S_n$. Note that, via \cite[Prop. 4.64, Lemma 4.65]{lee2010introduction}, there exists an open precompact neighborhood $U$ of $x$. Since, in particular, $x\in \limsup_{n\to\infty} S_n$, it follows that for each $n_0 > 0$ there exists $n>n_0$ satisfying \(U\cap S_n \neq \varnothing.\) However, as the sequence $\{S_n\}_{n=1}^\infty$ escapes to the horizon, there exists $n_0>0$ such that $\overline{U}\cap S_n=\varnothing$ for each $n > n_0$, yielding a contradiction. Therefore, it must be that $\limsup_{n\to \infty}S_n = \varnothing$.
\end{proof}

\begin{lemma}
    \label{lemma:convergentSubsequence}
    For each sequence of hybrid arcs $\{\phi_i\}_{i=1}^\infty$, where $\phi_i : \dom{\phi_i}\to \M$, such that there exists a compact set $K\subset \M$ satisfying $\phi_i(0,0)\in K$ for each $i\in \mathbb{N}$, there exists a subsequence that converges graphically. 
\end{lemma}

\begin{proof}
    The sequence of hybrid arcs generates a sequence of sets $\{\graph{\phi_i}\}_{i=1}^{\infty}$ on $\R{}_{\geq 0}\times \mathbb{N}\times \M$. Due to the existence of a compact set $K$ as in the statement of the lemma, applying 
    Lemma~\ref{lemma:escapeToHorizon}
    proves that the sequence $\{\graph{\phi_i}\}_{i=1}^\infty$ does not escape to the horizon. Then, due to Lemma~\ref{lemma:sequentialCompactness_beer}, there exists a subsequence of $\{\graph{\phi_i}\}_{i=1}^\infty$ that converges to a nonempty set. Alternatively, the sequence of hybrid arcs $\{\graph{\phi_i}\}_{i=1}^{\infty}$ has a graphically convergent subsequence. 
\end{proof}

\subsection{Auxiliary Results on Comparison Functions}

\begin{lemma}
\label{lemma:Kinfty_upperBound}
    Let $\gamma:\R{}_{>0} \to \R{}_{>0}$ be a nondecreasing, possibly discontinuous function satisfying $\lim_{s\to\infty}\gamma(s) = \infty$. Then, for each $r > 0$, there exists $\alpha\in \mathcal{K}_{\infty}$ such that $\gamma(s)\leq \alpha(s)$ for each $s \geq r$. 
\end{lemma}
\begin{proof}
    Pick any $r > 0$, and consider the sequence $\{s_i\}_{i=1}^{\infty} \subset \R{}_{>0}$ such that $s_i \coloneqq r + i$ for each $i\in\mathbb{N}\setminus\{0\}$. Let $s_0 \coloneqq r/2$. For each $i\in\mathbb{N}\setminus\{0\}$, define also $y_i \coloneqq \gamma(s_{i+1}) + i+1$, and let $y_0\coloneqq \gamma(r) + 1$. Let $s_{-1}\coloneqq 0$ and $y_{-1}\coloneqq 0$. Then, we define $\alpha$ as a piecewise affine function such that
    \begin{align*}
        \alpha(s) \coloneqq \left\{ 
        \begin{array}{cl}
            y_{-1} + s \frac{(y_0 - y_{-1})}{(s_0 - s_{-1})} & \text{if } s\in [s_{-1}, s_{0}] \\
            y_{0} + s \frac{(y_1 - y_0)}{(s_1 - s_0)} & \text{if } s\in [s_0, s_{1}]\\
            \vdots & \\
            y_{i} + s \frac{(y_{i+1} - y_{i})}{(s_{i+1} - s_{i})} & \text{if } s\in [s_i, s_{i+1}]\\
            \vdots & \\
        \end{array}
        \right. \qquad \forall s\geq 0. 
    \end{align*}

    Note that $\alpha$ is continuous, zero at zero, strictly increasing, and satisfies $\lim_{s \to \infty} \alpha(s) = \infty$. Therefore, $\alpha\in\mathcal{K}_{\infty}$. To verify that $\gamma(s)\leq \alpha(s)$ for each $s\geq 0$, pick any $i \in \{-1, 0, 1, 2, \ldots\}$ and any $s\in [s_{i}, s_{i+1}]$. Note that 
    \begin{align*}
        \alpha(s) &= y_i + s \frac{(y_{i+1} - y_i)}{(s_{i+1} - s_i)} \\
        & \geq y_i \\
        &= \gamma(s_{i+1}) + i + 1  \\
        &\geq \gamma(s). 
    \end{align*}
    This completes the proof. 
\end{proof}

\subsection{Metric Inflations of Compact Subsets of Riemannian Manifolds}

\begin{lemma}\label{lemma:metric_inflation_compact}
    Let $(\M, g_{\M})$ be a complete Riemannian manifold satisfying Assumption~\ref{ass:disconnectedManifold}, and let $\cal A \subset \M$ be nonempty and compact and define the closed metric inflation of $\cal A$ as
    \begin{align} \label{eq:closedMetricInflation}
         \mathbb{B}_r(\cal A) \coloneqq \{x \in \M : \inf_{a\in \cal A}d_{\cal M}(x,a) \leq r\},
    \end{align}
    where $d_{\M} : \M \times \M \to \R{}_{\geq 0}$ is the Riemannian distance metric. Then, the following hold:
    \begin{enumerate}[label=\emph{(\roman*)}]
        \item For each $r > 0$, $\mathbb{B}_r(\cal A)$ is a compact neighborhood of $\cal A$.
        \item For each neighborhood $W$ of $\cal A$, there exists $r > 0$ such
            that $\mathbb{B}_r(\cal A) \subset W$.
    \end{enumerate}
\end{lemma}

\begin{proof}(i) First, we show that $\mathbb{B}_r(\cal A)$ is a closed neighborhood of $\cal A$. Since $x \mapsto |x|_{\cal A}$ is continuous, $\mathbb{B}_r(\cal A)$ is closed as the preimage of the closed set $[0,r]$. For each $r > 0$ and each $x \in \cal A$, $|x|_{\cal A} = 0 < r$, so $\cal A \subset \{x \in X : |x|_{\cal A} < r\} \subset \mathbb{B}_r(\cal A)$. The set $\{x \in X : |x|_{\cal A} < r\} $ is open as the preimage of $[0,r)$, which is open in $\mathbb{R}_{\geq 0}$, under the continuous map $x \mapsto |x|_{\cal A}$. Therefore, $\cal A \subset \interior\mathbb{B}_r(\cal A)$, which implies that $\mathbb{B}_r(\cal A)$ is a neighborhood of $\cal A$ for each $r>0$. 

    By assumption, there exist $K \in \mathbb{N}$ and open sets $\{M_i\}_{i=1}^K$ with $\cal M = \bigcup_{i=1}^K M_i$ and $M_i \cap M_j = \varnothing$ for each $i \neq j$, namely the connected components of $\cal M$. For each $i \in \{1,\ldots,K\}$, $M_i = \cal M \setminus \bigcup_{j \neq i} M_j$ is also closed, and $\mathbb{B}_r(\cal A) = \bigcup_{i=1}^K (\mathbb{B}_r(\cal A) \cap M_i)$.

    Now, we claim that $\mathbb{B}_r(\cal A) \cap M_i$ is closed and bounded in $(M_i, d_g)$ for each $i \in \{1,\ldots,K\}$. Since $\mathbb{B}_r(\cal A)$ is closed in $\cal M$ and $M_i$ carries the subspace topology, $\mathbb{B}_r(\cal A) \cap M_i$ is closed in $M_i$. To show boundedness, fix $x_0 \in M_i$ and define $r_0 \coloneqq \max_{a \in \cal A}d(a,x_0) \in \mathbb{R}_{\geq 0}$, which is finite by compactness of $\cal A$ and continuity of $y \mapsto d_{\cal M}(y,x_0)$. Note that, since $\cal A$ is compact, for each $x \in \mathbb{B}_r(\cal A) \cap M_i$ there exists $a_x^\star \in \cal A$ with $d_{\cal M}(x,a_x^\star) = |x|_{\cal A} \leq r$. By compatibility of $d_{\cal M}$ with the Riemannian distance $d_g$, for each $x, y \in \mathbb{B}_r(\cal A) \cap M_i$,
    \begin{align*}
       &d_g(x,y) = d_{\cal M}(x, y)\\
                &\le d_{\cal M}(x,x_0) + d_{\cal M}(x_0,y)\\
                &\le d_{\cal M}(x,a^\star_x) + d_{\cal M}(a^\star_x, x_0) + d_{\cal M}(x_0,a_y^\star) + d_{\cal M}(a_y^\star, y)\\
                &\leq 2(r + r_0).
    \end{align*}
    Thus, $\mathbb{B}_r(\cal A)\cap M_i$ is closed and bounded in $(M_i,d_g)$.

    Since $\cal M$ is complete by assumption, each connected component $M_i$ is a complete connected Riemannian manifold. Thus, by the Hopf--Rinow theorem \cite[Thm.~16.17]{gallierDifferentialGeometry2020}, $\mathbb{B}_r(\cal A) \cap M_i$ is compact in $M_i$ for each $i \in \{1, 2, \ldots, K\}$. Now, fix $i\in\{1,2,\ldots, K\}$ and let $\{U_\alpha\}_{\alpha\in\Lambda}$ be an open cover of $\mathbb{B}_r(\cal A)\cap M_i$ by sets open in $\cal M$. Then, $\{U_\alpha\cap M_i\}_{\alpha\in\Lambda}$ is an open cover of $\mathbb{B}_r(\cal A)\cap M_i$ by sets open in $M_i$. Since $\mathbb{B}_r(\cal A)\cap M_i$ is compact in $M_i$, there exist finitely many indices $\{\alpha_j\}_{j=1}^m$ such that $\mathbb{B}_r(\cal A)\cap M_i\subset \bigcup_{j=1}^n (U_{\alpha_j}\cap M_i)\subset \bigcup_{j=1}^n U_{\alpha_j}$. Thus, $\{U_{\alpha_j}\}_{j=1}^n$ is a finite subcover of $\mathbb{B}_r(\cal A)\cap M_i$ in $\cal M$. Therefore, $\mathbb{B}_r(\cal A)\cap M_i$ is compact in $\cal M$ for each $i\in\{1,2,\ldots,K\}$. This implies that, for each $r\ge0$, $\mathbb{B}_r(\cal A)=\bigcup_{i=1}^K(\mathbb{B}_r(\cal A)\cap M_i)$ is a finite union of compact subsets of $\cal M$, hence compact.\medbreak
    
    \noindent(ii) Let $U \in \cal N_{\cal A}$, where $\cal N_{\cal A}$ defines the set of all neighborhoods of $\A$. By definition of neighborhood, there exists an open set $V$ with $\cal A \subset V \subset U$. Since $V$ is open, $\cal M \setminus V$ is closed and disjoint from $\cal A$. Thus, given that $\cal A$ is compact and $x \mapsto |x|_{\cal M \setminus V}$ is continuous there exists $\hat{a}\in \cal A$ such that $|\hat{a}|_{\cal M\setminus V}=\inf_{a\in \cal A} |a|_{\cal M\setminus V}$. Additionally, $|\hat{a}|_{\cal M \setminus V} > 0$ since $\hat{a}\in V$ and $\cal M\setminus V$ is closed. Let $r \coloneqq \frac{1}{2}|\hat{a}|_{\cal M \setminus V} > 0$, which implies that
    \begin{equation*}
        0<2r \le |a|_{\cal M\setminus V}=  \inf_{z\in \cal M\setminus V} d_{\cal M}(z,a)\le d_{\cal M}(z,a)
    \end{equation*}
    for all  $(z,a)\in (\cal M\setminus V) \times \cal A$.
     On the other hand, note that for each $y \in \mathbb{B}_{r}(\cal A)$ there exists  $a^\star \in \cal A$ such that $d_{\cal M}(y, a^\star) = |y|_{\cal A} \leq r$. Therefore, for each $z \in X \setminus V$ and each $y\in \mathbb{B}_{r}(\cal A)$, 
    \begin{align*}
        d_{\cal M}(z,y) &= d_{\cal M}(z,y) + d_{\cal M}(y,a^\star) - d_{\cal M}(y,a^\star)\\
               &\ge d_{\cal M}(z,a^\star)- d_{\cal M}(y,a^\star)\\
               &\ge 2r - r >0,
    \end{align*}
    which implies that $y \in V$ for each $y\in \mathbb{B}_{r}(\cal A)$. Therefore, for each $U\in \cal N_{\cal A}$, there exists $r>0$ such that $\mathbb{B}_{r}(\cal A) \subset \interior U \subset U$.
\end{proof}

\subsection{Geodesic Convexity}

Due to space constraints, we use some notions from Riemannian geometry without defining them here. In particular, for definitions of a \emph{geodesic}, a \emph{minimizing geodesic}, and a \emph{geodesic ball} on a Riemannian manifold $(\M, g_{\M})$, we refer the reader to~\cite[Ch.~6]{Lee_Riemannian}.

\begin{definition}[Geodesic convexity]
    \label{def:geodesicConvexity}
    Let $(\M, g_{\M})$ be a Riemannian manifold. A set $U \subset \M$ is geodesically convex if for each $p, q \in U$, there is a unique minimizing geodesic segment from $p$ to $q$, and the image of this geodesic segment lies entirely in $U$. 
\end{definition}

\begin{lemma}
\label{lemma:disconnectedManifold}
    Let $(\M, g_{\M})$ be a Riemannian manifold. Then, $\M$ has countably many connected components $\{\M_i\}_{i\in I}$ for some nonempty index set $I\subset \mathbb{N}$, and for each $i\in I$, $\M_i$ is a Riemannian manifold together with the metric $g_{\M}$ restricted to $\M_i$. 
\end{lemma}
\begin{proof}
    Using~\cite[Prop.~1.11]{Lee}, we establish that $\M$ has countably many connected components and that $\M$ is a topological manifold. Pick any $i\in I$. The restriction of the smooth structure and the Riemannian metric on $\M$ to $\M_i$ renders $\M_i$ a Riemannian manifold. 
\end{proof}

\begin{lemma}
    \label{lemma:geodesicBall}
    Consider a Riemannian manifold $(\M, g_{\M})$ with a distance metric $d_{\M}$ that satisfies Assumption~\ref{ass:disconnectedManifold}. Consider also a nonempty, open set $U \subset \M$. For each $p \in U$, there exists a geodesically convex set $X \subset U$ that is contained in the connected component of $\M$ containing $p$ such that $p\in \interior{X}$. 
\end{lemma}
\begin{proof}
    Using Lemma~\ref{lemma:disconnectedManifold}, let $\M = \cup_{i\in I} \M_i$, where $ I$ is a countable index set and each $\M_i$, which is a connected component of $\M$, is a connected Riemannian manifold. 
    Pick $p\in U$, and suppose that $p \in \M_i$ for some $i\in I$. By Lemma~\ref{lemma:geodesicBall} and Assumption~\ref{ass:disconnectedManifold}, $\M_i$ is a Riemannian manifold with distance metric $d_{\M_i}$, which is a restriction of $d_{\M}$ to $\M_i$. 

    Let $U' \coloneqq U \cap \M_i$ be a connected open set in $\M_i$ containing $p$. As $d_{\M_i}$ is compatible with the topology of $\M_i$, there exists $r' > 0$ small enough such that
    \begin{align}\label{eq:geodesicBallInclusion}
        \mathbb{B}_{r'}(p) \subset U',
    \end{align}
    where $\mathbb{B}_{r}(p) \coloneqq \{x\in \M_i : d_{\M_i}(x, p)\leq r\}$ denotes the closed metric ball in $\M_i$ of radius $r$ at $p$. 

    Next, using~\cite[Thm.~6.17]{Lee_Riemannian}, there exists $\epsilon > 0$ such that every open geodesic ball in $\M_i$, centered at $p$ of radius $r \leq \epsilon$, denoted by $B_{r}(p) \subset \M$, is geodesically convex. Additionally, $B_{r}(p)$ is also geodesically convex in $\M$.
    By~\cite[Cor.~6.13]{Lee_Riemannian} and connectedness of $\M_i$, the open geodesic ball $B_{r}(p)$ is equal to $\mathbb{B}^\circ_{r}(p) \coloneqq \interior{\mathbb{B}_r(p)}$. Using~\cite[Thm.~6.17]{Lee_Riemannian}, there exists $r'' > 0$ such that $B_{r''}(p)$ is geodesically convex in $\M_i$, and therefore, in $\M$. Let $r \coloneqq \min\{r', r''\}$ and $X \coloneqq B_{r}(p)$. Using~\eqref{eq:geodesicBallInclusion}, 
    \(
        X \subset U' \subset U.  
    \)
    The proof is complete by noting that $X$ is geodesically convex and open so that $p\in \interior{X}$.
\end{proof}

\begin{lemma}
    \label{lemma:boundsOnRiemannianDistance}
    Consider a Riemannian manifold $(\M, g_{\M})$ with a distance metric $d_{\M}$ that satisfies Assumption~\ref{ass:disconnectedManifold}. For each $x_0 \in \M$ and each coordinate chart $(U, \varphi)$ on $\M$ at $x_0$, there exist a geodesically convex set $X \subset U$ satisfying $x_0 \in \interior{X}$, and constants $c, C > 0$ with $c \leq C$ such that, for each $x_1, x_2 \in X$, 
    \begin{align}
        \label{eq:distanceInequalities}
        c |\varphi(x_1) - \varphi(x_2)| \leq d_{\M}(x_1, x_2) \leq C |\varphi(x_1) - \varphi(x_2)|.
    \end{align}
\end{lemma}

\begin{proof}
    Let $n = \dim{\M}$. Pick any $x_0 \in \M$ and any coordinate chart $(U, \varphi)$ on $\M$ at $x_0$. Let $\M_i$ be the connected component of $\M$ that contains $x_0$, and let $d_{\M_i}$ denote the restriction of $d_{\M}$ to $\M_i$. Let $V \coloneqq U \cap \M_i$, and note that $(V, \varphi)$ is a coordinate chart at $x_0$ with $V \subset \M_i$. Following Lemma~\ref{lemma:coordballs}, pick also a regular coordinate ball $U'$ on $\M_i$ such that $x_0 \in U'$ and $U' \subset \overline{U'} \subset V$. 

    To obtain the upper bound in~\eqref{eq:distanceInequalities}, pick a geodesically convex set $X \subset U'$ according to Lemma~\ref{lemma:geodesicBall}. Pick any $x_1, x_2 \in X $, and denote their coordinate representations by $y_1 \coloneqq \varphi(x_1)$ and $y_2\coloneqq \varphi(x_2)$. Since $U'$ is a regular coordinate ball on $\M$, $\varphi(U') \subset \R{n}$ is convex. Then, we define the straight-line segment $\tilde{\gamma} : [0,1] \to \varphi(U')$ connecting $y_1$ and $y_2$ as $\tilde{\gamma}(s) \coloneqq s y_1 + (1-s)y_2$ for each $s\in [0,1]$. Under the $C^1$-diffeomorphism $\varphi$, we obtain a $C^1$-function $\gamma \coloneqq \varphi^{-1} \circ \tilde{\gamma}$ on $U'$ connecting $x_1$ and $x_2$. We denote the length of the function $\gamma$ under the metric $g_{\M}$ by
    \begin{align}
        \label{eq:length_gamma}
        L_{g_{\M}}(\gamma) \coloneqq \int_{0}^{1} |{\gamma}'(s)|_{g_{\M}} ds. 
    \end{align}
    Recall that, as $x_1, x_2 \in X$ and $X$ is geodesically convex, there exists a minimizing $C^1$-function, parametrized over an interval with image in $X$, that connects $x_1$ and $x_2$. Then, as $\M_i$ is connected, the length of this minimizing function is the Riemannian distance $d_{\M_i}(x_1, x_2)$ between $x_1$ and $x_2$; see~\cite[p.~36]{Lee_Riemannian}. Consequently, we have that
    \begin{align}
    \label{eq:distanceAndLengths}
        d_{\M_i}(x_1, x_2) &\leq L_{g_{\M}}(\gamma).
    \end{align}
    Since $\varphi$ is a $C^1$-diffeomorphism from $V$ onto its image, we use~\cite[Prop.~13.9]{Lee} to induce a Riemannian metric on $\varphi(V)\subset \R{n}$, denoted by $\bar{g}$, such that $(V, g_{\M})$ is isometric to $(\varphi(V), \bar{g})$; see~\cite[Ch.~2]{Lee_Riemannian} for more details about Riemannian isometries. Using~\cite[Ex.~13.24]{Lee} and~\eqref{eq:length_gamma}, it follows that $L_{g_{\M}}(\gamma) = L_{\bar{g}}(\tilde{\gamma})$. Consequently,~\eqref{eq:distanceAndLengths} yields
    \begin{align}
        d_{\M_i}(x_1, x_2) &\leq L_{\bar{g}}(\tilde{\gamma}) \nonumber
        \\
        &\leq \int_{0}^{1} |\tilde{\gamma}'(s)|_{\bar{g}} ds. \label{eq:intermediateStep}
    \end{align}
    Then, applying~\cite[Lemma~13.28]{Lee} with the compact set $K$ therein replaced by $\varphi(\overline{U'})$, there exist constants $c, C > 0$ with $c \leq C$ such that, for each $y\in \varphi(\overline{U'})$ and each $v \in \T{y}{\R{n}}$,
    \begin{align}
    \label{eq:EuclideanMetricComparison}
        c |v| \leq |v|_{\bar{g}} \leq C |v|,
    \end{align}
    where, recall that, $|\cdot|$ denotes the standard Euclidean norm. Then, using~\eqref{eq:EuclideanMetricComparison} with~\eqref{eq:intermediateStep} yields
     \begin{align*}
        d_{\M_i}(x_1, x_2) &\leq \int_{0}^{1} |\tilde{\gamma}'(s)|_{\bar{g}} ds
        \\
        &\leq C \int_{0}^{1} |\tilde{\gamma}'(s)| ds
        \\
        &\leq C | \varphi(x_1) - \varphi(x_2)|,
    \end{align*}
    where the last inequality is using the fact that $\tilde{\gamma}$ is a straight-line segment connecting $y_1 = \varphi(x_1)$ to $y_2 = \varphi(x_2)$, so its length equals $|y_1 - y_2|$. 

    Finally, to prove the lower bound in~\eqref{eq:distanceInequalities}, pick any $x_1, x_2\in X$ and denote by $\gamma: [0,1] \to X$ the minimizing geodesic connecting $x_1$ to $x_2$ with its range in $X$. Denote by $\tilde{\gamma} \coloneqq \varphi \circ \gamma : [0, 1] \to \varphi(X)$ its coordinate representation that connects $y_1 \coloneqq \varphi(x_1)$ to $y_2 \coloneqq \varphi(x_2)$. Then, 
    \begin{align*}
        d_{\M_i}(x_1, x_2) &= L_{g_{\M}}(\gamma)
        \\
        & = L_{\bar{g}}(\tilde{\gamma})
        \\
        & = \int_{0}^{1} |\tilde{\gamma}'(s)|_{\bar{g}} ds
        \\
        & \geq c \int_{0}^{1} |\tilde{\gamma}'(s)| ds
        \\
        &\geq c |\varphi(x_1) - \varphi(x_2)|.
    \end{align*}
    As $d_{\M_i}(x_1, x_2) = d_{\M}(x_1, x_2)$ from Assumption~\ref{ass:disconnectedManifold}, the proof is complete. 
    
\end{proof}

\section{Equivalent Characterizations}

In this appendix, we present useful equivalent coordinate-based characterizations of the coordinate independent notions of local precompactness and local abolute continuity. 

\subsection{Local Precompactness}
\label{appendix:localPrecompactness}

We provide a useful characterization of local precompactness of the map $F : \M \rightrightarrows \T{}{\M}$ in terms of the coordinate charts of $\M$. Before doing so, we introduce the following auxiliary results. 

\begin{lemma}[Regular Coordinate Balls]
\label{lemma:coordballs}
    Let $\M$ be a topological manifold. Then, for any point $x \in M$ and any open neighborhood $U$ of $x$, there exists a regular coordinate ball $B$ such that $x \in B \subset \overline{B} \subset U$. 
\end{lemma}
\begin{proof}
By \cite[Prop 4.60]{lee2010introduction}, the collection of all regular coordinate balls forms a countable basis for the topology of $\mathcal{M}$. By definition of a basis, for any $x \in \mathcal{M}$ and open set $U$ containing $x$, there exists a regular coordinate ball $R$ such that $x \in R \subset U$.

By definition of a regular coordinate ball~\cite[p.~15]{Lee}, there exists a coordinate chart $(R', \varphi)$ with $\overline{R} \subset R'$ and positive real numbers $\rho < \rho'$ such that $\varphi(R) = \rho\cdot\mathrm{int}\mathbb{B}$, $\varphi(\overline{R}) = \rho\mathbb{B}$, and $\varphi(R') = \rho'\cdot\mathrm{int}\mathbb{B}$. Since $x \in R$, we have $\varphi(x) \in \rho\cdot\mathrm{int}\mathbb{B}$. Let $d = |\varphi(x)|$, choose any $r$ with $d < r < \rho$ and define $B \coloneqq \varphi^{-1}(r\cdot\mathrm{int}\mathbb{B})$.

Given that $r < \rho$, we have $\overline{B} \subset R \subset U$, and since $d < r$, we have $x \in B$. Also, $\overline{B} \subset R \subset R'$ means $B$ is a regular coordinate ball with associated coordinate chart $(R', \varphi)$. Therefore, there exists a regular coordinate ball $B$ with $x \in B \subset \overline{B} \subset U$.
\end{proof}

The following lemma formalizes the discussion presented in \cite[Chapter 8]{gallierDifferentialGeometry2020} regarding the local trivialization of tangent bundles. The map $\diffFunc{\varphi}$ replaces the map $\theta^{-1}$ therein.
\begin{lemma}[Local Trivialization of Tangent Bundle]
    \label{lemma:trivialization}
    Let $\M$ be a $C^1$-manifold. For any coordinate chart $(U, \varphi)$ where $\varphi: U \to \mathbb{R}^n$, the map $\Phi: \T{}{U} \to U \times \mathbb{R}^n$ defined by $\Phi(v) = (x, \diffFunc{\varphi_x}(v))$ for $v \in \T{x}{\M}$ is a homeomorphism.
\end{lemma}

Using the above results, local precompactness of a set-valued map can be equivalently characterized as follows.

\begin{lemma}
\label{lemma:LB_equivalence}
Let $\M$ be an $n$-dimensional $C^1$ manifold. A set-valued map $F: \M \rightrightarrows \T{}{\M}$ is locally precompact if and only if for each point $x \in \M$, there exists a coordinate chart $(U, \varphi)$ at $x$ and a constant $K > 0$ such that $( \diffFunc{\varphi_y} \circ F)(y)\subset K\mathbb{B}$ for all $y \in U$, where $\mathbb{B} \coloneqq \{v\in \R{n} : |v|\leq 1\}$.
\end{lemma}

\begin{proof}
    ($\Rightarrow $) Suppose $F$ is locally precompact, and let $x \in \M$ be arbitrary. Then there exists a neighborhood $U_1$ of $x$ and $K\subset T\M$ compact such that $F(U_1)\subset K$. Let $(U_2, \varphi)$ be a coordinate chart containing $x$ and define $U \coloneqq U_1 \cap U_2$.
    
    By Lemma \ref{lemma:coordballs}, there exists a regular coordinate ball $B$ with $x \in B\subset \overline{B} \subset U$. Define
    \begin{align*}
        F(\overline{B}) &\coloneqq\{(x,f)~:~x\in \overline{B},~f\in F(x)\}.
    \end{align*}
    Now, note that for any $(x,f)\in F(\overline{B})$, it follows that $x \in \overline{B}$ which implies that $F(\overline{B})\subset \pi^{-1}(\overline{B})$, where $\pi:\T{}{\M} \to \M$ is the canonical projection. Additionally, for any $(x,f)\in F(\overline{B})$ it follows that $(x,f)\in F(U)$, since $\overline{B}\subset U$. Therefore, we have that
    \begin{align*}
        F(\overline{B})\subset F(U) \cap \pi^{-1}(\overline{B})\subset K\cap \pi^{-1}(\overline{B}).
    \end{align*}
    Since $\overline{B}$ is closed, as it is compact, and $\pi$ is continuous, it follows that $\pi^{-1}(\overline{B})$ is closed. Then, $F(\overline{B})$ is contained in a compact set since $\pi^{-1}(\overline{B})\cap K\eqqcolon \tilde{K}$ is the intersection of a compact set with a closed set. 
    
    Now, by Lemma \ref{lemma:trivialization}, there exists a local trivialization given by $\Phi: \T{}{U} \to U \times \mathbb{R}^n$ given by $\Phi((x,f)) = (x, d\varphi_x(f))$ for $(x,f) \in \T{}{\M}$ with $x \in U$ and $f \in \T{x}{\M}$. Since $\tilde{K}$ is compact and $\Phi$ is a homeomorphism, it follows that $\Phi(\tilde{K})$ is compact. 
    
    Let $\pi_2: U \times \mathbb{R}^n \to \mathbb{R}^n$ be the canonical projection onto the second factor. Since $\pi_2$ is continuous, and $\Phi(\tilde{K})$ is compact, it follows that $\pi_2(\Phi(\tilde{K}))$ is a compact subset of $\mathbb{R}^n$. Therefore, there exists $K' > 0$ such that:
    \begin{align}
        \pi_2\left(\Phi(\tilde{K})\right)\subset K'\mathbb{B}\subset \mathbb{R}^n.
    \end{align}
    Given that $F(B)\subset F(\overline{B}) \subset \tilde{K}$, we obtain that $\pi_2\left(\Phi(F(B))\right)\subset \pi_2(\Phi(\tilde{K}))\subset K'\mathbb{B}$. This implies that for all $x \in B$ it follows that $(d\varphi_x \circ F)(B)\subset K'\mathbb{B}$. 
    
    Since $x$ was chosen arbitrarily, the previous result holds for all $x\in \M$. In other words, for all $x\in \M$, there exists a coordinate chart $(B,\varphi|_B)$ at $x$ and a constant $K'>0$ such that $\left(\diffFunc{\varphi_x} \circ F\right)(B)\subset K'\mathbb{B}$ for all $x\in B$, i.e., that $F$ is locally coordinate bounded.
    
    \noindent($\Leftarrow$) Suppose there exists a coordinate chart $(U, \varphi)$ at $x$ and a constant $K > 0$ such that $(\diffFunc{\varphi_y} \circ F)(y)\subset K\mathbb{B}$ for all $y \in U$.
    
    Then, by Lemma \ref{lemma:coordballs}, there exists a regular coordinate ball $B$ with $x \in B\subset \overline{B} \subset U$. Consider the set
    \begin{align*}
    F(\overline{B}) = \{(x,v) \in \T{}{\M} : x \in \overline{B}, v \in F(x)\}.
    \end{align*}
    By Lemma \ref{lemma:trivialization}, there exists a local trivialization $\Phi: \T{}{U} \to U \times \mathbb{R}^n$ induced by the homeomorphism $(x,f)\mapsto (x,\diffFunc{\varphi_x}(f))$. Using this map, and the fact that $(\diffFunc{\varphi_x}\circ F)(\overline{ B})\subset (\diffFunc{\varphi_x}\circ F)(U)\subset K\mathbb{B}$, we obtain that
    \begin{align*}
        \Phi\left(F(\overline{B})\right) &= \left\{(x, v) \in \overline{B} \times \mathbb{R}^n : \begin{array}{c}
            x \in \overline{B}, v=\diffFunc{\varphi_x}(f),
            \\
            f\in F(x)
        \end{array}
        \right\} \\
        &\subset \{(x, v) \in \overline{B} \times \mathbb{R}^n : x \in \overline{B},~|v|\le K\}\\
                  &= \overline{B} \times K\mathbb{B}.
    \end{align*}
    The set $\overline{B} \times K\mathbb{B}$ is compact in $U \times \mathbb{R}^n$ since $\overline{B}$ and $K\mathbb{B}$ are compact, and the cartesian product of compact sets is compact. Thus, since $\Phi$ is a homeomorphism, it follows that $\Phi^{-1}(\overline{B} \times K\mathbb{B})\eqqcolon\tilde{K}$ is compact.
    
    Therefore, there exists an neighborhood $B$ of $x$ and a compact set $\tilde{K}\subset T\M$ such that
    \begin{align}
        F(B)\subset F(\overline{B})\subset \tilde{K}.
    \end{align}
    
    Since $x\in\M$ was chosen arbitrarily, the result holds for all $x\in\M$. Thus, we obtain that $F$ is locally precompact.
    \end{proof}

\subsection{Local Absolute Continuity}

\begin{proposition}\label{prop:CkAtAnyChart}
    Let $\mathcal{M}$ be a $C^k$-manifold and $f\in C^k(\M)$. Then, for any $x\in \M$ and any chart $(V,\psi)$ at $x$ it follows that $f \circ \psi^{-1}: \psi(V) \rightarrow \mathbb{R}$ is a $C^k$-function.
    \end{proposition}
    
    \begin{proof} Let $f\in C^k(\M)$. Consider $x\in \M$ and a chart $(V, \psi)$ at $x$.\smallbreak
                
        For each $p \in \psi(V)$, define $y = \psi^{-1}(p) \in V$. Since $y\in V$, it follows that $(V, \psi)$ is a chart at $y$. Additionally, by definition of a $C^k$-function, since $f\in C^k(\M)$ there exists a chart $(U_y, \varphi_y)$ at $y$ such that $f \circ \varphi_y^{-1}: \varphi_y(U_y) \rightarrow \mathbb{R}$ is a $C^k$-function.\smallbreak
            
        Define $W_y\coloneqq U_y\cap V$.  Since $\psi$ is a homeomorphism, it is in particular an open map \cite[Ex.~16.4]{Munkres2000}. Therefore, $\psi(W_y)$ is an open neighborhood of $p = \psi(y)$. Additionally, 
         since $f \circ \varphi_y^{-1}$ is a $C^k$-function on $\varphi_y(U_y)$ and both $(U_y, \varphi_y)$ and $(V, \psi)$ are charts at $y$, it follows that $f \circ \psi^{-1}: \psi(W_y) \rightarrow \mathbb{R}$ is also a $C^k$-function.\smallbreak
     
        Thus, for all $p \in \psi(V)$, there exists an open neighborhood $\psi(W_y)$ of $p$ where $f \circ \psi^{-1}$ is a $C^k$-function. This implies that $f \circ \psi^{-1}: \psi(V) \rightarrow \mathbb{R}$ is $C^k$ on $\psi(V)$.
\end{proof}

Let $\mathcal{J}$ denote the set of intervals of the form $(a,d)$, $[b,d)$, $(a,c]$, and $[b,c]$, where $a, d \in \R{} \cup \{\pm \infty\}$, $b, c\in \R{}$, and $a \leq b \leq c \leq d$. 

\begin{lemma}[Local Absolute Continuity]
\label{lemma:locAbsCont_coord}
    Let $\mathcal{M}$ be an $n$-dimensional $C^k$-manifold. A function $f:\mathcal{J}\ni I\to \mathcal{M}$ is locally absolutely continuous if and only if for every $x \in \rge{f} \subset \mathcal{M}$ there exists a chart $(U,\varphi)$ at $x$, such that the restriction of $\varphi \circ f$ to each connected component of $f^{-1}(U)$ is locally absolutely continuous.
\end{lemma}

\begin{proof}
\noindent($\Rightarrow$) Suppose $f:\mathcal{J}\ni I \to \mathcal{M}$ is locally absolutely continuous.

Let $x \in \rge{f}$ and let $(U,\varphi)$ be any chart at $x$ with component functions $\varphi_i: U \to \mathbb{R}$, $i \in\{ 1,\ldots,n\}$. By Definition~\ref{def:localAbsoluteContinuity}, each composition $\varphi_i \circ f$ is locally absolutely continuous on $[a,b]$, which implies that $f$ is continuous.\smallbreak

Since $f$ is continuous and $U$ is open, $f^{-1}(U)$ is open relative to $I$, so $f^{-1}(U) = I \cap V$ for some open set $V$ in $\mathbb{R}$. Any open subset of $\mathbb{R}$ is a countable union of disjoint open intervals \cite[Proposition 0.21]{folland_real_1999}, so
    \begin{align}\label{eq:proof:lac:aux0}
        f^{-1}(U) = I \cap \left(\bigcup_{k=1}^\infty (a_k,b_k)\right)
    \end{align}
    for some collection of disjoint open intervals $\{(a_k,b_k)\}_{k=1}^\infty$.\smallbreak

Let $J\subset I$ be any connected component of $f^{-1}(U)$. Via \eqref{eq:proof:lac:aux0}, $J$ is an interval that is open relative to $I$. For each component function $\varphi_i$ of the chart map, since $\varphi_i\in C^k(U)$, it follows that the composition $\varphi_i \circ f$ is locally absolutely continuous on $I$ by assumption. When restricted to $J$, this local absolute continuity is preserved because any compact subinterval of $J$ is also a compact subinterval of $I$. Therefore, each component $\varphi_i \circ f|_J$ is locally absolutely continuous, which means the vector-valued function $\varphi \circ f|_J$ is locally absolutely continuous as well.\medbreak

\noindent ($\Leftarrow$) Suppose that for every $y \in \rge{f}$ there exists a chart $(U,\varphi)$ at $y$, such that the restriction of $\varphi \circ f$ to each connected component of $f^{-1}(U)$ is locally absolutely continuous.\smallbreak

\noindent\emph{Step 1. Continuity of $f$.}\\
For any $t_0 \in I$, let $y_0 = f(t_0)$ and choose a chart $(U,\varphi)$ at $y_0$ satisfying the assumption. Since $\varphi \circ f$ is locally absolutely continuous on the connected component of $f^{-1}(U)$ containing $t_0$, it is continuous at $t_0$. As $\varphi$ is a homeomorphism onto its image, $f$ must be continuous at $t_0$. Since $t_0$ was arbitrary, $f$ is continuous on $I$.\smallbreak

\noindent\emph{Step 2: Open cover of a compact subinterval.}\\
Let $\psi \in C^k(\mathcal{M})$, and $K\coloneqq[a,b]\subset I$ be an arbitrary compact subinterval. Since $f$ is continuous, $f(K)$ is compact in $\mathcal{M}$. By assumption, for each $y \in \rge{f}$, there exists a chart $(U_y, \varphi_y)$ at $y$ such that the restriction of $\varphi \circ f$ to each connected component of $f^{-1}(U_y)$ is locally absolutely continuous. Additionally, $\bigcup_{y \in f(K)} U_y$ forms an open cover of $f(K)$. By the compactness of $f(K)$, there exists a finite subcover of $f(K)$, $\bigcup_{j=1}^m U_{y_j}$. It follows that
    \begin{align}\label{eq:proof:lac:aux1}
        \bigcup_{j=1}^m f^{-1}(U_{y_j})\cap K
    \end{align}
forms an open cover of $K$ in the subspace topology of $K\subset I$. Indeed, $f^{-1}(U_{y_j}) \cap K$ is open relative to $K$ because $U_{y_j}$ is open in $\mathcal{M}$ and $f$ is continuous. Via \cite[Proposition 0.21]{folland_real_1999}, any open subset of $\mathbb{R}$ is a countable union of disjoint open intervals. Thus, each $f^{-1}(U_{y_j})$ can be written as:
    $
        f^{-1}(U_{y_j}) = \bigcup_{k=1}^\infty \left(a_k^{(j)},b_k^{(j)}\right),
    $
    for some collection of open intervals $\left\{(a_k^{(j)},b_k^{(j)})\right\}_{k=1}^\infty$.
    Therefore, the intersection with $K$ gives:
    \begin{align}\label{eq:proof:lac:aux2}
    f^{-1}(U_{y_j}) \cap K &= \left(\bigcup_{k=1}^{\infty} \left(a_k^{(j)},b_k^{(j)}\right)\right) \cap K\notag\\
    &= \bigcup_{k=1}^{\infty} \left(\left(a_k^{(j)},b_k^{(j)}\right) \cap K\right).
    \end{align}
    
    While this expression involves an infinite union, only finitely many terms can be non-empty. To see why, suppose infinitely many intervals have non-empty intersection with $K$. For each such interval $\left(a_k^{(j)},b_k^{(j)}\right) \cap K \neq \emptyset$, choose a point $t_k \in \left(a_k^{(j)},b_k^{(j)}\right) \cap K$. Since $K = [a,b]$ is compact, the sequence $\{t_k\}_{k=1}^\infty$ has a convergent subsequence $\{t_{k_i}\}_{i=1}^\infty$ with limit $t^* \in K$. Given any $\varepsilon > 0$, there exists $i_0 \in \mathbb{N}$ such that for all $i \geq i_0$, $t_{k_i}\in t^* + \varepsilon\mathbb{B}$. This means infinitely many terms of the sequence lie within the $\varepsilon$-neighborhood of $t^*$.  However, since the points $t_{k_i}$ come from disjoint intervals, and any $\varepsilon$-neighborhood can only intersect finitely many disjoint intervals, we reach a contradiction. Therefore, we can write:
        \begin{align}\label{eq:proof:lac:axu3}
            f^{-1}(U_{y_j}) \cap K = \bigcup_{k=1}^{n_j} V_k^{(j)},
        \end{align}
        for some $n_j>0$ and where $V_k^{(j)}\coloneqq \left(a_k^{(j)},b_k^{(j)}\right)\cap K$. Using \eqref{eq:proof:lac:aux1} and \eqref{eq:proof:lac:axu3} we obtain that the set
        \begin{align*}
            \mathcal{U}\coloneqq \bigcup_{j=1}^m\bigcup_{k=1}^{n_j}V_k^{(j)}
        \end{align*}
      is an open cover of $K=[a,b]\subset I$.\medbreak
      \noindent\emph{Step 3. Construction of a closed partition for the compact subinterval.}\\
        Given that $K\subset \mathbb{R}$ with the absolute value $|\cdot|$ constitutes a compact metric space, it follows via \cite[Lemma 7.18]{lee2010introduction} that the cover $\mathcal{U}$ has a Lebesgue number. Namely, for the cover $\mathcal{U}$ of $K$ there exists $\ell>0$, called the Lebesgue number, such that for every subset $K'\subset K$ with $\mathrm{diam}(K')\coloneqq\sup\{|k_1-k_2|:k_1,k_2\in K'\} < \ell$, there exists an element of the cover $V_{k}^j\in\mathcal{U}$ satisfying $V_k^{(j)}\supset K'$.

        Using this fact, consider the following partition of $K=[c,d]$:
        \begin{align*}
            K = \bigcup_{i=1}^{N} [t_{i-1}, t_{i}],
        \end{align*}
        where $\{t_i\}_{i=0}^N\subset K$ are chosen such that $|t_{i}-t_{i-1}|<\ell$ for all $i\in \{1,2,\cdots, N\}$, and where $\ell>0$ is the Lebesgue number associated with the cover $\mathcal{U}$. This partition can be obtained, for instance, by letting $N=\left\lceil  2|d-c|/\ell\right\rceil$, where $\lceil\cdot\rceil$ denotes the ceil function, $t_{0}=c$, $t_N=d$, and $t_{i} =  t_{i-1} + \ell/2$ for all $i\in\{1,2,\ldots, N-1\}$.

        Since $\mathrm{diam}([t_{i-1},t_{i}])<\ell$, by \cite[Lemma 7.18]{lee2010introduction}, for all $i\in\{1,2,\cdots,N\}$ there exists one $V_k^{(j)}\in \mathcal{U}$ such that $[t_{i-1},t_{i}]\subset V_{k}^j$. Using \eqref{eq:proof:lac:axu3}, this implies that for all $i\in \{1,2,\cdots,N\}$ there exists a $j\in \{1,\cdots,m\}$ such that
        \begin{align*}
            [t_{i-1},t_{i}]\subset f^{-1}(U_{y_j}).
        \end{align*}    

\noindent\emph{Step 4: Local absolute continuity of $f$ on each component of the partition.}\\
Since each $[t_{i-1},t_{i}]$ is connected and contained in some $f^{-1}(U_{y_{j}})$, it must lie within a single connected component of $f^{-1}(U_{y_{j}})$. Thus, by assumption, $\varphi_{y_j} \circ f$ is locally absolutely continuous on $[t_{i-1},t_{i}]$.  Thus,  $(\varphi_{y_j} \circ f)([t_{i-1},t_{i}])$ is compact. 

On the other hand, since $\psi \in C^k(\mathcal{M})$, it follows by Proposition \ref{prop:CkAtAnyChart} that the composition $\psi \circ \varphi_{y_j}^{-1}: \varphi_{y_j}(U_{y_j}) \to \mathbb{R}$ is a $C^k$-function. Thus, $\psi \circ \varphi_{y_j}^{-1}$ is Lipschitz on $(\varphi_{y_j} \circ f)([t_{i-1},t_{i}])$. Using the fact that composition of Lipschitz functions with absolutely continuous function is absolutely continuous, it follows that $\left(\psi \circ \varphi_{y_j}^{-1}\right)\circ \left(\varphi_{y_j} \circ f\right) = \psi \circ f$ is absolutely continuous on each $[t_i,t_{i+1}]$.\medbreak

\noindent\emph{Step 5: Absolute continuity on $K$.}\\
If $N=1$, i.e. $K$ is covered by a single interval $[t_{i-1},t_{i}]$, then $\psi \circ f$ is directly absolutely continuous on $K$.

For $N > 1$, let $\varepsilon>0$ be arbitrary. Then, by the local absolute continuity of $\psi\circ f$ on each $[t_{i-1}, t_i]$, there exists $\delta_i > 0$ such that for any countable collection of disjoint intervals $\{[t_{i-1}^{(k)}, t_i^{(k)}]\}_k \subset [t_{i-1}, t_i]$ with $\sum_k (t_i^{(k)} - t_{i-1}^{(k)}) \le \delta_i$, it follows that
\begin{align}\label{eq:proof:lac:auxLACInSubintervals}
    \sum_k |(\psi \circ f)(t_i^{(k)}) - (\psi \circ f)(t_{i-1}^{(k)})| \le \frac{\varepsilon}{N}.
\end{align}

Define $\delta \coloneqq \min\{\delta_1, \delta_2, \ldots, \delta_N\} > 0$, and consider an arbitrary countable collection of disjoint intervals $\{[a_p, b_p]\}_p \subset K$ with $ \sum_p (b_p - a_p) \le \delta$. We can partition this collection $\{[a_p, b_p]\}_p$ into another collection  $\{[c_l, d_l]\}_l$ where each interval lies within some $[t_{i-1}, t_i]$. Indeed, for any interval $[a_j, b_j]$ spanning multiple subintervals, split it at the partition points $t_i, t_{i+1}, \ldots, t_{k-1}$ contained within it, creating intervals $[a_j, t_i], [t_i, t_{i+1}], \ldots, [t_{k-1}, b_j]$. The resulting collection $\{[c_l, d_l]\}_l$  is still countable, as the parition points $\{t_i\}_{i=1}^N$ are finite. Also, each element of $\{[c_l, d_l]\}_l$ is contained within a single $[t_{i-1}, t_i]$. Since we only added division points we have that
    \begin{align*}
        \sum_l (d_l - c_l) = \sum_p (b_p - a_p) \le \delta.
    \end{align*}
For each $i \in \{1, 2, \ldots, N\}$, define $\mathcal{I}_i = \{l : [c_l, d_l] \subset [t_{i-1}, t_i]\}$. Then, for all $i\in \{1, 2, \ldots, N\}$ it follows that
    \begin{align}\label{eq:proof:lac:aux4}
        \sum_{l \in \mathcal{I}_i} (d_l - c_l) \le \sum_l (d_l - c_l) \le \delta=\min\{\delta_1,\delta_2,\ldots,\delta_N\}\le \delta_i.
    \end{align}
Since $\{[c_l, d_l]\}_{l\in \mathcal{I}_i}$ is a countable collection of disjoint intervals of $[t_{i-1}, t_i]$ satisfying \eqref{eq:proof:lac:aux4}, by using \eqref{eq:proof:lac:auxLACInSubintervals} we obtain that 
    \begin{align*}
        \sum_{l \in \mathcal{I}_i} |(\psi \circ f)(d_l) - (\psi \circ f)(c_l)| \le \frac{\varepsilon}{N}
    \end{align*}
for all $i\in \{1, 2, \ldots, N\}$. Therefore:
    \begin{align}\label{eq:proof:lac:aux5}
        &\sum_l |(\psi \circ f)(d_l) - (\psi \circ f)(c_l)| \notag
        \\&= \sum_{i=1}^N \sum_{l \in \mathcal{I}_i} |(\psi \circ f)(d_l) - (\psi \circ f)(c_l)| \le \sum_{i=1}^N \frac{\varepsilon}{N} = \varepsilon.
    \end{align}
Now, for the original collection of intervals $\{[a_p, b_p]\}_p$, we relate their total variation to that of the refined collection $\{[c_l, d_l]\}_l$. For each $p$, define the set $\mathcal{J}_p \coloneqq \{l : [c_l, d_l] \subset [a_p, b_p]\}$. Note that by construction $[a_p,b_p]=\bigcup_{l\in \mathcal{J}_p}[c_l,d_l]$, and the set $\mathcal{J}_p$ is finite since $\{[c_l,d_l]\}_{l}$ was derived from $\{[a_p,b_p]\}_{p}$ by splitting at finite parition points. Then, by using the triangle inequality we obtain:
\begin{align}\label{eq:proof:lac:aux6}
    |(\psi \circ f)(b_p) - (\psi \circ f)(a_p)|&
    \leq \sum_{l \in \mathcal{J}_p} |(\psi \circ f)(d_l) - (\psi \circ f)(c_l)|
\end{align}
for all $p$. Thus, using \eqref{eq:proof:lac:aux5} and \eqref{eq:proof:lac:aux6}, for any $\varepsilon>0$ and any arbitrary countable collection of disjoint intervals $\{[a_p, b_p]\}_p$ satisfying $\sum_p (b_p - a_p) \le \delta$, where $\delta$ is as defined below \eqref{eq:proof:lac:auxLACInSubintervals}, it follows that
\begin{align*}
    \sum_p |(\psi \circ f)(b_p) &- (\psi \circ f)(a_p)| 
    \\
    &\le\sum_p \sum_{l \in \mathcal{J}_p} |(\psi \circ f)(d_p) - (\psi \circ f)(c_l)| 
    \\
    &= \sum_l |(\psi \circ f)(d_l) - (\psi \circ f)(c_l)|
    \\
    &\leq \varepsilon,
\end{align*}
which proves that $\psi \circ f$ is absolutely continuous on $K$. Since $K = [a,b] \subset I$ was an arbitrary compact subinterval $\psi\circ f$ is absolutely continuous on $I$ by~\cite[Def.~A.20]{HybridFeedbackControl}.

Given that $\psi \in C^k(\mathcal{M})$ was arbitrary, $f$ is locally absolutely continuous by Definition~\ref{def:localAbsoluteContinuity}.
\end{proof}

\subsection{Locally Lipschitz Functions}

The following result is a generalization of~\cite[Prop.~B.4]{gissler2025irreducibilityconvergenceclassnonsmooth} where the manifolds $\M$ and $\N$ are assumed to be connected, and only one-sided implication is shown.

\begin{lemma}
\label{lemma:LocLipRiemannian}
\textit{(Riemannian Equivalence of Locally Lipschitz Functions)}
    Consider $C^1$-Riemannian manifolds $(\M, g_{\M})$ and $(\N, g_{\N})$ with distance functions $d_{\M}$ and $d_{\N}$, respectively, that satisfy Assumption~\ref{ass:disconnectedManifold}. Then, a function $f: \M \to \N$ is locally Lipschitz in the sense of Definition~\ref{def:Lipschitz} if and only if, for each $x_0\in \M$, there exist $k >0 $ and an open set $O\subset \M$ with $x_0\in O$ such that %
    {\color{mygreen}\relax{}\vspace{-0.1cm}
    \begin{equation}\label{eq:lipschitzRiemannian}
        d_{\N}(f(x_1), f(x_2)) \leq k d_{\M}(x_1, x_2) \qquad \forall x_1,x_2\in O. 
    \end{equation}}
    {} 
\end{lemma}

\begin{proof}
    $(\Rightarrow)$ First, we show that if \eqref{eq:lipschitzRiemannian} is satisfied, $f$ is continuous. Indeed, consider a sequence $\{x_i\}_{i=1}^\infty$ satisfying $x_i \to x \in \M$. Due to~\eqref{eq:lipschitzRiemannian}, $d_{\M}(x_i, x) \to 0$ and, therefore, $d_{\N}(f(x_i), f(x)) \to 0$. Then, as $d_{\N}$ is compatible with the topology of $\N$ by Assumption~\ref{ass:disconnectedManifold}, $f(x_i) \to f(x)$ in the topology of $\N$. Consequently, $f$ is continuous at $x$. As $x\in\M$ was arbitrary, $f$ is continuous. 

    Pick any $x_0\in \M$, and let $\M_i$ be the connected component of $\M$ that contains $x_0$. Let $d_{\M_i}$ denote the restriction of $d_{\M}$ to $\M_i$ according to Assumption~\ref{ass:disconnectedManifold}. Consider the open set $O\subset \M$ containing $x_0$ such that~\eqref{eq:lipschitzRiemannian} holds for some $k > 0$. Consequently,~\eqref{eq:lipschitzRiemannian} also holds for each $x_1, x_2 \in O' \coloneqq O \cap \M_i$. 

    Now, pick a coordinate chart $(U', \varphi)$ of $\M$ at $x_0$ such that $U'\subset O'$, and a coordinate chart $(W', \psi)$ of $\N$ at $f(x_0)$ such that, without loss of generality, $W'\subset f(U')$. Note that, as $\M_i$ is connected and $f$ is continuous, $f(\M_i) \subset \N$ is also connected. Consequently, $W' \subset f(\M_i)$ belongs to the connected component of $\N$ containing $f(x_0)$. Additionally, as $f$ is continuous and $W'$ is open, $f^{-1}(W') \subset U'$ is open. 
    
    Applying Lemma~\ref{lemma:boundsOnRiemannianDistance} to manifolds $(\M, g_{\M})$ and $(\N, g_{\N})$ for charts $(U', \varphi)$ and $(V', \psi)$, respectively, let $X \subset U'$ and $Y \subset W'$, respectively, denote the corresponding geodesically convex neighborhoods of $x_0$ and $f(x_0)$, respectively. As $x_0 \in \interior{X}$ and $x_0 \in f^{-1}(\interior{Y})$, it follows that that $X \cap f^{-1}(Y) \neq \varnothing$. Then, from the above application of Lemma~\ref{lemma:boundsOnRiemannianDistance}, there exist constants $C_{\M}, c_{\N}> 0$ such that, for each $x_1, x_2\in X \cap f^{-1}(Y)$,
    \begin{gather*}
        d_{\M}(x_1, x_2) \leq C_{\M} |\varphi(x_1) - \varphi(x_2)|,
        \\
        c_{\N} |\psi(f(x_1)) - \psi(f(x_2))| \leq d_{\N}(f(x_1), f(x_2)).
    \end{gather*}
    Using~\eqref{eq:lipschitzRiemannian} with the above inequalities yields
    \begin{align*}
        |\psi(f(x_1)) - \psi(f(x_2))| \leq \frac{k C_{\M}}{c_{\N}} |\varphi(x_1) - \varphi(x_2)|. 
    \end{align*}
    Writing $y_1 \coloneqq \varphi(x_1)$, $y_2 \coloneqq \varphi(x_2)$, and noting that $\varphi$ is invertible as it is a $C^1$-diffeomorhism onto its image, we have
    \begin{align}
    \label{eq:coordinateLipschitzBounds}
        |\psi \circ f \circ \varphi^{-1}(y_1) - \psi \circ f \circ \varphi^{-1}(y_2)| \leq \frac{k C_{\M}}{c_{\N}} |y_1 - y_2|. 
    \end{align}

    Consider an open neighborhood $U$ of $x_0$ such that $U \subset X\cap f^{-1}(Y)$. Such an open neighborhood exists as $\interior{X}$ and $\interior{Y}$ (see Lemma~\ref{lemma:geodesicBall}) are open sets and $f$ is continuous. Similarly, consider also an open neighborhood $W$ of $f(x_0)$ such that $W \subset Y$.  Then,~\eqref{eq:coordinateLipschitzBounds} holds for each $y_1, y_2 \in \varphi(U)$. Consequently, $\psi \circ f \circ \varphi^{-1}$ is locally Lipschitz. 

    $(\Leftarrow)$ Suppose that $f$ is locally Lipschitz in the sense of Definition~\ref{def:Lipschitz}. Continuity of $f$ can be established using continuity of its local representation in each pair of coordinate charts on $\M$ and $\N$. Pick any $x_0 \in \M$. Then, pick a coordinate chart $(U, \varphi)$ on $\M$ at $x_0$, and a coordinate chart $(W, \psi)$ on $\N$ at $f(x_0)$ according to Definition~\ref{def:Lipschitz} such that the coordinate representation $\tilde{f}\coloneqq \psi\circ f\circ \varphi^{-1}$ is locally Lipschitz; in particular, there exist an open set $\hat{U}\subset \varphi(U \cap \M_i)$ and $k'>0$ such that 
    \begin{align}
    \label{eq:localLipschitz}
        |\tilde{f}(y_1) - \tilde{f}(y_2)| \leq k' |y_1 - y_2| && \forall y_1, y_2 \in \hat{U}.
    \end{align}

    Since $\varphi$ is continuous and $\hat{U}$ is open, $\varphi^{-1}(\hat{U}) \subset U \cap \M_i$ is open. Let $O' \subset \varphi^{-1}(\hat{U})$ denote an open set containing $x_0$. Applying Lemma~\ref{lemma:boundsOnRiemannianDistance}, there exist a geodesically convex set $X\subset O'$ with $x_0 \in \interior{X}$, and constants $c_{\M}, C_{\N} > 0$ such that, for each $x_1, x_2 \in X$, 
    \begin{gather*}
        c_{\M}|\varphi(x_1) - \varphi(x_2)| \leq d_{\M}(x_1, x_2),
        \\
        d_{\N}(f(x_1), f(x_2)) \leq C_{\N}|\psi(f(x_1)) -  \psi(f(x_2))|.
    \end{gather*}

    Using the above inequalities with~\eqref{eq:localLipschitz} with the fact that $\varphi$ and $\psi$ are $C^1$-diffeomorphisms onto their image, and writing $y_1 = \varphi(x_1)$ and $y_2 = \varphi(x_2)$, we have that for each $x_1, x_2\in X$, 
    \begin{align}
    \label{eq:localLipschitz_Riemann}
        d_{\N}(f(x_1), f(x_2)) \leq \frac{k' C_{\N}}{c_{\M}} d_{\M}(x_1, x_2).
    \end{align}
    Since $x_0 \in \interior{X}$, there exists an open neighborhood $O\subset X$ of $x_0$ such that~\eqref{eq:localLipschitz_Riemann} holds for each $x_1, x_2\in O$. Setting $k\coloneqq k' C_{\N}/c_{\M}$ completes the proof.
\end{proof}

\subsection{Riemannian Equivalence of Proper Functions}

\begin{proposition}
\label{prop:Kinfty_lowerBoundOnV}
\textit{(Riemannian Equivalence of Proper Functions)}
    Let $(\M, g_{\M})$ be a complete Riemannian manifold that satisfies Assumption~\ref{ass:disconnectedManifold}. Given a nonempty, compact set $\A\subset \M$, a continuous function $V\in \PD{\A}$ is proper if and only if there exists $\alpha\in\mathcal{K}_{\infty}$ such that $\alpha(\distfromA{x}) \leq V(x)$~for~each~$x\in \M$.
\end{proposition}

\begin{proof}
    ($\Leftarrow$) Pick any $c \geq 0$. To prove that the $c$-sublevel set of $V$, denoted by $\sublevelSet{V}(c)$, is compact, it is sufficient due to Lemma~\ref{lemma:metric_inflation_compact} to show that $L_V(c)$ is closed and bounded.\footnote{A subset $S$ of a Riemannian manifold $(\M, g_{\M})$ is bounded if there exist $x\in \M$ and $r \geq 0$ such that $S\subset \mathbb{B}_{r}(x)$.} Since $V$ is continuous, the preimage of a closed set is closed. Therefore, $\sublevelSet{V}(c)$ is closed relative to $\M$. Next, the existence of $\alpha_1\in\mathcal{K}_\infty$ ensures that, for each $x\in L_V(c)$, $\distfromA{x}\leq \alpha^{-1}(V(x))\leq \alpha^{-1}(c)$. Therefore, $L_V(c)$ is bounded as $L_V(c) \subset \mathbb{B}_{\alpha^{-1}(c)}(\cal A)$. Following Lemma~\ref{lemma:metric_inflation_compact}, $\mathbb{B}_{\alpha^{-1}(c)}(\cal A)$ is compact. Consequently, for each $c \geq 0$, $L_V(c)$ is compact as it is a closed subset of a compact set~\cite[Prop.~A.45(e)]{Lee}. 
    
    ($\Rightarrow$) Define $\rho : \R{}_{\geq 0} \to \R{}_{\geq 0}$ as
    \begin{align}
    \label{eq:rho}
        \rho(r) \coloneqq \inf_{ \{x \in \M : \distfromA{x} = r \}} V(x)  \qquad \forall r\geq 0. 
    \end{align}
    Note that $\rho\in \PD{0}$ as $V\in \PD{\A}$.
    Next, we show that $\rho(r) \to \infty$ as $r \to \infty$. Suppose by contradiction that there exists a finite $k \geq 0$ such that $\liminf_{r\to \infty}\rho(r) = k$. Then, as
    \begin{align*}
        \liminf_{r\to \infty} \rho(r) &= \lim_{r\to \infty} \inf_{s > r} \rho(s)
        \\
        &= \lim_{r\to \infty} \inf_{s>r} \inf_{ \{ x\in \M : \distfromA{x} = s \} } V(x)
        \\
        &= \lim_{r\to \infty} \inf_{ \{ x\in \M : \distfromA{x} \geq  r \} } V(x)
        \\
        &= \liminf_{\distfromA{x}\to \infty} V(x),
    \end{align*}
    it follows that $\liminf_{\distfromA{x}\to\infty} V(x) = k$. Consequently, there exist $k' > k$ and a sequence $\{x_i\}_{i=1}^\infty$ of points $x_i \in \M$ with $\distfromA{x_i} \to\infty$ such that $V(x_i) \leq k'$ for each large enough $i$, causing $x_i \in L_V(k')$ for each large enough $i$. As $V$ is proper, $L_V(k')$ is compact. Then, using continuity of the map $x\mapsto \distfromA{x}$, there exists $c > 0$ such that $L_V(k') \subset \mathbb{B}_{c}(\cal A)$. This is a contradiction since $\distfromA{x_i} \to \infty$. Consequently, the above choice of $k\geq 0$ does not exist, causing $\rho(r) \to \infty$ as $r \to \infty$.

    Finally, following similar arguments to \cite[Lemma~1]{PaulACC2025}, $\rho$ is lower semicontinuous. Now, we construct a function $\gamma : \R{}_{\geq 0} \to \R{}_{\geq 0}$ as
    \begin{align}
    \label{eq:gamma_def}
        \gamma(r) \coloneqq \inf_{s \geq r} \rho(s) \qquad \forall r \geq 0. 
    \end{align}
    From positive definiteness and radial unboundedness of $\rho$, we have $\gamma\in \PD{0}$ and $\liminf_{r\to\infty} \gamma(r) = \infty$. Note also that $\gamma$ is nondecreasing, and for each $r\geq 0$, $\gamma(r)\leq \rho(r)$ by~\eqref{eq:gamma_def}. Furthermore, as $\rho$ is lower semicontinuous, $\gamma$ is lower semicontinuous as well.

    Now, we construct a continuous function $\sigma : \R{}_{\geq 0} \to \R{}_\geq $ that lower bounds $\gamma$. Indeed, let
    \begin{align*}
        \sigma(r) \coloneqq \inf_{s\geq 0} \left( |r - s| + \gamma(s) \right) \qquad \forall r \in \R{}_{\geq 0}. 
    \end{align*}
    Following~\cite[Prop.~1]{PaulACC2025}, $\sigma$ is continuous, positive definite, and satisfies $\sigma(r)\leq \gamma(r)$ for each $r\geq 0$. Additionally, as $\liminf_{r\to\infty}\gamma(r) = \infty$, it follows that $\liminf_{r\to\infty}\sigma(r) = \infty$. Then, using~\cite[Lemma~1]{Kellett2014}, there exists $\alpha\in\mathcal{K}_{\infty}$ such that
    \begin{align}
    \label{eq:comparison}
        \alpha(r) \leq \sigma(r) \leq \gamma(r) \leq \rho(r) \qquad \forall r \geq 0. 
    \end{align}
    Finally, from~\eqref{eq:rho}, $\rho(\distfromA{x})\leq V(x)$ for each $x\in \M$. Using this fact with~\eqref{eq:comparison} completes the proof. 
\end{proof}

\section{Example~\ref{ex:mobius-stability} Continued: \texorpdfstring{$\dot{V}$}{V} computation}
\label{app:mobius-Vdot}
    
    We present additional details from Example~\ref{ex:mobius-stability} to establish that $\dot{V}(x)\leq 0$ for each $x\in C$. Indeed, for each $x\in C$,
    \begin{align*}
        \dot{V}(x) &= \diffFunc{{V}_{y}}\left(\mathfrak{S}(y) + \vlift{y}{-\grad{V_q}{y_1} - b y_2}\right).
    \end{align*}
    Noting that $\diffFunc{V}_y$ is a linear map, we obtain, for each $x\in C$,
    \begin{align}
    \label{eq:mobius_Vdot}
        \dot{V}(x) &= \diffFunc{{V}_{y}}\left(\mathfrak{S}(y)\right) + \diffFunc{{V}_{y}} \left(\vlift{y}{-\grad{V_q}{y_1} - b y_2}\right).
    \end{align}
     Now, recalling~\eqref{eq:mobius-lyapunov}, the first term above yields
    \begin{align}
        \diffFunc{V_y}(\mathfrak{S}(y)) & = \diffFunc{(V_q)_{y_1}}(y_2) + \diffFunc{\left(\frac{1}{2}g_{\mobius}^{y_1}(y_2, y_2)\right)_{y}}(\mathfrak{S}(y)) \nonumber
        \\
        &= \diffFunc{(V_q)_{y_1}}(y_2) + \diffFunc{K}_{y}(\mathfrak{S}(y)) \notag
        \\
        &= g_{\mobius}^{y_1}(\grad{V_q}{y_1}, y_2) + \diffFunc{K}_{y}(\mathfrak{S}(y))\label{eq:mobius_PE_rateOfChange-1}
    \end{align}
    {where $K : \T{}{\mobius} \to \R{}_{\geq 0}$ denotes the kinetic energy function, defined as $K(y)\coloneqq g_{\mobius}^{y_1}(y_2, y_2)/2$ for each $y\in \T{}{\mobius}$. First, we simplify the second term in~\eqref{eq:mobius_PE_rateOfChange-1}. Recall from the definition of geodesic sprays that the integral curves of $\mathfrak{S}$ are geodesics on $\mobius$. Let $t\mapsto \psi(t)$ denote an integral curve such that $\psi(0) = y_1$ and $\dot{\psi}(0) = y_2$. Define $\eta(t) \coloneqq (\psi(t), \dot{\psi}(t))$ for each $t\in\dom{\psi}$. Note that $\eta(t)\in\T{}{\mobius}$ for each $t\in\dom{\psi}$. By definition of differential of a map, we have that
    \begin{align}
        \diffFunc{K}_{y}(\mathfrak{S}(y)) &= \left.\frac{d}{dt} K(\eta(t)) \right|_{t = 0} \notag
        \\
        &= \left. \frac{1}{2} \frac{d}{dt} \left(g_{\mobius}^{\psi(t)}(\dot{\psi}(t), \dot{\psi}(t))\right) \right|_{t = 0} \label{eq:Kdot}
    \end{align}
    Now, by definition of the Levi-Civita connection~\cite[Def.~15.8]{gallierDifferentialGeometry2020},
    \begin{align*}
        \left. \frac{d}{dt} g_{\mobius}^{\psi(t)}(\dot{\psi}(t), \dot{\psi}(t)) \right|_{t = 0} &= g_{\mobius}^{{\psi}(0)}(\nabla_{\dot{\psi}(0)} \dot{\psi}(0), \dot{\psi}(0))
        \\
        & \quad \qquad + g_{\mobius}^{{\psi}(0)}(\dot{\psi}(0), \nabla_{\dot{\psi}(0)} \dot{\psi}(0))
        \\
        &= 2 g_{\mobius}^{y_1}(\dot{\psi}(0), \nabla_{\dot{\psi}(0)}\dot{\psi}(0))
    \end{align*}
    where the last equality follows from symmetry of the metric $g_{\mobius}^{y_1}$. Then, as $\psi$ is a geodesic by construction, the definition of a geodesic causes $\nabla_{\dot{\psi}(0)}\dot{\psi}(0) = 0$. Therefore, we have that 
    \[
        \left. \frac{d}{dt} g_{\mobius}^{\psi(t)}(\dot{\psi}(t), \dot{\psi}(t)) \right|_{t = 0} = 0. 
    \]
    Consequently, from~\eqref{eq:Kdot}, $\diffFunc{K}_y(\mathfrak{S}(y)) = 0$. Substituting this equality in~\eqref{eq:mobius_PE_rateOfChange-1}, we have that, for each $x\in C$,
    \begin{align}
    \label{eq:mobius_PE_rateOfChange}
        \diffFunc{V_y}(\mathfrak{S}(y)) = g_{\mobius}^{y_1}(\grad{V_q}{y_1}, y_2).
    \end{align}}
    Next, letting $v \coloneqq -\grad{V_q}{y_1} - by_2$  for brevity, we simplify the second term in~\eqref{eq:mobius_Vdot} by constructing a smooth curve $\gamma: (-1, 1) \to \T{}{\mobius}$ as $\gamma(s)\coloneqq (y_1, y_2 + sv)$ for each $s\in \dom{\gamma}$, so that $d\gamma(0)/ds =\vlift{y}{v}$ by definition of $\operatorname{vlft}$. Then, 
    \begin{align}
        \diffFunc{V_y}(\vlift{y}{v}) &= \left.\frac{d}{ds} V(\gamma(s)) \right|_{s=0} \nonumber
        \\
        &= \left. \frac{d}{ds} V_q(y_1) \right|_{s = 0} + \left. \frac{1}{2}\frac{d}{ds} g_{\mobius}^{y_1}(y_2 + s v, y_2 + s v) \right|_{s=0}. \nonumber
    \end{align}
    As $V_q(y_1)$ is independent of $s$, the first term above equals zero. Then, using the chain rule on the second term yields
    \begin{align}
        \diffFunc{V_y}(\vlift{y}{v}) &= \left. \frac{1}{2}\frac{d}{ds} g_{\mobius}^{y_1}(y_2 + s v, y_2 + s v) \right|_{s=0} \nonumber
        \\
        &= \frac{1}{2} \left( g_{\mobius}^{y_1}(y_2, v) + g_{\mobius}^{y_1}(v, y_2) \right) \nonumber
        \\
        &= g_{\mobius}^{y_1}(y_2, v), \label{eq:mobius_KE_rateOfChange}
    \end{align}
    {where the last inequality above follows as the Riemannian inner product is symmetric. }
    Then, substituting~\eqref{eq:mobius_PE_rateOfChange} and~\eqref{eq:mobius_KE_rateOfChange} into~\eqref{eq:mobius_Vdot}, we have, for each $x\in C$,
    \begin{align}
        \dot{V}(x) &= g_{\mobius}^{y_1}(\grad{V_q}{y_1}, y_2) + g_{\mobius}^{y_1}(y_2, -\grad{V_q}{y_1} - by_2). \nonumber
    \end{align}
    {As the Riemannian inner product is additive, the second term above simplifies as $g_{\mobius}^{y_1}(y_2, -\grad{V_q}{y_1} - by_2)= g_{\mobius}^{y_1}(-\grad{V_q}{y_1}, y_2) + g_{\mobius}^{y_1}(y_2, - by_2)$. Then, using bilinearity of the Riemannian inner product, we obtain that}
    \begin{align}
        \dot{V}(x) &= g_{\mobius}^{y_1}(\grad{V_q}{y_1}, y_2) - g_{\mobius}^{y_1}(\grad{V_q}{y_1}, y_2) \nonumber
        \\ 
        & \phantom{=} + g_{\mobius}^{y_1}(y_2, - by_2) \nonumber
        \\
        &= -b g_{\mobius}^{y_1}(y_2, y_2) \leq 0.
    \end{align}

}

\end{document}